\documentclass[a4paper,reqno,10pt]{amsart}
\usepackage{amsfonts}
\usepackage{amsmath}
\usepackage{amsthm, amsmath, amssymb, enumerate}
\usepackage{amssymb, enumitem}
\usepackage{bbm}
\usepackage{amscd}
\usepackage[all]{xy}
\usepackage[left=2cm,right=2cm,bottom=3cm,top=3cm]{geometry}
\usepackage[hidelinks]{hyperref}
\usepackage{amsmath}
\usepackage{amssymb}
\usepackage{amsthm}
\usepackage[sort&compress,numbers]{natbib}

\newtheorem{theorem}{Theorem}[section]
\newtheorem*{theorem*}{Theorem}

\newtheorem*{claim*}{Claim}

\newtheorem{corollary}[theorem]{Corollary}

\newtheorem{example}[theorem]{Example}

\newtheorem{proposition}[theorem]{Proposition}

\theoremstyle{definition}

\newtheorem{remark}[theorem]{Remark}
\newtheorem{lemma}[theorem]{Lemma}
\newtheorem{definition}[theorem]{Definition}
\AtBeginDocument{   \def\MR#1{}
}

\input{tcilatex}

\begin{document}
\def\cprime{$'$}

\title[Model theory and Rokhlin dimension for compact quantum group actions]{Model theory and Rokhlin
dimension\\ for compact quantum group actions}
\author[Eusebio Gardella]{Eusebio Gardella}
\address{Eusebio Gardella\\
Westf\"{a}lische Wilhelms-Universit\"{a}t M\"{u}nster, Fachbereich
Mathematik, Einsteinstrasse 62, 48149 M\"{u}nster, Germany}
\email{gardella@uni-muenster.de}
\urladdr{https://wwwmath.uni-muenster.de/u/gardella/}
\author[Mehrdad Kalantar]{Mehrdad Kalantar}
\address{Mehrdad Kalantar, Department of Mathematics, University of Houston,
Philip Guthrie Hoffman Hall,
3551 Cullen Blvd., Houston, TX 77204, USA}
\email{kalantar@math.uh.edu}
\urladdr{https://www.math.uh.edu/~kalantar/}
\author{Martino Lupini}
\address{Martino Lupini\\
Mathematics Department\\
California Institute of Technology\\
1200 E. California Blvd\\
MC 253-37\\
Pasadena, CA 91125}
\email{lupini@caltech.edu}
\urladdr{http://www.lupini.org/}
\date{\today }
\thanks{This work was initiated during a visit of E.G. and M.L. to the
Mathematisches Forschungsinstitut Oberwolfach in August 2016, supported by
an Oberwolfach Leibnitz Fellowship of M.L. Parts of this work were carried
out during a visit of E.G. and M.K. to the California Institute of
Technology in January 2017, and during a visit of E.G. and M.L. to the
Centre de Recerca Matem{\`{a}}tica in March 2017 in occasion of the
Intensive Research Programme on Operator Algebras. The authors gratefully
acknowledge the hospitality and the financial support of all these
institutions. E.G. was partially funded by SFB 878 \emph{Groups, Geometry
and Actions}, and by a postdoctoral fellowship from the Humboldt Foundation.
M.K. was partially supported by the NSF Grant DMS-1700259.
M.L. was partially supported by the NSF Grant DMS-1600186.
This work is part of the project supported by the grant 
H2020-MSCA-RISE-2015-691246-QUANTUM DYNAMICS.}
\dedicatory{}

\begin{abstract}
We show that, for a given compact or discrete quantum group $G$, the class
of actions of $G$ on C*-algebras is first-order axiomatizable in the logic
for metric structures. As an application, we extend the notion of Rokhlin
property for $G$-C*-algebra, introduced by Barlak, Szab\'{o}, and Voigt in
the case when $G$ is second countable and coexact, to an arbitrary compact
quantum group $G$. All the the preservations and rigidity results for
Rokhlin actions of second countable coexact compact quantum groups obtained
by Barlak, Szab\'{o}, and Voigt are shown to hold in this general context.
As a further application, we extend the notion of equivariant order zero
dimension for equivariant *-homomorphisms, introduced in the classical
setting by the first and third authors, to actions of compact quantum
groups. This allows us to define the Rokhlin dimension of an action of a
compact quantum group on a C*-algebra, recovering the Rokhlin property as
Rokhlin dimension zero. We conclude by establishing a preservation result
for finite nuclear dimension and finite decomposition rank when passing to
fixed point algebras and crossed products by compact quantum group actions
with finite Rokhlin dimension.
\end{abstract}

\maketitle

\section{Introduction}

The Rokhlin property is a freeness condition for actions of groups on
C*-algebras. It has been intensively studied in recent years due to, among
other things, the strong implications it has on the structural properties of
fixed point algebras and crossed products. While generalizations to
noncompact groups, such as $\mathbb{Z}$ \cite{kishimoto_rohlin_1995}, or $%
\mathbb{R}$ \cite{hirshberg_rokhlin_2016}, have been considered, the most
common setting where the Rokhlin property has been studied is the one of
finite or, more generally, compact groups; see \cite%
{izumi_finite_2004,gardella_crossed_2014}. In this case, it has recently
been shown implicitly \cite{barlak_sequentially_2016}, and explicitly in 
\cite{gardella_equivariant_2016}, that the Rokhlin property is of
model-theoretic nature. Namely it corresponds to the *-homomorphism defining
the action being positively existential in the sense of first order logic
for metric structures.

Building on this work, the notion of Rokhlin property and the
above-mentioned preservation results have been generalized to the more
general setting of actions of coexact compact \emph{quantum }groups on
C*-algebras in \cite{barlak_spatial_2017}. This work has significantly
expanded the scope of the preservation results for Rokhlin actions, paving
the way of finding several new examples of classifiable C*-algebras arising
from compact quantum group actions. Furthermore, approach from \cite%
{barlak_spatial_2017} has also contributed to a simplification and better
understanding of the Rokhlin property, even in the classical setting.

While very fruitful, the Rokhlin property is also quite restrictive. In
order to circumvent this problem, the notion of Rokhlin dimension has been
recently introduced in the setting of actions of finite groups \cite%
{hirshberg_rokhlin_2015}, compact groups \cite{gardella_rokhlin_2014}, or $%
\mathbb{R}$ \cite{hirshberg_rokhlin_2016}. In this setting, the Rokhlin
property corresponds to having Rokhlin dimension equal to zero. Furthermore,
unlike Rokhlin actions, actions with finite Rohlin dimension are prevalent,
even in the case of C*-algebras that do not admit any Rokhlin action, such
as the Jiang--Su algebra.

The model-theoretic description of the Rokhlin property has been generalized
to Rokhlin dimension in \cite{gardella_equivariant_2016}, in terms of the
notion, introduced therein, of equivariant order zero dimension for an
equivariant *-homomorphism. Such a notion subsumes the notion of positively
existential equivariant *-homomorphism, which corresponds to having
equivariant order zero dimension equal to zero. This new perspective has
been used in \cite{gardella_equivariant_2016} to recover and extend various
preservation results for fixed point algebras and crossed products of
actions with finite Rokhlin dimension.

The goal of this paper is to show how these notions and results naturally
extend to the more general setting of actions of compact quantum groups. To
this purpose, we begin by showing that, for a fixed discrete or compact
quantum group $G$, the class of actions of $G$ on C*-algebras ($G$%
-C*-algebras) is first-order axiomatizable in a suitable language in the
logic for metric structures. This provides a notion of positively
existential $G$-equivariant *-homomorphism between $G$-C*-algebras, as well
as a notion of ultraproducts and reduced products for $G$-C*-algebras
consistent with the one considered in \cite{barlak_spatial_2017}. This
perspective is then used to show how to remove all coexactness and
separability assumptions in the preservation and rigidity results from \cite%
{barlak_spatial_2017}.

We then consider the natural generalization to the notion of equivariant
order zero dimension for equivariant *-homomorphisms, and define Rokhlin
dimension for compact quantum group actions in terms of such a notion. We
use the perspective of order zero dimension to establish many relevant facts
about actions with finite Rokhlin dimension, some of which are new even in
the well studied case of finite group actions. We conclude by showing how
the preservation results for nuclear dimension and decomposition rank under
fixed point algebras and crossed product from \cite%
{hirshberg_rokhlin_2015,gardella_rokhlin_2014,gardella_equivariant_2016}
admit natural extensions to this setting.

In order to make the present paper accessible to readers that are not
necessarily familiar with quantum groups, we recall all the definitions and
results that we use. The interested reader can find more information in the
monographs \cite{timmermann_invitation_2008, de_commer_actions_2016} and the
surveys \cite{maes_notes_1998,kustermans_survey_1999,kustermans_survey_2000}%
. Both compact and discrete quantum groups are subsumed by the more general
class of locally compact quantum groups as defined and studied in \cite%
{kustermans_locally_2000,kustermans_locally_2003,kustermans_operator_2000,vaes_locally_2001}%
. While this allows one to give a unified treatment, it is also technically
more demanding. In order to make the paper more accessible, and since all
our results regard quantum groups that either compact or discrete, we will
present all the notions that we consider in these special cases. The
interested reader is referred to \cite%
{kustermans_locally_2000,kustermans_locally_2003,kustermans_operator_2000,vaes_locally_2001}
as well as the monograph \cite{timmermann_invitation_2008} for more
information on locally compact quantum groups.

For convenience of the readers unfamiliar with first order logic for metric
structures, we include an appendix containing the notions and results from
the logic for metric structures that are used in the present paper. We will
work in the framework of logic for metric structures with domains of
quantification as considered in \cite{farah_model_2014}. A good introduction
to this topic is offered by the monograph \cite{ben_yaacov_model_2008}. A
systematic study of C*-algebras from the perspective of model theory has
been undertaken \cite{farah_model_2016}; see also \cite%
{farah_model_2013,farah_model_2014,farah_model_2014-1,goldbring_kirchbergs_2015,eagle_saturation_2015,farah_relative_2017,carlson_omitting_2014,farah_countable_2013,eagle_fraisse_2016,masumoto_fraisse_2016,masumoto_jiang-su_2016,goldbring_axiomatizability_2016,eagle_pseudoarc_2016}%
.

The present paper is divided into five sections, besides this introduction.
In Section \ref{Section:axiomatization-discrete} we recall the notion of
discrete quantum group, discrete quantum group action, and show that the
class of actions of a given discrete quantum group on C*-algebras is
first-order axiomatizable. The same is done in Section \ref%
{Section:axiomatization-compact} for compact quantum groups. Section \ref%
{Section:crossed} contains some results, to be used in the following
sections, relating the notion of ultraproduct of $G$-C*-algebras with
crossed products and stabilization. In Section \ref{Section:existential} we
introduce the notion of positive existential embeddings and Rokhlin property
for a $G$-C*-algebra, generalizing notions introduced in \cite%
{barlak_spatial_2017} when $G$ is coexact and second countable. The main
results of \cite{barlak_spatial_2017} are then generalized to the case of an
arbitrary compact quantum group $G$. Finally, Section \ref%
{Section:order-zero} contains the notion of $G$-equivariant order zero
dimension for morphisms between $G$-C*-algebras, which is used to define the
Rokhlin dimension of a $G$-C*-algebra. Our main preservation results for
nuclear dimension and decomposition rank for $G$-C*-algebras with finite
Rokhlin dimension are presented here.

Given a subset $X$ of a Banach space $E$, we denote by $[ X] $ the closure
of the linear span of $X$ inside $E$. If $A $ is a C*-algebra, then we let $%
M( A) $ be its \emph{multiplier algebra}, and by $\tilde{A}$ its minimal 
\emph{unitization}. We canonically identify $A$ with an essential ideal of $%
M( A) $ and of $\tilde{A}$. We denote by $\otimes $ the \emph{injective }%
tensor product of Banach spaces, and the \emph{minimal }tensor product of
C*-algebras (which indeed coincides with the injective tensor product as
Banach spaces). The \emph{algebraic }tensor product of complex algebras is
denoted in this paper by $\odot $. A *-homomorphism $\pi\colon A\to B$
between C*-algebras is said to be \emph{nondegenerate} if $[ \pi ( A) B] =B$%
. Given a nondegenerate *-homomorphism $\pi \colon A\rightarrow M(B)$, we
also denote by $\pi $ its unique extension to a (unital, strictly continuous)%
\emph{\ }*-homomorphism $\pi \colon M( A) \rightarrow M( B) $. We say that a
C*-subalgebra $A$ of $B$ is nondegenerate if the inclusion map from $A $ to $%
B$ is nondegenerate. We will frequently use in the following that if $A $ is
a C*-algebra, and $I$ is a dense two-sided ideal in $A$, then $I$ contains
an increasing approximate unit for $A$.

Given a Hilbert space $\mathcal{H}$, we let $\mathbb{B}(\mathcal{H})$ be the
space of bounded linear operators on $\mathcal{H}$, and $\mathbb{K}(\mathcal{%
H})$ be the space of compact operators on $\mathcal{H}$. Given vectors $\xi
,\eta \in \mathcal{H}$ we denote by $\phi _{\xi ,\eta }$ the corresponding
vector linear functional $\phi _{\xi ,\eta }(T):=\langle \eta ,T\xi \rangle $%
. We will frequently use the \emph{leg notation }for elements in a tensor
product \cite[Notation 7.1.1]{timmermann_invitation_2008}. For a Hilbert
space $\mathcal{H}$, we let $\Sigma \in \mathbb{B}(\mathcal{H}\otimes 
\mathcal{H})$ denote the flip unitary. If $T\in \mathbb{B}(\mathcal{H}%
\otimes \mathcal{H})$, we let $T_{12},T_{23},T_{13}\in \mathbb{B}(\mathcal{H}%
\otimes \mathcal{H}\otimes \mathcal{H})$ be defined by $T_{12}=T\otimes 
\mathrm{id}_{\mathcal{H}}$, $T_{23}=\mathrm{id}_{\mathcal{H}}\otimes T$, and 
$T_{13}=\Sigma _{23}T_{12}\Sigma _{23}=\Sigma _{12}T_{23}\Sigma _{12}$. More
generally, one can similarly define the operators $T_{i_{1}\ldots i_{k}}\in
B(\mathcal{H}\otimes \cdots \otimes \mathcal{H})$ for any sequence of
indices $i_{1},\ldots ,i_{k}$.

\section{An axiomatization of discrete quantum group actions\label%
{Section:axiomatization-discrete}}

\subsection{Discrete quantum groups\label{Subsection:group-discrete}}

A \emph{discrete quantum group }$G$\emph{\ }is a C*-algebra $c_{0}(G)$ which
is a direct sum of full matrix algebras endowed with a nondegenerate
*-homomorphism\emph{\ }$\Delta \colon c_{0}(G)\rightarrow M(c_{0}(G)\otimes
c_{0}(G))$ (\emph{comultiplication}) such that:

\begin{itemize}
\item $( \Delta \otimes \mathrm{id}) \circ \Delta =( \mathrm{id}\otimes
\Delta ) \circ \Delta $;

\item $[ ( c_{0}( G) \otimes 1) \Delta ( c_{0}( G) ) ] =[ ( 1 \otimes c_{0}(
G)) \Delta ( c_{0}( G) ) ] =c_{0}( G) \otimes c_{0}( G) $.
\end{itemize}

The discrete quantum group $G$ is said to be \emph{second countable} if $%
c_{0}( G) $ is separable.

Let $G$ be a discrete quantum group. Then there exist an index set $\Lambda$
and finite dimensional Hilbert spaces $\mathcal{H}_{\lambda }$, for $%
\lambda\in\Lambda$, such that $c_{0}( G) $ is isomorphic to $%
\bigoplus_{\lambda \in \Lambda }\mathbb{K}( \mathcal{H}_{\lambda })$. Set $%
c_{0}(G)_{\lambda }=\mathbb{K}( \mathcal{H}_{\lambda })$ for $\lambda\in\Lambda$%
, and denotes its unit by $1_{\lambda }$. We identify $c_0(G)_{\lambda }$ with a
subalgebra of $c_{0}( G) $, $M( c_{0}(G)) $ with $\prod_{\lambda \in \Lambda
}c_{0}(G)_{\lambda }$, and $M( c_{0}( G) \otimes c_{0}( G) ) $ with $%
\prod_{\mu ,\nu \in \Lambda }( c_{0}(G)_{\mu }\otimes c_{0}(G)_{\nu }) $.

For $\lambda,\mu,\nu\in\Lambda$, we fix a homomorphism $\Delta _{\mu ,\nu
}^{\lambda }\colon c_{0}(G)_{\lambda }\rightarrow c_{0}(G)_{\mu }\otimes
c_{0}(G)_{\nu }$ satisfying 
\begin{equation*}
\Delta ( x) =( \Delta _{\mu ,\nu }^{\lambda }( x) ) _{\mu ,\nu \in \Lambda
}\in \prod_{\mu ,\nu \in \Lambda }c_{0}(G)_{\mu }\otimes c_{0}(G)_{\nu }
\end{equation*}%
for every $x\in c_{0}(G)_{\lambda }$.

\begin{remark}
\label{Remark:finite-discrete} Let $\mu ,\nu \in \Lambda $, and set $\Lambda
_{\mu ,\nu }=\{\lambda\in\Lambda\colon \Delta _{\mu ,\nu }^{\lambda }\neq0\}$%
. By \cite[Proposition 2.2]{van_daele_discrete_1996}, the set $%
\Lambda_{\mu,\nu}$ is finite.
\end{remark}

Since $\Delta $ is nondegenerate, we have 
\begin{equation*}
\sum_{\lambda \in \Lambda _{\mu ,\nu }}\Delta _{\mu ,\nu }^{\lambda }(
1_{\lambda }) =1_{\mu }\otimes 1_{\nu }
\end{equation*}%
for every $\mu ,\nu \in \Lambda $. Therefore the canonical extension of $%
\Delta $ to a unital *-homomorphism 
\begin{equation*}
\Delta \colon \prod_{\lambda \in \lambda }c_{0}(G)_{\lambda }\rightarrow
\prod_{\mu ,\nu \in \Lambda }c_{0}(G)_{\mu }\otimes c_{0}(G)_{\nu }
\end{equation*}%
is defined by%
\begin{equation*}
\Delta ( ( x_{\lambda }) _{\lambda \in \Lambda }) =\left( \sum_{\lambda \in
\Lambda _{\mu ,\nu }}\Delta _{\mu ,\nu }^{\lambda }( x) \right) _{\mu ,\nu
\in \Lambda }\in \prod_{\mu ,\nu \in \Lambda }c_{0}(G)_{\mu }\otimes
c_{0}(G)_{\nu }\text{.}
\end{equation*}

\subsection{Discrete quantum group actions\label{Subsection:action-discrete}}

Let $G$ be a discrete group, and let $A$ be a C*-algebra.

\begin{definition}
\label{Definition:discrete-action}A (left) \emph{action }of $G$ on $A$ is an
injective nondegenerate *-homomorphism $\alpha \colon A\rightarrow M( c_{0}(
G) \otimes A) $ such that:

\begin{enumerate}
\item $( \Delta \otimes \mathrm{id}) \circ \alpha =( \mathrm{id}\otimes
\alpha ) \circ \alpha $ (action condition);

\item $[ ( c_{0}( G) \otimes 1) \alpha ( A) ] =c_{0}( G) \otimes A$ (density
condition).
\end{enumerate}
\end{definition}

We will refer to a C*-algebra $A$ endowed with a distinguished action of $G$
as a $G$-C*-algebra. If $( A,\alpha ) $ and $( B,\beta ) $ are $G$%
-C*-algebras, then a *-homomorphism $\phi \colon A\rightarrow B$ is said to
be $G$-\emph{equivariant} if it satisfies $( \mathrm{id}\otimes \phi ) \circ
\alpha =\beta \circ \phi $. It follows from the density condition that, if $%
( u_{i}) $ is an approximate unit for $A$, then $( \alpha ( u_{i}) ) $ is an
approximate unit for $c_{0}( G) \otimes A$.

Let $\alpha \colon A\rightarrow M( c_{0}(G)\otimes A)$ be a nondegenerate
injective *-homomorphism. For $\lambda\in\Lambda$, let $\alpha _{\lambda
}\colon A\rightarrow c_{0}(G)_{\lambda }\otimes A$ be the coordinate
function for $\alpha$, so that 
\begin{equation*}
\alpha ( a) =( \alpha _{\lambda }( a) ) _{\lambda \in \Lambda }\in
\prod_{\lambda \in \Lambda }c_{0}(G)_{\lambda }\otimes A\text{.}
\end{equation*}%
The conditions from Definition \ref{Definition:discrete-action} can be
restated as

\begin{enumerate}
\item[(1')] $(\mathrm{id}_{c_{0}(G)_{\mu }}\otimes \alpha _{\nu })\circ
\alpha _{\mu }=\sum_{\lambda \in \Lambda _{\mu ,\nu }}( \Delta _{\mu ,\nu
}^{\lambda }\otimes \mathrm{id}_{A}) \circ \alpha _{\lambda } $ for every $%
\mu ,\nu \in \Lambda $;

\item[(2')] $[ (c_{0}(G)_{\lambda }\otimes 1)\alpha _{\lambda }( A) ]
=c_{0}(G)_{\lambda }\otimes A$ for every $\lambda \in \Lambda $.
\end{enumerate}

In turn, by Cohen's factorization theorem, (2') is equivalent to the
assertion that for every $\lambda \in \Lambda $, for every $a\in
c_{0}(G)_{\lambda }\otimes A$ of norm at most $1$, and every $\varepsilon >0$%
, there exist $x\in c_{0}(G)_{\lambda }$ and $b\in A$ of norm at most $1$
such that $\Vert (x\otimes 1)\alpha (b)-a\Vert <\varepsilon $.

\begin{example}
If $G$ is a classical discrete group, then one can regard $G$ as a discrete
quantum group by considering the C*-algebra $c_{0}( G) $ of functions from $%
G $ to $\mathbb{C}$ vanishing at infinity. In this case, one has that $%
\Lambda =G$ and $c_{0}(G)_{\lambda }=\mathbb{C}$ for every $\lambda \in G$.
The comultiplication function $\Delta \colon c_{0}( G) \rightarrow M( c_{0}(
G) \otimes c_{0}( G) ) $ is defined by setting $\Delta_{\mu\nu}^{%
\lambda}(1)=\delta_{\lambda,\mu\nu}$ for $\lambda ,\mu ,\nu \in G$.

Every discrete quantum group $G$ such that $c_{0}(G)$ is a commutative
C*-algebra arises from a classical discrete group in this fashion.
\end{example}

\subsection{Axiomatization\label{Subsection:axiomatization-discrete}}

We continue to fix a discrete quantum group $G$. We now describe a natural
(multi-sorted) language $\mathcal{L}_{G}^{\text{C*}}$ that has $G$%
-C*-algebras as structures. For comparison, one can refer to the language
considered in the axiomatization of operator systems from \cite[Appendix B]%
{goldbring_kirchbergs_2015}.

\subsubsection{The language\label{Subsubsection:language-discrete}}

The language $\mathcal{L}_{G}^{\text{C*}}$ has, for every $n\geq 1$, sorts $%
\mathcal{S}^{( n) }$ and $\mathcal{C}^{( n) }$ to be interpreted as $M_{n}( 
\mathbb{C}) \otimes A$ and $M_{n}( \mathbb{C}) $, respectively. For each of
these, the domains of quantifications should be interpreted as the balls
with respect to the norm centered at the origin.

The function and relation symbols consist of:

\begin{enumerate}
\item function and relation symbols for the C*-algebra operations and the
C*-algebra norm on $M_{n}( \mathbb{C}) \otimes A$ and $M_{n}( \mathbb{C}) $
for every $n\in \mathbb{N}$;

\item function symbols for the $M_{n}( \mathbb{C}) $-bimodule structure on $%
M_{n}( \mathbb{C}) \otimes A$;

\item constant symbols the elements of $M_{n}( \mathbb{C}) $;

\item for any bounded linear map $T\colon M_{n}( \mathbb{C}) \rightarrow
M_{m}( \mathbb{C}) $, function symbols $\mathcal{C}^{(n)}\rightarrow 
\mathcal{C}^{(m)}$ and $\mathcal{S}^{(n)}\rightarrow \mathcal{S}^{(m)}$ to
be interpreted as $T$ and $T\otimes \mathrm{id}_{A}$;

\item function symbols for the canonical inclusions of each of the tensor
factors in $M_{n}( \mathbb{C}) \otimes A$;

\item function symbols $\mathcal{S}\rightarrow \mathcal{S}^{(d_{\lambda })}$
to be interpreted as the injective *-homomorphisms $\alpha _{\lambda }\colon
A\rightarrow c_{0}(G)_{\lambda }\otimes A$ that define the action.
\end{enumerate}

It is clear that any $G$-C*-algebra can be seen as an $\mathcal{L}_{G}^{%
\text{C*}}$-structure in a canonical way.

\subsubsection{The axioms\label{Subsubection:axioms-discrete}}

We now describe axioms for the class of $G$-C*-algebras in the language $%
\mathcal{L}_{G}^{\text{C*}}$ described above. Such axioms are designed to
guarantee the following:

\begin{enumerate}
\item the interpretations of the sort $\mathcal{S}^{( n) }$ and $\mathcal{C}%
^{( n) }$ are C*-algebras;

\item the domains of quantifications are interpreted as the balls (see \cite[%
Example 2.2.1]{farah_model_2016});

\item the sorts $\mathcal{S}^{( n) }$ and $\mathcal{C}^{( n) }$, the symbols
for the canonical inclusions of each of the tensor factors in $M_{n}( 
\mathbb{C}) \otimes A$, the symbols for the maps $T$ and $T\otimes \mathrm{id%
}_{A}$ for any bounded linear map $T\colon M_{n}( \mathbb{C}) \rightarrow
M_{m}( \mathbb{C}) $, the symbols for the $M_{n}( \mathbb{C}) $-bimodule
structure on $M_{n}( \mathbb{C}) \otimes A$, and the symbols for the
elements of $M_{n}( \mathbb{C}) $, are interpreted as they should be (see 
\cite[Appendix C]{goldbring_kirchbergs_2015});

\item the interpretation of the symbols for the maps $\alpha _{\lambda }$
that define the action are isometric *-homomorphisms;

\item for every $\mu ,\nu \in \Lambda $, 
\begin{equation*}
(\mathrm{id}_{c_{0}(G)_{\mu }}\otimes \alpha _{\nu })\circ \alpha _{\mu
}=\sum_{\lambda \in \Lambda _{\mu ,\nu }}( \Delta _{\mu ,\nu }^{\lambda
}\otimes \mathrm{id}_{A}) \circ \alpha _{\lambda };
\end{equation*}

\item for every $a\in c_{0}(G)_{\lambda }\otimes A $ of norm at most $1$,
and every $\varepsilon >0$, there exist $x\in c_{0}(G)_{\lambda }$ and $b\in
A$, both of norm at most $1$, such that $\| ( x\otimes 1) \alpha ( b) -a\|
<\varepsilon $.
\end{enumerate}

It is clear from the discussion above that these axioms indeed axiomatize
the class of $G$-C*-algebras. Furthermore, it is easy to see that these
axioms are all given by conditions of the form $\sigma \leq r$ where $\sigma 
$ is a positive primitive $\forall \exists $-$\mathcal{L}_{G}^{\text{C*}}$%
-sentence and $r\in \mathbb{R}$. Therefore, this shows that the class of $G$%
-C*-algebras is positively primitively $\forall \exists $-axiomatizable in
the language $\mathcal{L}_{G}^{\text{C*}}$ in the sense of Definition \ref%
{Definition:pp-axiomatizable}. One can observe that, when $G$ is
second-countable, the language $\mathcal{L}_{G}^{\text{C*}}$ is separable
for the class of $G$-C*-algebras; see Definition \ref%
{Definition:separable-language}. More generally, the density character of $%
\mathcal{L}_{G}^{\text{C*}}$ for the class of $G$-C*-algebras is equal to
the size of $\Lambda $. It is not difficult to verify that an $\mathcal{L}%
_{G}^{\text{C*}}$-morphism between $G$-C*-algebras is precisely a $G$%
-equivariant *-homomorphism, while an $\mathcal{L}_{G}^{\text{C*}}$%
-embedding is an injective $G$-equivariant *-homomorphism.

\subsubsection{Ultraproducts\label{Subsubection:ultraproducts-discrete}}

Once $G$-C*-algebras are regarded as $\mathcal{L}_{G}^{\text{C*}}$%
-structures as described above, one can consider ultraproducts of $G$%
-C*-algebras as a particular instance of ultraproducts in first-order logic
for metric structures; see \cite{farah_model_2014}. It follows from \L os'
theorem that the ultraproduct of $G$-C*-algebras is again a $G$-C*-algebra.

More generally, one can consider reduced products with respect to an
arbitrary filter $\mathcal{F}$ as defined in \cite{ghasemi_reduced_2016}. It
is easy to see that, in the particular case when $\mathcal{F}$ is the filter
of cofinite subsets of $\mathbb{N}$, then the reduced power of a $G$%
-C*-algebra $( A,\alpha ) $ with respect to $\mathcal{F}$ coincides with the 
$G$-C*-algebra $( A_{\infty },\alpha _{\infty }) $ constructed in \cite%
{barlak_spatial_2017}. It follows from the fact that the class of $G$%
-C*-algebras is positively primitively $\forall \exists $-axiomatizable in
the language $\mathcal{L}_{G}^{\text{C*}}$ together with Corollary \ref%
{Corollary:Los} that the reduced product of $G$-C*-algebras is again a $G$%
-C*-algebra. This recovers \cite[Lemma 2.6]{barlak_spatial_2017} as a
particular case.

\subsubsection{Other languages\label{Subsubsection:other-discrete}}

Occasionally, it is useful to consider C*-algebras as structures in a
language other than the standard language for C*-algebras $\mathcal{L}^{%
\text{C*}}$. These other languages are useful to capture properties that are
preserved by more general classes of morphisms, rather than just
*-homomorphisms, such as for example completely positive contractive maps or
completely positive contractive order zero maps. Several languages are
considered in \cite[Section 3]{gardella_equivariant_2016}. For each such
language $\mathcal{L}$, one can consider a corresponding $G$-equivariant
language $\mathcal{L}_{G}$. This can defined as in Subsubsection \ref%
{Subsubsection:other-discrete}, by starting with the language $\mathcal{L}$
rather than $\mathcal{L}^{\text{C*}}$.

\section{An axiomatization of compact quantum group actions\label%
{Section:axiomatization-compact}}

\subsection{Compact quantum groups\label{Subsection:groups-compact}}

A (reduced, C*-algebraic) \emph{compact quantum group }$G$ is given by a
unital C*-algebra $C( G) $ endowed with a unital *-homomorphism (\emph{%
comultiplication}) $\Delta \colon C( G) \rightarrow C( G) \otimes C( G) $
and a faithful state $h\colon C(G)\to \mathbb{C}$ (the \emph{Haar state})
satisfying

\begin{itemize}
\item $( \Delta \otimes \mathrm{id}) \circ \Delta =( \mathrm{id}\otimes
\Delta ) \circ \Delta $,

\item $[ \Delta ( C( G) ) ( 1\otimes C( G) ) ] =[ \Delta ( C( G) ) ( C( G)
\otimes 1) ] =C( G) \otimes C( G) $, and

\item $( h\otimes \mathrm{id}) \circ \Delta =( \mathrm{id}\otimes h) \circ
\Delta =h$
\end{itemize}

where in the last equation we identify $\mathbb{C}$ with the space of scalar
multiples for the identity in $C( G) $. The first two conditions assert that 
$C( G) $ endowed with the comultiplication $\Delta $ is a unital Hopf
C*-algebra \cite[Definition 2.2]{nest_equivariant_2010}. A compact quantum
group is said to be \emph{second countable }if $C( G) $ is separable.

A \emph{unitary representation} of $G$ on a\emph{\ }Hilbert space $\mathcal{H%
}$ is a unitary $u\in M\left( C(G)\otimes \mathbb{K}(\mathcal{H})\right) $
satisfying $(\Delta \otimes \mathrm{id})(u)=u_{13}u_{23}$. In the following
we will only consider the case when $\mathcal{H}$ is finite-dimensional, in
which case one has that $u\in C\left( G\right) \otimes \mathbb{B}\left( 
\mathcal{H}\right) $. We also identify the unitary representation $u$ of $G$
on $\mathcal{H}$ with the linear map $\mathcal{H}\rightarrow C(G)\otimes 
\mathcal{H}$ given by $\eta \mapsto u(1\otimes \eta )$ for all $\eta \in 
\mathcal{H}$. A subspace $\mathcal{K}\subseteq\mathcal{H}$ is \emph{invariant%
} if $u(1\otimes \mathcal{K} )\subseteq C(G)\otimes \mathcal{K}$. Direct sum
and tensor product of unitary representations is defined in the usual way. A
unitary representation is called \emph{irreducible} if it has no non-trivial
invariant closed subspace. Every irreducible unitary representation of $G$
is finite-dimensional, and every unitary representation of $G$ is equivalent
to a direct sum of irreducible unitary representations.

We let $\mathrm{Rep}( G) $ be the set of unitary representations of $G$ on
finite-dimensional Hilbert spaces. For $\lambda \in \mathrm{Rep}( G) $ we
let $u^{\lambda }\in C( G) \otimes \mathbb{K}( \mathcal{H}_{\lambda }) $ be
the associated representation, and we denote by $d_{\lambda }$ the dimension
of $\mathcal{H}_{\lambda }$. We also define $\mathrm{Irr}( G) $ to be the
set of equivalence classes of irreducible unitary representations of $G$.
Given $\lambda ,\mu \in \mathrm{Irr}( G) $, we set $\delta _{\lambda,\mu }=1$
if $\lambda$ and $\mu $ are equivalent, and $\delta _{\lambda, \mu }=0$
otherwise. For $\lambda \in \mathrm{Rep}( G) $ and $\xi ,\eta \in \mathcal{H}%
_{\lambda }$, set $u_{\xi ,\eta }^{\lambda }%
%=\left\langle \xi \right\vert _{(2)}u^{\lambda}\left\vert \eta \right\rangle _{(2)}
=(\mathrm{id}\otimes \phi _{\xi ,\eta })(u^{\lambda })\in C( G) $.\ These
are called the \emph{matrix coefficients }of the unitary representation $%
\lambda $. If $\lambda \in \mathrm{Irr}( G) $, then we fix an orthonormal
basis $\{ e^{\lambda}_{k}\}_{k=1}^{d_{\lambda }} $ of $\mathcal{H}_{\lambda }$
and set $u_{jk}^{\lambda }=u_{e_{j},e_{k}}^{\lambda }$ for $1\leq j,k\leq
d_{\lambda }$. For an arbitrary $\lambda \in \mathrm{Rep}( G) $, we write $%
\lambda $ as a direct sum of irreducible unitary representations $\lambda
_{1}\oplus \cdots \oplus \lambda _{n}$, and then we consider the orthonormal
basis of $\mathcal{H}_{\lambda }=\mathcal{H}_{\lambda _{1}}\oplus \cdots
\oplus \mathcal{H}_{\lambda _{n}}$ associated with the given orthonormal
bases of $\mathcal{H}_{\lambda _{1}},\ldots ,\mathcal{H}_{\lambda _{n}}$.

Define $\mathcal{O}(G)$ to be the dense selfadjoint subalgebra of $C(G)$
given by 
\begin{equation*}
\mathcal{O}( G) =\left\{ u_{\xi ,\eta }^{\lambda }\colon \lambda \in \mathrm{%
Rep}( G) ,\xi ,\eta \in \mathcal{H}_{\lambda }\right\} =\mathrm{span}\left\{
u_{ij}^{\lambda }\colon \lambda \in \mathrm{Irr}( G) ,1\leq i,j\leq
d_{\lambda }\right\},
\end{equation*}%
and observe that it is invariant under the comultiplication. The induced
*-bialgebra structure on $\mathcal{O}( G) $ turns it into a Hopf *-algebra 
\cite[Definition 1.3.24]{timmermann_invitation_2008}. The counit map $%
\epsilon $ on $\mathcal{O}( G) $ is defined by $\epsilon ( u_{\xi ,\eta
}^{\lambda }) =\left\langle \xi ,\eta \right\rangle$, while the antipode map 
$S$ is given by$S( u_{\xi ,\eta }^{\lambda }) =( u_{\eta ,\xi }^{\lambda })
^{\ast }$, for $\xi,\eta\in\mathcal{H}_{\lambda }$ and $\lambda\in$ Rep$(G)$.

\begin{remark}
The C*-algebra $C( G) $ together with its comultiplication can be recovered
from $\mathcal{O}( G) $ as the closure of the image of $\mathcal{O}( G) $ in
the GNS representation induced by the restriction of the Haar state to $%
\mathcal{O}( G) $.
\end{remark}

By \cite[Theorem~1.8]{de_commer_actions_2016}, for every $\lambda \in 
\mathrm{Irr}(G)$ there is a positive invertible operator $F_{\lambda }\in 
\mathbb{K}(\mathcal{H}_{\lambda })$ satisfying: 
%There exists a linear functional $f$ on $\mathcal{O}( G) $ with
%the property that, for any $\lambda \in \mathrm{Rep}( G) $, $%
%F_{\lambda }:=( f\otimes \mathrm{id}) u^{\lambda }\in \mathbb{K}%
%( \mathcal{H}^{\lambda }) $ is a positive invertible operator
%\cite[Theorem 1.8]{de_commer_actions_2016}. One sets $\dim _{q}(
%\lambda ) :=\mathrm{\mathrm{Tr}}[ F_{\lambda }] $ to be the
%quantum dimension of $\lambda $. For any $\lambda ,\mu \in \mathrm{Irr}%
%( G) $, $\xi ,\eta ,\zeta ,\chi \in \mathcal{H}^{\lambda }$, one
%has that%
\begin{equation*}
h(u_{\xi ,\eta }^{\lambda }(u_{\zeta ,\chi }^{\lambda })^{\ast })=\frac{%
\left\langle \xi ,\zeta \right\rangle \left\langle \chi ,F_{\lambda }\eta
\right\rangle }{\dim _{q}(\lambda )}\ \ \mbox{and}\ \ h((u_{\xi ,\eta
}^{\lambda })^{\ast }u_{\zeta ,\chi }^{\lambda })=\frac{\left\langle \xi
,(F_{\lambda })^{-1}\zeta \right\rangle \left\langle \chi ,\eta
\right\rangle }{\dim _{q}(\lambda )}\text{,}
\end{equation*}%
for all $\xi ,\eta \in \mathcal{H}_{\lambda }$, where $\dim _{q}(\lambda )=%
\mathrm{\mathrm{Tr}}(F_{\lambda })$ is the \emph{quantum dimension} of $%
\lambda $. Furthermore, one has 
\begin{equation*}
h(u_{\xi ,\eta }^{\lambda }(u_{\zeta ,\chi }^{\mu })^{\ast })=h((u_{\xi
,\eta }^{\lambda })^{\ast }u_{\zeta ,\chi }^{\mu })=0
\end{equation*}%
whenever $\lambda ,\mu $ are nonequivalent, for all $\xi ,\eta \in \mathcal{H%
}_{\lambda }$, and all $\zeta ,\chi \in \mathcal{H}_{\mu }$. In particular,
we deduce that%
\begin{equation}
h(u_{ij}^{\lambda }(u_{k\ell }^{\mu })^{\ast })=\delta _{\lambda \mu }\delta
_{ik}\frac{F_{\ell j}^{\lambda }}{\dim _{q}(\lambda )}\text{.\label%
{Equation:Schur}}
\end{equation}%
for $\lambda ,\mu \in \mathrm{Irr}(G)$, $1\leq i,j\leq d_{\lambda }$, and $%
1\leq k,\ell \leq d_{\mu }$. %One can use the operator $%
%F_{\lambda }$ to define the \emph{contragradient }representation of $\lambda
%$. This is the representation of $G$ on the conjugate space $\overline{%
%\mathcal{H}}_{\lambda }$ of $\mathcal{H}_{\lambda }$ endowed with the inner
%product $\left\langle \overline{\xi },\overline{\eta }\right\rangle
%=\left\langle \eta ,F_{\lambda }\xi \right\rangle $ defined by $\overline{%
%\xi }\mapsto \overline{u( 1\otimes \eta ) }$.
Similarly to the classical case, to each $\lambda ,\mu \in \mathrm{Irr}(G)$
there corresponds a \emph{contragradient }representation $\overline{\lambda} 
\mathrm{Irr}(G)$ (cf. \cite[Section 6]{maes_notes_1998}).

For $\lambda \in \mathrm{Rep}(G)$, set 
\begin{equation*}
C(G)_{\lambda }=\left\{ u_{\xi ,\eta }^{\lambda }\colon \xi ,\eta \in 
\mathcal{H}_{\lambda }\right\} =\mathrm{span}\left\{ u_{ij}^{\lambda }\colon
1\leq i,j\leq d_{\lambda }\right\} \text{.}
\end{equation*}%
Then $C(G)_{\lambda }$ is a finite-dimensional subspace of $C(G)$ which is
invariant under the comultiplication. If $\lambda ,\mu \in \mathrm{Irr}(G)$,
then $C(G)_{\lambda }$ and $C(G)_{\mu }$ are orthogonal with respect to the
inner product defined by the Haar state $h$. It follows that $\mathcal{O}(G)$
is equal to the algebraic direct sum of $C(G)_{\lambda }$ for $\lambda \in 
\mathrm{Irr}(G)$. For $\lambda ,\mu \in \mathrm{Rep}(G)$, we have%
\begin{equation*}
C(G)_{\lambda }+C(G)_{\mu }=C(G)_{\lambda \oplus \mu }\text{,}\ \ \mathrm{%
\mathrm{span}}(C(G)_{\lambda }C(G)_{\mu })=C(G)_{\lambda \otimes \mu }\text{,%
}\ \ \mbox{and}\ \ (C(G)_{\lambda })^{\ast }=C(G)_{\overline{\lambda }}\text{%
.}
\end{equation*}

Define $\mathcal{O}(\hat{G})$ to be the space of linear functionals on $C(
G) $ of the form $x\mapsto h( xy) $ for some $y\in \mathcal{O}( G) $. For $%
\lambda \in \mathrm{Rep}( G) $ and $1\leq s,m\leq d_{\lambda }$, let $\omega
_{sm}^{\lambda }\in \mathcal{O}(\hat{G})$ be given by%
\begin{equation*}
\omega _{sm}^{\lambda }( x) =\dim _{q}( \lambda ) \sum_{k=1}^{d_{\lambda }}(
F_{\lambda }) _{mk}^{-1}h(x( u_{sk}^{\lambda }) ^{\ast })\text{.}
\end{equation*}%
It follows from Equation \ref{Equation:Schur} that 
\begin{equation*}
\omega _{sm}^{\lambda }( u_{ln}^{\mu }) =\delta _{\mu \lambda }\delta
_{ls}\delta _{mn} \ \ \mbox{and} \ \ ( \omega _{sm}^{\lambda }\otimes \omega
_{ln}^{\mu }) \circ \Delta =\delta _{ml}\delta _{\mu \nu }\omega _{sn}^{\mu }
\end{equation*}%
for $\lambda ,\mu \in \mathrm{Irr}( G) $, $1\leq s,m\leq d_{\lambda }$, and $%
1\leq l,n\leq d_{\lambda }$. More generally, one can define%
\begin{equation*}
\omega _{\xi ,\eta }^{\lambda }=\dim _{q}( \lambda ) \sum_{k=1}^{d_{\lambda
}}\left\langle (F_{\lambda })^{-1}e_{k},\eta \right\rangle h(x( u_{\xi
,e_{k}}^{\lambda }) ^{\ast })
\end{equation*}%
for $\xi ,\eta \in \mathcal{H}_{\lambda }$, and observe that$( \omega _{\xi
,\eta }^{\lambda }\otimes \omega _{\zeta ,\chi }^{\lambda }) \circ \Delta
=\left\langle \eta ,\zeta \right\rangle \omega _{\xi ,\chi }^{\lambda }$ for
every $\xi ,\eta ,\zeta ,\chi \in \mathcal{H}_{\lambda }$.

Set $P_{ks}^{\lambda }=( \omega _{ks}^{\lambda }\otimes \mathrm{id}) \circ
\Delta $ for $1\leq k,s\leq d_{\lambda }$. From the relations above, it
follows that $P_{sm}^{\lambda }P_{ln}^{\mu }=\delta _{\lambda \mu }\delta
_{ml}P_{sn}^{\lambda }$ for $\lambda ,\mu \in \mathrm{Irr}( G) $, $1\leq
s,m\leq d_{\lambda }$, and $1\leq l,n\leq d_{\mu }$. Furthermore, $%
P_{kk}^{\lambda }$ is a projection operator onto 
\begin{equation*}
C( G) _{\lambda ,k}:=\mathrm{span}\left\{ u_{ik}^{\lambda }\colon 1\leq
i\leq d_{\lambda }\right\} \text{,}
\end{equation*}%
and $P^{\lambda }:=\sum_{k=1}^{d_{\lambda }}P_{kk}^{\lambda }$ is a
projection onto $C( G) _{\lambda }$.

\subsection{The fundamental unitaries}

We now recall the important concept of fundamental unitaries associated with
a given compact quantum group $G$.

\begin{definition}
\label{Definition:multiplicative} The fundamental unitaries associated with $%
G$ are the unitary operators $V,W\in \mathbb{B}( L^{2}( G) \otimes L^{2}( G)
) $ defined as follows.

\begin{itemize}
\item The operator $W$ is the adjoint of the unitary operator induced by the
linear map with dense image $C( G) \otimes C( G) \rightarrow C( G) \otimes
C( G) $ given by $x\otimes y\mapsto \Delta ( y) ( x\otimes 1) $.

\item Similarly, $V$ is the unitary operator induced by the linear map with
dense image $C( G) \otimes C( G) \rightarrow C( G) \otimes C( G) $ given by $%
x\otimes y\mapsto \Delta ( x) ( 1\otimes y) $.
\end{itemize}
\end{definition}

Abbreviate $\mathbb{K}(L^{2}(G))$ to $\mathbb{K}_{G}$, and identify the
Banach-space dual $\mathbb{K}_{G}^{\ast }$ with the predual of $\mathbb{B}%
(L^{2}(G))$. The fundamental unitaries satisfy the \emph{pentagon equations}%
\begin{equation*}
V_{12}V_{13}V_{23}=V_{23}V_{12}\text{\quad and\quad }%
W_{12}W_{13}W_{23}=W_{23}W_{12}\text{.}
\end{equation*}%
The C*-algebra $C(G)\subseteq \mathbb{B}(L^{2}(G))$ can be recovered as the 
\emph{left leg }of $W$:%
\begin{equation*}
C(G)=[\{\left( \mathrm{id}\otimes \phi _{\xi ,\xi ^{\prime }}\right) W\colon
\xi ,\xi ^{\prime }\in L^{2}(G)\}]=[\left\{ (\mathrm{id}\otimes \omega
)(W)\colon \omega \in \mathbb{K}_{G}^{\ast }\right\} ]\text{.}
\end{equation*}%
Similarly, the group C*-algebra $C^{\ast }(G)=c_{0}(\hat{G})$ can be
recovered as the \emph{right leg }of $W$: 
\begin{equation*}
C^{\ast }(G)=[\{\left( \phi _{\xi ,\xi ^{\prime }}\otimes \mathrm{id}\right)
(W) \colon \xi ,\xi ^{\prime }\in L^{2}(G)\}]=[\left\{ (\omega \otimes 
\mathrm{id})(W)\colon \omega \in \mathbb{K}_{G}^{\ast }\right\} ]\text{.}
\end{equation*}

One can recover $C( G) $ as the right leg of $V$, and $c_{0}(\hat{G})$ as
the left leg of $V$. Furthermore, $W$ belongs to $M(C( G) \otimes c_{0}(\hat{%
G}))$ and $V$ belongs to $M(c_{0}(\hat{G})\otimes C( G) )$.

%Denote by $\Sigma\in \mathbb{B}(L^2(G)\otimes L^2(G))$ the flip unitary, and
Define normal *-homomorphisms $\Delta ,\hat{\Delta }\colon \mathbb{B}(
L^{2}( G) ) \rightarrow \mathbb{B}( L^{2}( G) \otimes L^{2}( G) ) $ by%
\begin{equation*}
\Delta ( x) =W^{\ast }( 1\otimes x) W \ \ \mbox{ and } \ \ \hat{\Delta }( x)
=\hat{W}{}^{\ast }{}( 1\otimes y) \hat{W}\text{,}
\end{equation*}%
for all $x,y\in \mathbb{B}(L^2(G))$. The restriction of $\Delta $ to $C( G) $
gives the comultiplication of $C( G) $, while the restriction of $\hat{%
\Delta }$ to $c_{0}(\hat{G})$ gives the comultiplication of $c_{0}(\hat{G})$.

\subsection{Compact quantum group actions\label{Subsection:action-compact}}

Let $G$ be a (reduced, C*-algebraic) compact quantum group, and let $A$ be a
C*-algebra. %The notion of action of $G$ on $A $ is
%considered in \cite[Definition 1.15]{barlak_spatial_2017}; see also \cite[%
%Section 9.2]{timmermann_invitation_2008} and \cite[Definition 2.3]%
%{nest_equivariant_2010}.

\begin{definition}
\label{Definition:compact-action}A (left, continuous) \emph{action} of $G$
on $A$ is an injective nondegenerate *-homomorphism $\alpha \colon
A\rightarrow C( G) \otimes A$ satisfying the following conditions:

\begin{enumerate}
\item $( \Delta \otimes \mathrm{id}) \circ \alpha =( \mathrm{id}\otimes
\alpha ) \circ \alpha $ (action condition);

\item $[ ( C( G) \otimes 1) \alpha ( A) ] =C( G) \otimes A$, where $C( G)
\otimes 1$, $\alpha ( G) $, and $C( G) \otimes A$ are canonically regarded
as subalgebras of $M( C( G) \otimes A) $ (density condition).
\end{enumerate}
\end{definition}

If $A$ is a C*-algebra endowed with a distinguished action $\alpha $ of $G$
on $A$, then we refer to the pair $( A,\alpha ) $ as a $G$-C*-algebra. A
canonical example of $G$-C*-algebra is $C( G) $ endowed with the left
translation action of $G$ on $C( G) $ given by the comultiplication $\Delta $%
.

Suppose now that $\lambda \in \mathrm{Rep}( G) $. An \emph{intertwiner }%
between $\lambda $ and $\alpha $ is a linear map $v\colon \mathcal{H}%
_{\lambda }\rightarrow A$ such that%
\begin{equation*}
\alpha ( v\xi ) =( \mathrm{id}\otimes v) u^{\lambda }( 1\otimes \xi )
\end{equation*}%
for every $\xi \in \mathcal{H}_{\lambda }$. The space of intertwiners
between $\lambda $ and $\alpha $ is denoted by $\mathrm{Int}( \lambda
,\alpha ) $. The $\lambda $-\emph{isotypical component }(or $\lambda $-\emph{%
spectral subspace}) is the closed subspace%
\begin{equation*}
A_{\lambda }=\left\{ v\xi \colon \xi \in \mathcal{H}_{\lambda },v\in \mathrm{%
Int}( \lambda ,\alpha ) \right\} \text{.}
\end{equation*}

Define $E_{ks}^{\lambda }\colon A\to A_{\lambda }$ as $E_{ks}^{\lambda}=(
\omega _{ks}^{\lambda }\otimes \mathrm{id}) \circ \alpha $ for $1\leq
k,s\leq d_{\lambda }$. It follows from the Schur orthogonality relations %
\eqref{Equation:Schur} that $E_{sm}^{\lambda }E_{ln}^{\mu }=\delta _{\lambda
\mu }\delta _{ml}E_{sn} $ for $\lambda ,\mu \in \mathrm{Irr}(G)$. The proof
of \cite[Theorem 1.5]{podles_symmetries_1995} shows that $A_{\lambda }$ is a
closed subspace of $A $. Indeed, one can alternatively describe $A_{\lambda
} $ as%
\begin{equation*}
A_{\lambda }=\left\{ a\in A\colon \alpha ( a) \in C( G) _{\lambda }\otimes
A\right\} =\left\{ a\in A\colon\alpha ( a) \in C( G) _{\lambda }\otimes
A_{\lambda }\right\} \text{.}
\end{equation*}%
Furthermore $A_{\lambda }+A_{\mu }\subseteq A_{\lambda \oplus \mu }$, $%
A_{\lambda }A_{\mu }\subseteq A_{\lambda \otimes \mu }$, and $( A_{\lambda
}) ^{\ast }=A_{\overline{\lambda }}$ for every $\lambda ,\mu \in \mathrm{Rep}%
( G) $.

It is also shown in \cite[Theorem 1.5]{podles_symmetries_1995} that $E^{\lambda }=\sum_{k=1}^{d_{\lambda }}E_{kk}^{\lambda }$ is a projection onto $%
A_{\lambda }$. Moreover, $E_{kk}^{\lambda }$ is a projection onto a closed
subspace $A_{\lambda ,k}$ of $A_{\lambda }$ of dimension $c_{\alpha }\in 
\mathbb{N}\cup \left\{ \infty \right\} $. Such a dimension $c_{\alpha }$ is
called the \emph{multiplicity }of $\lambda $ in the spectrum of $\alpha $.
Let $(a_{1,i}^{\lambda })_{i=1}^{c_{\lambda }}$ be a basis of $A_{\lambda
,1} $. Then $(E_{s1}^{\lambda }(a_{1,i}^{\lambda }))_{i=1}^{c_{\lambda }}$
is a basis of $A_{\lambda ,s}$. For $s=1,\ldots ,d_{\lambda }$, set $%
a_{s,i}^{\lambda }=E_{s1}^{\lambda }(a_{1,i}^{\lambda })$. Then $%
E_{sk}^{\lambda }(a_{k,i}^{\lambda })=a_{s,i}^{\lambda }$, and $\alpha
(a_{k,i}^{\lambda })=\sum_{s=1}^{d_{\lambda }}u_{ks}^{\lambda }\otimes
a_{s,i}^{\lambda }$ for every $k=1,2,\ldots ,d_{\lambda }$. From this it
easily follows that $(\epsilon \otimes \mathrm{id})\alpha (c)=c$ for every $%
c\in A_{\lambda }$. Define $\mathcal{O}(A)=\mbox{span}\left\{ A_{\lambda
}\colon \lambda \in \mathrm{Irr}(A)\right\} $. It is shown in \cite[Theorem
1.5]{podles_symmetries_1995} that $\mathcal{O}(A)$ is a dense selfadjoint
subalgebra of $A$, called \emph{Podle\'{s} subalgebra }or \emph{algebraic
core} of $(A,\alpha )$.

The spectral subspace $A^{\alpha }:=A_{\mathrm{t}}$ associated with the
trivial representation $\mathrm{t}$ of $G$ is a nondegenerate C*-subalgebra
of $A$, called the \emph{fixed point algebra}. The projection $E^{\mathrm{t}%
} $ onto $A^{\alpha }$ is a faithful conditional expectation, and $%
A_{\lambda } $ is an $A^{\alpha }$-bimodule for every $\lambda \in \mathrm{%
Rep}( G) $. The formula $\left\langle a,b\right\rangle =E^{\mathrm{t}}(
a^{\ast }b) $ defines a (right) full Hilbert $A^{\alpha }$-module structure
on $A$ \cite[Lemma 3.8 and Lemma 3.19]{de_commer_actions_2016}, and two
spectral subspaces $A_{\lambda },A_{\mu }$ associated with nonequivalent $%
\lambda ,\mu \in \mathrm{Irr}( G) $ are orthogonal with respect to this
C*-bimodule structure.

\begin{lemma}
\label{Lemma:norm-compact} Let $\lambda \in \mathrm{Rep}( G) $. For $%
a=\sum_{ij}u_{ij}^{\lambda }\otimes a_{ij}\in C( G) _{\lambda }\otimes A$,
one has%
\begin{equation*}
\| a\| _{C( G) \otimes A}=\sup \left\{ \| \sum_{ij}\mu _{ij}a_{ij}\|
_{A}\colon \mu _{ij}\in \mathbb{C},\| \sum_{ij}\mu _{ij}u_{ij}^{\lambda }\|
_{C( G) }\leq 1\right\} \text{.}
\end{equation*}
\end{lemma}

\begin{proof}
It is enough to observe that $C( G) _{\lambda }\otimes A$ is endowed with
the injective Banach space tensor product, and that $\left\{ u_{ij}^{\lambda
}\colon 1\leq i,j\leq d_{\lambda }\right\} $ is a basis for $C( G) _{\lambda
}$.
\end{proof}

We close this subsection with the following lemma.

\begin{lemma}
\label{Lemma:action-density-compact} Let $\lambda \in \mathrm{Irr}( G)$.

\begin{enumerate}
\item Let $1\leq i,j\leq d_{\lambda }$. Then $( P_{ij}^{\lambda }\otimes 
\mathrm{id}) \circ \alpha =\alpha \circ E_{ij}^{\lambda }$. Thus, $(
P^{\lambda }\otimes \mathrm{id}) \circ \alpha =\alpha \circ E^{\lambda }$.

\item Set $E^{\lambda }=\sum_{j=1}^{d_{\lambda }}E^{\lambda}_{jj}$. For $a\in A$,
one has%
\begin{equation*}
1\otimes E^{\lambda }( a) =\sum_{1\leq s,t\leq d_{\lambda }}( (
u_{st}^{\lambda }) ^{\ast }\otimes 1) ( \alpha \circ E_{st}^{\lambda }) ( a) 
\text{.}
\end{equation*}
\end{enumerate}
\end{lemma}

\begin{proof}
Part~(1) is straightforward, so we show (2). By linearity, we can assume
that $a$ has the form $a=a_{k,i}^{\lambda }$ for some $1\leq i\leq
c_{\lambda }$ and $1\leq k\leq d_{\lambda }$. Using the fact that $\mathcal{O%
}( G) $ is a Hopf *-algebra, we have%
\begin{equation*}
( m\circ ( S\otimes \mathrm{id}) \circ \Delta ) ( x) =\varepsilon ( x) 1
\end{equation*}%
for every $x\in \mathcal{O}( G) $. Therefore%
\begin{align*}
1\otimes a_{k,i}^{\lambda }&=( \epsilon \otimes \mathrm{id}) \alpha (
a_{k,i}^{\lambda }) =( m\otimes \mathrm{id}) ( S\otimes \mathrm{id}\otimes 
\mathrm{id}) ( \Delta \otimes \mathrm{id}) \alpha ( a_{k,i}^{\lambda }) \\
&=\sum_{1\leq j\leq d_{\lambda }}( m\otimes \mathrm{id}) ( S\otimes \mathrm{%
id}\otimes \mathrm{id}) ( \Delta \otimes \mathrm{id}) ( u_{kj}\otimes
a_{j,i}^{\lambda }) \\
&=\sum_{1\leq j,s\leq d_{\lambda }}( m\otimes \mathrm{id}) ( S\otimes 
\mathrm{id}\otimes \mathrm{id}) ( u_{ks}^{\lambda }\otimes u_{sj}^{\lambda
}\otimes a_{j,i}^{\lambda }) \\
&=\sum_{1\leq j,s\leq d_{\lambda }}\left( \left( u_{sk}^{\lambda }\right)
^{\ast }u_{sj}^{\lambda }\otimes a_{j,i}^{\lambda }\right) \\
&=\sum_{1\leq s\leq d_{\lambda }}\left( \left( u_{sk}^{\lambda }\right)
^{\ast }\otimes 1\right) \left( \sum_{j=1}^{d_{\lambda }}u_{sj}^{\lambda
}\otimes a_{j,i}^{\lambda }\right) \\
&=\sum_{1\leq s\leq d_{\lambda }}\left( \left( u_{sk}^{\lambda }\right)
^{\ast }\otimes 1\right) \alpha \left( a_{s,i}^{\lambda }\right) \\
&=\sum_{1\leq s\leq d_{\lambda }}\left( \left( u_{sk}^{\lambda }\right)
^{\ast }\otimes 1\right) \left( \alpha \circ E_{sk}^{\lambda }\right) \left(
a_{k,i}^{\lambda }\right) \\
&=\sum_{1\leq s,t\leq d_{\lambda }}\left( \left( u_{st}^{\lambda }\right)
^{\ast }\otimes 1\right) \left( \alpha \circ E_{st}^{\lambda }\right) \left(
a_{k,i}^{\lambda }\right) \text{.}
\end{align*}
This concludes the proof.
\end{proof}

\subsection{Axiomatization\label{Subsection:axiomatization-compact}}

Throughout this subsection, we fix a compact quantum group $G $. Our goal is
to show that there is a natural language $\mathcal{L}_{G}^{\text{C*}}$ in
the logic for metric structures that allows one to regard $G$-C*-algebra as $%
\mathcal{L}_{G}^{\text{C*}}$-structures, in such a way that the class of $G$%
-C*-algebra is axiomatizable. We begin by describing the language, and then
present the axioms for the class of $G$-C*-algebras.

\subsubsection{The language\label{Subsubsection:language-compact}}

The language $\mathcal{L}_{G}^{\text{C*}}$ has the following sorts:

\begin{itemize}
\item a sort $\mathcal{S}$, to be interpreted as the C*-algebra $A$ where $G$
acts;

\item a sort $\mathcal{S}^{( 0) }$ to be interpreted as $C( G) $;

\item a sort $\mathcal{S}^{( 1) }$ to be interpreted as $C( G) \otimes A$;

\item a sort $\mathcal{C}$ to be interpreted as the algebra of complex
numbers.
\end{itemize}

For each of the sorts above, the domains of quantifications are as follows:

\begin{itemize}
\item $\mathcal{S}$ has domains of quantification $D_{n}^{\lambda }$ for $%
\lambda \in \mathrm{Rep}( G) $ and $n\in \mathbb{N}$, to be interpreted as
the ball of radius $n$ of $A_{\lambda }\subseteq A$ centered at the origin;

\item $\mathcal{S}^{( 0) }$ has domains of quantification $D_{n}^{( 0)
,\lambda }$ for $\lambda \in \mathrm{Rep}( G) $ and $n\in \mathbb{N}$, to be
interpreted as the ball of radius $n$ of $C( G) _{\lambda }\subseteq C( G) $
centered at the origin;

\item $\mathcal{S}^{( 1) }$ has domains of quantification $D_{n}^{( 1)
,\lambda }$ for $\lambda \in \mathrm{Rep}( G) $ and $n\in \mathbb{N}$, to be
interpreted as the ball of radius $n$ of $C( G) _{\lambda }\otimes
A_{\lambda }\subseteq C( G) \otimes A$ centered at the origin;

\item $\mathcal{C}$ has domains of quantification $D_{n}$ of $n\in \mathbb{N}
$, to be interpreted as the balls of radius $n$ of $\mathbb{C}$ centered at
the origin.
\end{itemize}

The function and relation symbols consist of:

\begin{itemize}
\item function and relation symbols for the C*-algebra operations and
C*-algebra norm of $A$, $C( G) $, $C( G) \otimes A$, and $\mathbb{C}$;

\item constant symbols for the elements of $\mathbb{C}$ and for the elements
of $\mathcal{O}( G) $;

\item function symbols for the $\mathcal{O}( G) $-bimodule structure of $C(
G) \otimes A$;

\item for every $\lambda \in \mathrm{Rep}( G) $ and $u\in C( G) _{\lambda }$%
, a function symbol $\mathcal{S}\rightarrow \mathcal{S}^{( 1) }$ for the
canonical map $A\rightarrow C( G) \otimes A$;

\item for every bounded linear map $T\colon C( G) \rightarrow C( G) $ that
maps $\mathcal{O}( G) $ to itself, a function symbol $\mathcal{S}^{( 1)
}\rightarrow \mathcal{S}^{( 1) }$ to be interpreted as $T\otimes \mathrm{id}%
_{A}\colon C( G) \otimes A\rightarrow C( G) \otimes A$ and a function symbol 
$\mathcal{S}^{( 0) }\rightarrow \mathcal{S}^{( 0) }$ to be interpreted as $%
T\colon C( G) \rightarrow C( G) $;

\item for every $\omega \in \mathcal{O}(\hat{G})$, a function symbol $%
\mathcal{S}^{( 1) }\rightarrow \mathcal{S}$ to be interpreted as the map $%
\omega \otimes \mathrm{id}_{A}\colon C( G) \otimes A\rightarrow A$ (slice
maps) and a function symbol $\mathcal{S}^{( 0) }\rightarrow \mathcal{C}$ to
be interpreted as $\omega \colon C( G) \rightarrow \mathbb{C}$;

\item a function symbol $\mathcal{S}\rightarrow \mathcal{S}^{(1)}$ for the
*-homomorphisms $\alpha \colon A\rightarrow C( G) \otimes A$ defining the
action.
\end{itemize}

\subsubsection{Axioms\label{Subsubsection:axioms-compact}}

We now describe axioms for the class of $G$-C*-algebras in the language $%
\mathcal{L}_{G}^{\text{C*}}$ described above. We will use the notation
introduced in Subsection \ref{Subsection:groups-compact} and Subsection \ref%
{Subsection:action-compact}. Particularly, we set $E_{ij}^{\lambda }:=(
\omega _{ij}^{\lambda }\otimes \mathrm{id}) \circ \alpha $ for $\lambda \in 
\mathrm{Rep}( G) $ and $1\leq i,j\leq d_{\lambda }$, and $E^{\lambda
}:=\sum_{i=1}^{d_{\lambda }}E_{ii}^{\lambda }$. We also set $P^{\lambda
}\otimes \mathrm{id}:=\sum_{i=1}^{d_{\lambda }}P_{ii}^{\lambda }\otimes 
\mathrm{id}$ and $P^{\lambda }\otimes E^{\lambda }:=( P^{\lambda }\otimes 
\mathrm{id}) \circ E^{\lambda }$. The axioms are designed to guarantee the
following:

\begin{enumerate}
\item the interpretations of the sort $\mathcal{S} $, $\mathcal{S}^{( 0) }$,
and $\mathcal{S}^{( 1) }$ are C*-algebras;

\item the interpretation of the symbol for $\alpha $ is an isometric
*-homomorphism;

\item the sort $\mathcal{C}$ is interpreted as the complex numbers, and its
domains are interpreted correctly;

\item the interpretation of the domains $D_{n}^{\lambda ,\mathcal{S}}$ is
the ball of radius $n$ of the range of $E^{\lambda }$ centered at the origin
(see also \cite[Example 2.2.1]{farah_model_2016});

\item the interpretation of the domain $D_{n}^{\lambda ,\mathcal{S}}$ is the
ball of radius $n$ of the range of $P^{\lambda }\otimes E^{\lambda }$
centered at the origin;

\item the interpretation of the sort $\mathcal{S}^{( 0) }$ is isomorphic to $%
C( G) $ as a C*-algebra;

\item the interpretation of the sort $\mathcal{S}^{( 1) }$ is isomorphic to $%
C( G) \otimes A$ as a $C( G) $-bimodule;

\item the norm on $\mathcal{S}^{( 1) }$ is the minimal (injective) tensor
product norm (see Lemma \ref{Lemma:norm-compact});

\item the function symbols for $T\colon C( G) \rightarrow C( G) $ and $%
T\otimes \mathrm{id}_{A}\colon C( G) \otimes A\rightarrow C( G) \otimes A$,
where $T$ is a bounded linear map sending $\mathcal{O}( G) $ to itself, are
interpreted correctly;

\item the function symbols for the slice maps $\omega _{ij}^{\lambda
}\otimes \mathrm{id}_{A}\colon C( G) \otimes A\rightarrow A$ are interpreted
correctly;

\item the interpretations of the symbols for $E_{ij}^{\lambda }$ satisfy 
\begin{equation*}
E_{ij}^{\lambda }E_{sn}^{\mu }=\delta _{\lambda \mu }\delta _{is}\delta _{jn}
\end{equation*}%
for $\lambda ,\mu \in \mathrm{Irr}( G) $, $1\leq i,j\leq d_{\lambda }$, and $%
1\leq s,n\leq d_{\mu }$;

\item the interpretations of the symbols for $E^{\lambda } $ satisfy $%
E^{\lambda }=E^{\lambda _{1}}+\cdots +E^{\lambda _{n}}$ whenever $\lambda
\in \mathrm{Rep}( G) $ is the direct sum of irreducible unitary
representations $\lambda _{1}\oplus \cdots \oplus \lambda _{n}$, and the
fixed orthonormal basis on $\mathcal{H}_{\lambda }=\mathcal{H}^{\lambda
_{1}}\oplus \cdots \oplus \mathcal{H}^{\lambda _{n}}$ is obtained from the
fixed orthonormal bases of $\mathcal{H}^{\lambda _{1}},\ldots ,\mathcal{H}%
^{\lambda _{n}}$;

\item the interpretations of the symbols for $E^{\lambda }$ and $P^{\lambda
} $ satisfy 
\begin{equation*}
( P_{ij}^{\lambda }\otimes \mathrm{id}) \circ \alpha =\alpha \circ
E_{ij}^{\lambda }
\end{equation*}%
for every $\lambda \in \mathrm{Irr}( G) $ and $1\leq i,j\leq d_{\lambda }$
(see part~(1) of Lemma \ref{Lemma:action-density-compact});

\item for every $\lambda \in \mathrm{Irr}( G) $, we have 
\begin{equation*}
1\otimes E^{\lambda }=\sum_{1\leq s,t\leq d_{\lambda }}( ( u_{st}^{\lambda
}) ^{\ast }\otimes 1) ( \alpha \circ E_{st}^{\lambda }).
\end{equation*}
\end{enumerate}

Next, we show that the axioms (1)--(14) indeed axiomatize the class of $G$%
-C*-algebras.

\begin{proposition}
If $G$ is a compact quantum group, then an $\mathcal{L}_{G}^{\text{C*}}$%
-structure satisfies (1)--(14) above if and only if it is given by a $G$%
-C*-algebra.
\end{proposition}

\begin{proof}
The discussion in Subsection \ref{Subsection:action-compact} shows that any $%
G$-C*-algebra $( A,\alpha ) $, when regarded as an $\mathcal{L}_{G}^{\text{C*%
}}$-structure, satisfies the axioms (1)--(14) above. We prove the converse,
by showing that (13) implies that the action condition from Definition \ref%
{Definition:compact-action} holds, and that (14) implies that the density
condition from Definition \ref{Definition:compact-action} holds.

For every $\lambda \in \mathrm{Irr}( G) $ and $1\leq i,j\leq d_{\lambda }$,
in view of the axioms from (13) we have%
\begin{align*}
( \omega _{ij}^{\lambda }\otimes \mathrm{id}\otimes \mathrm{id}) \circ ( 
\mathrm{id}\otimes \alpha ) \circ \alpha &=( \mathrm{id}\circ \alpha ) \circ
( \omega _{ij}^{\lambda }\otimes \mathrm{id}\otimes \mathrm{id}) \circ \alpha
\\
&=( \mathrm{id}\circ \alpha ) \circ ( P_{ij}^{\lambda }\circ \mathrm{id})
\circ \alpha \\
&=( \omega _{ij}^{\lambda }\otimes \mathrm{id}\otimes \mathrm{id}) \circ (
\Delta \otimes \mathrm{id}) \circ \alpha \text{.}
\end{align*}%
Since this is true for every $\lambda \in \mathrm{Irr}( G) $ and $1\leq
i,j\leq d_{\lambda }$, we conclude that $( \mathrm{id}\otimes \alpha ) \circ
\alpha =( \Delta \otimes \mathrm{id}) \circ \alpha $, which is exactly the
action condition from Definition \ref{Definition:compact-action}.

For the density condition, let $\lambda \in \mathrm{Irr}( G) $ and $a\in
A_{\lambda }$. Then 
\begin{equation*}
1\otimes a=\sum_{1\leq s,t\leq d_{\lambda }}( ( u_{st}^{\lambda }) ^{\ast
}\otimes 1) ( \alpha \circ E_{st}^{\lambda }) ( a) \in \lbrack ( C( G)
\otimes 1) \alpha (A_{\lambda })]\text{.}
\end{equation*}%
Since this holds for every $\lambda \in \mathrm{Irr}( G) $, we have that $[
( C( G) \otimes 1) \alpha ( A) ] =C( G) \otimes A$, as desired.
\end{proof}

\begin{remark}
One similarly checks that an $\mathcal{L}_{G}^{\text{C*}}$-morphism between $%
G$-C*-algebra is precisely a $G$-equivariant *-homomorphism, while an $%
\mathcal{L}_{G}^{\text{C*}}$-embedding is an injective $G$-equivariant
*-homomorphism.
\end{remark}

It is easy to see that the axioms above are given by conditions of the form $%
\sigma \leq r$, where $\sigma $ is a positive primitive $\forall \exists $-$%
\mathcal{L}_{G}^{\text{C*}}$-sentence and $r\in \mathbb{R}$. Therefore, the
class of $G$-C*-algebras is positively primitively $\forall \exists $%
-axiomatizable in the language $\mathcal{L}_{G}^{\text{C*}}$, in the sense
of Definition \ref{Definition:pp-axiomatizable}. 
%When $G$ is second-countable,
%the language $\mathcal{L}_{G}^{\text{C*}}$ is separable for the class of $G$%
%-C*-algebras in the sense of Definition \ref{Definition:separable-language}.

\subsubsection{Freeness}

We continue to fix a compact quantum group $G$. We recall here the notion of
freeness for a $G$-C*-algebra $A$ from \cite{ellwood_new_2000}.

\begin{definition}
\label{Definition:freeness}A $G$-C*-algebra $(A,\alpha )$ is \emph{free }if $%
[(1\otimes A)\alpha (A)]=C(G)\otimes A$.
\end{definition}

It is proved in \cite[Theorem 2.9]{ellwood_new_2000} that such a definition
recovers the usual notion of freeness in topological dynamics when $G$ is a
classical compact group and $A$ is an abelian C*-algebra.

\begin{example}
\label{eg:free} The left translation action of $G$ on $C(G)$ is free. More
generally, if $A$ is any C*-algebra, then the $G$-action on $C(G)\otimes A$
given by $\Delta \otimes \mathrm{id}_{A}$ is free as well.
\end{example}

The following equivalent reformulation of freeness is an easy consequence of
the fact that $\mathcal{O}( A) $ is a dense *-subalgebra of $A$, the fact
that $\mathcal{O}( G) $ is a dense *-subalgebra of $C( G) $, and part~(1) of
Lemma \ref{Lemma:action-density-compact}.

\begin{lemma}
\label{Lemma:freeeness-definable} A $G$-C*-algebra $(A,\alpha )$ is free if
and only if $[(1\otimes \mathcal{O}(A))\alpha (A_{\lambda })]$ is equal to $%
C(G)_{\lambda }\otimes A$, for every $\lambda \in \mathrm{Irr}(G)$. It
follows that the class of free $G$-C*-algebras is definable by a uniform
family of positive existential $\mathcal{L}_{G}^{\text{C*}}$-formulas.
\end{lemma}

\subsubsection{Ultraproducts\label{Subsubsection:ultraproducts}}

Regarding $G$-C*-algebras as structures in the language $\mathcal{L}_{G}^{%
\text{C*}}$ described above provides a natural notion of \emph{ultraproduct }%
of $G$-C*-algebras, as a particular instance of the notion of ultraproduct
in the logic for metric structures. Since the class of $G$-C*-algebras is
axiomatizable in the language $\mathcal{L}_{G}^{\text{C*}}$, it follows that
the ultraproduct of $G$-C*-algebras is a $G$-C*-algebra.

More generally, one can consider reduced products of $G$-C*-algebras with
respect to an arbitrary filter $\mathcal{F}$ as particular instances of
reduced products of metric structures as defined in \cite%
{ghasemi_reduced_2016}. Concretely, suppose that $\mathcal{F}$ is a filter
on a set $I$, and let $( A_{i},\alpha _{i}) $ be an $I$-sequence of $G$%
-C*-algebras. For $\lambda \in \mathrm{Rep}( G) $, let $A_{\lambda }$ be the
Banach space obtained as the vector space of bounded sequences $( a_{i}) \in
\prod_{i\in I}(A_{i})_{\lambda }$ endowed with the seminorm $\| ( a_{i}) \|
=\limsup_{\mathcal{F}}\| a_{i}\| $. (This is just the reduced product $%
\prod_{\mathcal{F}}(A_{i})_{\lambda }$ of the $I$-sequence of Banach spaces $%
(A_{i})_{\lambda }$.) One can regard $A_{\lambda }$ as a closed subspace of $%
A_{\mu }$ whenever $\lambda ,\mu \in \mathrm{Rep}( G) $ and $\lambda $ is
contained in $\mu $. Therefore, the union $\bigcup_{\lambda\in\mathrm{Rep}%
(G)}A_{\lambda }$ has a natural normed vector space structure. Let $A$
denote its completion. We will write $[ a_{i}] _{\mathcal{F}}$ for the
element of $A$ corresponding to the bounded sequence $( a_{i})_{i\in I} $ in 
$\prod_{i\in I}A_{i}$.

Multiplication and involution on $A$ are induced by pointwise operations (it
is important to notice here that $( a_{i}b_{i}) $ is automatically a bounded
sequence in $\prod_{i\in I}(A_{i})_{\lambda \otimes \mu }$ and $(
a_{i}^{\ast }) $ is a bounded sequence in $\prod_{i\in I}(A_{i})_{\overline{%
\lambda }}$). Finally, since $C( G) _{\lambda }$ is finite-dimensional for
every $\lambda \in \mathrm{Rep}( G) $, one can isometrically identify $C( G)
_{\lambda }\otimes A_{\lambda }$ with the reduced product of Banach spaces $%
\prod_{\mathcal{F}}(C( G) _{\lambda }\otimes (A_{i})_{\lambda })$; see for
instance \cite[Lemma 7.4]{heinrich_ultraproducts_1980}, where the case of
ultraproducts is considered. Therefore, the assignment $( a_{i})_{i\in I}
\mapsto ( \alpha ( a_{i}) )_{i\in I} $, induces a function $\alpha \colon
A_{\lambda }\rightarrow \prod_{\mathcal{F}}(C( G) _{\lambda }\otimes
(A_{i})_{\lambda })=C( G) _{\lambda }\otimes A_{\lambda }$ for every $%
\lambda \in \mathrm{Rep}(G)$. The induced function $\alpha \colon
A\rightarrow C( G) \otimes A$ is easily seen to be an action of $G$ on $A$.
In the following, we will denote such an action by $\alpha _{\mathcal{F}}$.

\begin{remark}
In view of \cite[Proposition 2.13]{barlak_spatial_2017}, considering such an
explicit construction of the reduced product shows that it coincides with
the continuous part of the sequence algebra as defined in \cite[Definition
2.11]{barlak_spatial_2017} when $\mathcal{F}$ is the filter of cofinite
subsets of $\mathbb{N}$ and $G$ is a second countable coexact compact
quantum group (which is the only case considered in \cite[Definition 2.11]%
{barlak_spatial_2017}).
\end{remark}

We now show that, when $G$ is a classical compact group, the reduced product
of $G$-C*-algebras as $\mathcal{L}_{G}^{\text{C*}}$-structures agrees with
the continuous part of the sequence algebra of a $G$-C*-algebra considered
in the C*-algebra literature; see, for example, \cite%
{gardella_equivariant_2016}.

\begin{proposition}
Let $G$ be a compact group. Suppose that $\mathcal{F}$ is a filter on an
index set $I$, and let $\{(A_{i},\alpha ^{(i)})\colon i\in I\}$ be $G$%
-C*-algebras. Denote by $\alpha _{\mathcal{F}}$ the (not necessarily
continuous) action of $G$ on $\prod_{\mathcal{F}}A_{i}$ given by pointwise
application of the actions $\alpha ^{(i)}$. Then the reduced ultraproduct $%
\prod_{\mathcal{F}}^{G}A_{i}$ agrees with the subalgebra of $\prod_{\mathcal{%
F}}A_{i}$ where $\alpha _{\mathcal{F}}$ is continuous.
\end{proposition}

\begin{proof}
For convenience, denote by $B$ the subalgebra of $\prod_{\mathcal{F}}A_{i}$
where $G$ acts continuously. Regard $\prod_{\mathcal{F}}^{G}A_{i}$ as a
C*-subalgebra of $\prod_{\mathcal{F}}A_{i}$, in such a way that the
inclusion $\prod_{\mathcal{F}}^{G}A_{i}\subseteq \prod_{\mathcal{F}}A_{i}$
is $G$-equivariant. Since the action of $G$ on $\prod_{\mathcal{F}}^{G}A_{i}$
is continuous, one has $\prod_{\mathcal{F}}^{G}A_{i}\subseteq B$. On the
other hand, since the action of $G$ on $B$ is continuous, the *-subalgebra $%
\mathcal{O}( B) =\bigcup_{\lambda \in \mathrm{Rep}( G) }B_{\lambda }$
consisting of the spectral subspaces for $B$, is dense in $B$. For every $%
\lambda \in \mathrm{Rep}( G) $, and using the projections onto the spectral
subspaces, one can see directly that 
\begin{equation*}
B_{\lambda }=\prod\nolimits_{\mathcal{F}}(A_{i})_{\lambda }\subseteq
\prod\nolimits_{\mathcal{F}}^{G}A_{i}.
\end{equation*}%
Since this holds for every $\lambda \in \mathrm{Rep}( G) $, we have $%
\mathcal{O}( B) \subseteq \prod\nolimits_{\mathcal{F}}^{G}A_{i}$ and hence $%
B\subseteq \prod\nolimits_{\mathcal{F}}^{G}A_{i}$.
\end{proof}

In particular, it follows from these observations that the (positive) \emph{%
existential theory }of a $G$-C*-algebras as defined in \cite[Subsection 2.2]%
{gardella_equivariant_2016} coincides with the (positive) existential theory
of a $G$-C*-algebras as an $\mathcal{L}_{G}^{\text{C*}}$-structure.

\subsubsection{Other languages}

As for the case of discrete quantum group actions, it is sometimes useful to
consider C*-algebras and $G$-C*-algebras as structures in a language that is
different from the standard language for C*-algebras. These other languages
allow one to capture properties that are preserved under not necessarily
multiplicative maps. Several such languages are considered in \cite[Section 3%
]{gardella_equivariant_2016}. If $\mathcal{L}$ is such a language, one can
define its corresponding $G$-equivariant analogue $\mathcal{L}_{G}$ as
above, so that $G$-C*-algebras can also be regarded as $\mathcal{L}_{G}$%
-structures. %For
%instance, one can considered the ordered selfadjoint operator space language
%$\mathcal{L}^{\mathrm{osos}}$, and its corresponding $G$-equivariant version
%$\mathcal{L}_{G}^{\mathrm{osos}}$. The $\mathcal{L}_{G}^{\mathrm{osos}}$%
%-morphisms between $G$-C*-algebras are then precisely the $G$-equivariant
%completely positive contractive linear maps. Similar conclusions apply to
%the other languages considered in \cite[Section 3]{gardella_equivariant_2016}%
%.

\section{Crossed products and reduced products\label{Section:crossed}}

\subsection{The dual of a compact quantum group}

The \emph{dual} $\hat{G}$ of a compact quantum group $G$ is the discrete
quantum group defined as follows. By assumption, the GNS representation
associated with the Haar state $h$ on $C( G) $ defines a faithful
representation of $C( G) $. We denote by $L^{2}( G) $ the corresponding
Hilbert space, and we identify $C( G) $ with a subalgebra of $B( L^{2}( G) ) 
$. We let $x\mapsto \left\vert x\right\rangle $ be the canonical map from $%
C( G) $ to $L^{2}( G) $, so that $x\left\vert y\right\rangle =\left\vert
xy\right\rangle $, and $\left\langle x,y\right\rangle =h( x^{\ast }y) $ for
all $x,y\in C( G) $. For $\lambda \in \mathrm{Rep}( G) $, set $%
L^{2}(G)_{\lambda }=\left\{ u_{\xi ,\eta }^{\lambda }\colon \xi ,\eta \in 
\mathcal{H}_{\lambda }\right\} \subseteq L^{2}( G) $.

Recall that $\mathcal{O}(\hat{G})$ denotes the space of linear functionals
on $C( G) $ of the form $x\mapsto h( xa) $ for some $a\in \mathcal{O}( G) $.
It coincides with the space of linear functionals on $C( G) $ of the form $%
x\mapsto h( ax) $ for some $a\in \mathcal{O}( G) $, and it is also equal to
the linear span of $\omega _{ij}^{\lambda }$ for $\lambda \in \mathrm{Irr}(
G) $ and $1\leq i,j\leq d_{\lambda }$. This is an algebra with respect to
convolution. The antipode map $S$ of $\mathcal{O}( G) $ defines an
involution on $\mathcal{O}(\hat{G})$ by setting $\phi ^{\ast }=\overline{%
\phi ( S( x) ^{\ast }) }$. For $\lambda \in \mathrm{Rep}( G) $, set 
\begin{equation*}
c_{0}(\hat{G})_{\lambda }=\mathrm{span}\left\{ \omega _{\xi ,\eta }^{\lambda
}\colon \xi ,\eta \in \mathcal{H}_{\lambda }\right\} \text{,}
\end{equation*}%
and observe that $\omega _{\xi ,\eta }^{\lambda }\mapsto T_{\xi ,\eta }$ is
a *-homomorphism from $c_{0}(\hat{G})_{\lambda }$ onto $\mathbb{K}( \mathcal{%
H}_{\lambda }) $, where $T_{\xi ,\eta }$ is the usual rank-one operator.
Furthermore $\mathcal{O}(\hat{G})$ is isomorphic to the algebraic direct sum
of $c_{0}(\hat{G})_{\lambda }$ for $\lambda \in \mathrm{Irr}( G) $. One can
define an injective *-representation of $\mathcal{O}(\hat{G})$ on $L^{2}( G) 
$ by setting $\omega \left\vert x\right\rangle =\left\vert ( \mathrm{id}%
\otimes \omega ) \Delta ( x) \right\rangle $ for $x\in C( G) $ and $\omega
\in \mathcal{O}(\hat{G})$. We will identify $\mathcal{O}(\hat{G})$ with its
image inside $B( L^{2}( G) ) $, and let $C^{\ast }( G) =c_{0}(\hat{G})$ be
the closure of $\mathcal{O}(\hat{G})$ inside $B(L^{2}( G)) $. For $\lambda
\in \mathrm{Rep}( G) $, the subspace $L^{2}(G)_{\lambda }$ is invariant
under $c_{0}(\hat{G})_{\lambda }$, and the inclusion $c_{0}(\hat{G}%
)_{\lambda }\subseteq B( L^{2}(G)_{\lambda }) $ is isometric. The C*-algebra 
$C^{\ast }( G) =c_{0}(\hat{G})$ is also called the \emph{group C*-algebra }%
of the compact quantum group $G$.

The multiplication operation on $C(G)$ defines a nondegenerate
*-homomorphism $\hat{\Delta}\colon \mathcal{O}(\hat{G})\rightarrow M(%
\mathcal{O}(\hat{G})\odot \mathcal{O}(\hat{G}))$ such that $(\omega
_{1}\otimes 1)\hat{\Delta}(\omega _{2})$ belongs to $\mathcal{O}(\hat{G}%
)\odot \mathcal{O}(\hat{G})$ for every $\omega _{1},\omega _{2}\in \mathcal{O%
}(\hat{G})$, and%
\begin{eqnarray*}
(\omega _{1}\otimes 1)\hat{\Delta}(\omega _{2}) &:&\ x\otimes y\mapsto
(\omega _{1}\otimes \omega _{2})(\Delta (x)(1\otimes y)) \\
\hat{\Delta}(\omega _{2})(1\otimes \omega _{1}) &:&x\otimes y\mapsto (\omega
_{2}\otimes \omega _{1})((x\otimes 1)\Delta (y))\text{.}
\end{eqnarray*}%
Such a nondegenerate *-homomorphism extends to $c_{0}(\hat{G})$ and defines
a *-homomorphism $\hat{\Delta}\colon c_{0}(\hat{G})\rightarrow
M(c_{0}(\hat{G})\otimes c_{0}(\hat{G}))$ which is also nondegenerate. This defines a discrete quantum
group $\hat{G}$, which is the dual of $G$.

For the purpose of defining the dual of a given action, it is convenient to
consider the \emph{opposite} discrete group $\check{G}$ of $\hat{G}$. 
%In the following definition, we denote by $\Sigma \in B(\mathcal{H}\otimes \mathcal{%H})$ the flip unitary.

\begin{definition}
\label{def:Gcheck} Let $G$ be a compact quantum group. Following the
notation of \cite{barlak_spatial_2017}, we define $\check{G}$ to be the
discrete group such that $c_{0}(\check{G})$ is equal to $c_{0}(\hat{G})$ as
a C*-algebra, but endowed with the \emph{opposite comultiplication} $\check{%
\Delta}(\omega)=\Sigma \circ \hat{\Delta}(\omega)\circ \Sigma $.
\end{definition}

The Hopf *-algebra $\mathcal{O}(\check{G})$ is equal to $\mathcal{O}(\hat{G}%
) $ but endowed with the opposite comultiplication.

%Given a discrete quantum group $G$, one can similarly define its dual $\hat{G%
%}$, which is a compact quantum group. These constructions are inverse of each
%other, in the sense that the double dual $\hat{\hat{G}}$ is naturally
%isomorphic to the original quantum group $G$.

\subsection{The crossed product of a compact quantum group action}

We fix a compact quantum group $G$ and a $G$-C*-algebra $(A,\alpha )$. We
proceed to define the associated (reduced) crossed product.

\begin{definition}
(See \cite[Definition 5.28]{de_commer_actions_2016} and \cite[Proposition
5.32]{de_commer_actions_2016}). Fix a nondegenerate faithful
*-re\-pre\-sen\-ta\-tion $A\rightarrow B(\mathcal{H})$, under which we regard $%
\alpha (A)\subseteq C(G)\otimes A$ as a subalgebra of $B(L^{2}(G)\otimes 
\mathcal{H})$ and similarly for $c_{0}(\hat{G})\otimes 1$. The \emph{reduced
crossed product }$G\ltimes _{\alpha ,\mathrm{r}}A$ is defined as 
\begin{equation*}
\lbrack (c_{0}(\hat{G})\otimes 1)\alpha (A)]\subseteq B(L^{2}(G)\otimes 
\mathcal{H}).
\end{equation*}%
The crossed product is canonically endowed with an action $\check{\alpha}$
of $\check{G}$, called the \emph{dual action}, which is defined by setting%
\begin{equation*}
\check{\alpha}((\omega \otimes 1)\alpha (a))=(\check{\Delta}(\omega )\otimes
1)(1\otimes \alpha (a))\in M(c_{0}(\check{G})\otimes G\ltimes _{\alpha
}A)\subseteq B(L^{2}(G)\otimes L^{2}(G)\otimes \mathcal{H}).
\end{equation*}
\end{definition}

We will regard $G\ltimes _{\alpha ,\mathrm{r}}A$ as a $\check{G}$-C*-algebra
endowed with such an action of $\check{G}$. Given $\omega \in c_{0}(\check{G}%
)$ and $a\in A$, we denote by $\omega \ltimes a$ the element $(\omega
\otimes 1)\alpha (a)$ of $G\ltimes _{\alpha ,\mathrm{r}}A$. It is shown in 
\cite[Theorem 5.31]{de_commer_actions_2016} that the reduced crossed product 
$G\ltimes _{\alpha ,\mathrm{r}}A$ as defined above coincides with the \emph{%
full }crossed product in the sense of \cite[Definition 5.27]%
{de_commer_actions_2016}, and therefore can also be denoted by $G\ltimes
_{\alpha }A$. It follows from the universal property of the full crossed
products that, if $(B,\beta )$ and $(C,\gamma )$ are $G$-C*-algebras, then a 
$G$-equivariant *-homomorphism $\phi \colon B\rightarrow C$ induces a $%
\check{G}$-equivariant *-homomorphism $G\ltimes \phi \colon G\ltimes _{\beta
,\mathrm{r}}B\rightarrow G\ltimes _{\gamma ,\mathrm{r}}C$ by setting $%
(G\ltimes \phi )(\omega \ltimes a)=\omega \ltimes \phi (a)$. Furthermore, if 
$\phi $ is nondegenerate, then $G\ltimes \phi $ is nondegenerate \cite[%
Theorem 9.4.8]{timmermann_invitation_2008}. Recall that a completely
positive order zero map $\phi :A\rightarrow B$ between C*-algebra is \emph{%
order zero }if, whenever $a,b$ are positive elements of $A$ satisfying $%
ab=ba=0$, then $\phi \left( a\right) \phi \left( b\right) =\phi \left(
b\right) \phi \left( a\right) =0$ \cite[Definition 2.3]%
{winter_completely_2009}. One can then easily deduce from this and the
structure theorem for completely positive order zero maps the following
lemma; see also \cite[Proposition 2.3]{gardella_regularity_?}.

\begin{lemma}
\label{Lemma:crossed-oz}Suppose that $( A,\alpha ) $, $( B,\beta ) $, $(
C,\gamma ) $ are $G$-C*-algebras. Assume that $A\subseteq B$ and $A\subseteq
C$ are $G$-invariant subalgebras. Let $\theta \colon B\rightarrow C$ be a $G$%
-equivariant completely positive order zero $A$-bimodule map. Then the
assignment%
\begin{equation*}
\omega \ltimes b\mapsto \omega \ltimes \theta ( b)
\end{equation*}%
defines a completely positive $\check{G}$-equivariant order zero $(G\ltimes
_{\alpha }A)$-bimodule map $G\ltimes \theta \colon G\ltimes _{\beta
}B\rightarrow G\ltimes _{\gamma }C$. If furthermore $\theta $ is a
(nondegenerate) *-homomorphism, then $G\ltimes \theta $ is a (nondegenerate)
*-homomorphism.
\end{lemma}

Suppose now that $\mathcal{F}$ is a filter on some index set $I$. For every $%
\ell \in I$, let $( A_{\ell },\alpha _{\ell }) $ be a $G$-C*-algebra. Fix
also, for every $\ell \in I$, a nondegenerate faithful *-representation $%
A_{\ell }\to B( \mathcal{H}_{\ell }) $. Consider the corresponding reduced
product of $G$-C*-algebras $\prod_{\mathcal{F}}^{G}A_{\ell }\subseteq
\prod\nolimits_{\mathcal{F}}B( \mathcal{H}_{\ell }) \subseteq B(
\prod\nolimits_{\mathcal{F}}\mathcal{H}_{\ell }) $, which is a $G$%
-C*-algebra endowed with the action $\alpha _{\mathcal{F}}$. Then the
crossed product $G\ltimes _{\alpha _{\mathcal{F}}}\prod\nolimits_{\mathcal{F}%
}^{G}A_{\ell }$ can be naturally represented on $L^{2}( G) \otimes
\prod\nolimits_{\mathcal{F}}\mathcal{H}_{\ell }$. On the other hand, one can
also consider, for every $\ell \in I$, the crossed product $G\ltimes
_{\alpha _{\ell }}A_{\ell }\subseteq B( L^{2}( G) \otimes \mathcal{H}) $,
and then the reduced product of $\check{G}$-C*-algebras 
\begin{equation*}
\prod\nolimits_{\mathcal{F}}^{\check{G}}( G\ltimes _{\alpha _{\ell }}A_{\ell
}) \subseteq \prod\nolimits_{\mathcal{F}}B( L^{2}( G) \otimes \mathcal{H}%
_{\ell }) \subseteq B( \prod\nolimits_{\mathcal{F}}( L^{2}( G) \otimes 
\mathcal{H}_{\ell }) ) \text{.}
\end{equation*}

These algebras do not coincide in general, but there is always a canonical
map in one direction, as we show next.

\begin{proposition}
\label{Proposition:ultra-crossed-compact} Let the notation be as in the
discussion above. Then the assignment $\theta \ltimes \lbrack a^{(\ell )}]_{%
\mathcal{F}}\mapsto \lbrack \theta \ltimes a^{(\ell )}]_{\mathcal{F}}$, for $%
\lambda \in \mathrm{Rep}( G) $, for $\theta \in c_{0}(\check{G})_{\lambda }$%
, and $a\in ( \prod\nolimits_{\mathcal{F}}^{G}A_{\ell }) ^{\lambda }$,
determines an injective $\check{G}$-equivariant *-homomorphism 
\begin{equation*}
G\ltimes _{\alpha _{\mathcal{F}}}\prod\nolimits_{\mathcal{F}}^{G}A_{\ell
}\rightarrow \prod\nolimits_{\mathcal{F}}^{\check{G}}( G\ltimes _{\alpha
_{\ell }}A_{\ell }) \text{.}
\end{equation*}
\end{proposition}

\begin{proof}
To simplify the notation, we drop the subscript $\mathcal{F}$ when denoting
elements of a reduced product via their representative sequences. Suppose
that 
\begin{equation*}
a_{ij}=[a_{ij}^{(\ell )}]\in ( \prod\nolimits_{\mathcal{F}}^{G}A_{\ell })
_{\lambda }
\end{equation*}%
for $1\leq i,j\leq d_{\lambda }$, and set%
\begin{equation*}
z=\sum_{ij}\omega _{ij}^{\lambda }\ltimes \lbrack a_{ij}^{(\ell )}]\in
G\ltimes _{\alpha _{\mathcal{F}}}\prod\nolimits_{\mathcal{F}}^{G}A_{\ell
}\subseteq B( L^{2}( G) \otimes \prod\nolimits_{\mathcal{F}}\mathcal{H}) 
\text{.}
\end{equation*}%
Let $p_{\lambda }\in B( L^{2}( G) ) $ be the orthogonal projection onto $%
L^{2}(G)_{\lambda }$. Then $z( p_{\lambda }\otimes 1) =( p_{\lambda }\otimes
1) z( p_{\lambda }\otimes 1)$ and $\| z\| =\| ( p_{\lambda }\otimes 1)
zp_{\lambda }\|$.

For $\ell \in I$, set 
\begin{equation*}
z^{(\ell )}=\sum_{ij}\omega _{ij}^{\lambda }\ltimes a_{ij}^{(\ell )}\in
G\ltimes _{\alpha }A\subseteq B(L^{2}(G)\otimes \mathcal{H})\text{,}
\end{equation*}%
and 
\begin{equation*}
\lbrack z^{(\ell )}]\in \prod\nolimits_{\mathcal{U}}^{\check{G}}(G\ltimes
_{\alpha _{\ell }}A_{\ell })\subseteq \prod\nolimits_{\mathcal{F}%
}B(L^{2}(G)\otimes \mathcal{H})\subseteq B(\prod\nolimits_{\mathcal{F}%
}(L^{2}(G)\otimes \mathcal{H}))\text{.}
\end{equation*}%
Then $[z^{(\ell )}]_{\mathcal{F}}(p_{\lambda }\otimes 1)=(p_{\lambda
}\otimes 1)[z^{(\ell )}](p_{\lambda }\otimes 1)$ and $\Vert \lbrack z^{(\ell
)}]\Vert =\Vert (p_{\lambda }\otimes 1)[z^{(\ell )}](p_{\lambda }\otimes
1)\Vert $. Since $p_{\lambda }$ is a finite-rank projection with range $%
L^{2}(G)_{\lambda }$, we have%
\begin{equation*}
(p_{\lambda }\otimes 1)[z^{(\ell )}](p_{\lambda }\otimes 1)\in (p_{\lambda
}\otimes 1)B(\prod\nolimits_{\mathcal{F}}(L^{2}(G)\otimes \mathcal{H}%
))(p_{\lambda }\otimes 1)=B(L^{2}(G)_{\lambda }\otimes \prod\nolimits_{%
\mathcal{F}}\mathcal{H})\text{.}
\end{equation*}%
We conclude that%
\begin{align*}
\Vert \lbrack z^{(\ell )}]\Vert & =\Vert (p_{\lambda }\otimes 1)[z^{(\ell
)}](p_{\lambda }\otimes 1)\Vert =\limsup_{\mathcal{F}}\Vert (p_{\lambda
}\otimes 1)z^{(\ell )}(p_{\lambda }\otimes 1)\Vert \\
& =\Vert (p_{\lambda }\otimes 1)z(p_{\lambda }\otimes 1)\Vert =\Vert z\Vert 
\text{.}
\end{align*}%
This shows that the assignment $\theta \ltimes \lbrack a^{(\ell )}]\mapsto
\lbrack \theta \ltimes a^{(\ell )}]$ yields a well-defined isometric linear
map 
\begin{equation*}
\Phi \colon G\ltimes _{\alpha _{\mathcal{F}}}\prod\nolimits_{\mathcal{F}%
}^{G}A_{\ell }\rightarrow \prod\nolimits_{\mathcal{F}}^{\check{G}}(G\ltimes
_{\alpha }A_{\ell })\text{.}
\end{equation*}%
The fact that $\Phi $ is a $\check{G}$-equivariant *-homomorphism can be
verified directly by means of the expression for the multiplication and
involution in the crossed product, together with the definition of the dual
action. This finishes the proof.
\end{proof}

Suppose now that $\alpha $ is a continuous action of $\check{G}$ on a
C*-algebra $A$. Fix a nondegenerate faithful *-representation $A\to B( 
\mathcal{H}) $. One defines the \emph{reduced crossed product }$\check{G}%
\ltimes _{\alpha ,\mathrm{r}}A$ to be%
\begin{equation*}
\check{G}\ltimes _{\alpha ,\mathrm{r}}A=\lbrack (C(G)\otimes 1)\alpha ( A)
]\subseteq B( L^{2}( G) \otimes \mathcal{H}) \text{,}
\end{equation*}%
As before, we denote the element $( x\otimes 1) \alpha ( a) $ of $\check{G}%
\ltimes _{\alpha ,\mathrm{r}}A$ by $x\ltimes a $. The reduced crossed
product is endowed with a canonical action of $G$ (\emph{dual action})
defined by the *-homomorphism $\check{\alpha}\colon \check{G}\ltimes
_{\alpha ,\mathrm{r}}A\rightarrow C(G)\otimes (\check{G}\ltimes _{\alpha ,%
\mathrm{r}}A)$, $x\ltimes a\mapsto ( \Delta ( x) \otimes 1) ( 1\otimes a) $.
The reduced crossed product construction is functorial, and a
(nondegenerate) $\check{G}$-equivariant *-homomorphism $\phi $ induces a
(nondegenerate) $G$-equivariant *-homomorphism $\check{G}\ltimes \phi $
between the reduced crossed products. This allows one to prove the analogue
of Lemma \ref{Lemma:crossed-oz} in this context. It is clear that, for $%
\lambda \in \mathrm{Rep}(G)$, the corresponding spectral subspace of $\check{%
G}\ltimes _{\alpha ,\mathrm{r}}A$ is just%
\begin{equation*}
(\check{G}\ltimes _{\alpha ,\mathrm{r}}A)^{\lambda }=\mathrm{span}\{\omega
\ltimes a\colon \omega \in C(G)_{\lambda },a\in A\}\text{.}
\end{equation*}

\begin{lemma}
\label{Lemma:ultraproduct-discrete}Suppose that $G$ is a compact quantum
group, $\mathcal{F}$ is a filter on a set $I$, and $( A_{\ell },\alpha
_{\ell }) $ is a $\check{G}$-C*-algebra for every $\ell \in I$. The
assignment of 
\begin{equation*}
u\ltimes \lbrack a^{(\ell )}]_{\mathcal{F}}\mapsto \lbrack u\ltimes a^{(\ell
)}]_{\mathcal{F}}
\end{equation*}%
for $\lambda \in \mathrm{Rep}( G) $, $\theta \in C(G)_{\lambda }$, and $a\in
\prod\nolimits_{\mathcal{F}}^{\check{G}}A_{\ell }$ defines an injective $G$%
-equivariant *-homomorphism 
\begin{equation*}
\check{G}\ltimes _{\alpha _{\mathcal{F},\mathrm{r}}}\prod\nolimits_{\mathcal{%
F}}^{\check{G}}A_{\ell }\rightarrow \prod\nolimits_{\mathcal{F}}^{G}(\check{G%
}\ltimes _{\alpha _{\ell },\mathrm{r}}A_{\ell })\text{.}
\end{equation*}
\end{lemma}

\begin{proof}
This is easy to see directly, using the fact that $C( G) _{\lambda }$ is
finite-dimensional for every $\lambda \in \mathrm{Rep}( G) $.
\end{proof}

\begin{remark}
\label{Remark:ultraproduct-discrete}Consider the particular case of Lemma %
\ref{Lemma:ultraproduct-discrete} when $( A_{\ell },\alpha _{\ell }) $ is
equal to a fixed $\check{G}$-C*-algebra $( A,\alpha ) $ for every $\ell \in
I $. Then the $G$-equivariant *-isomorphism $\theta \colon \check{G}\ltimes
_{\alpha _{\mathcal{F},\mathrm{r}}}\prod\nolimits_{\mathcal{F}}^{\check{G}%
}A\rightarrow \prod\nolimits_{\mathcal{F}}^{G}(\check{G}\ltimes _{\alpha ,%
\mathrm{r}}A)$ has the property that 
\begin{equation*}
\theta \circ (\check{G}\ltimes \Delta _{A})=\Delta _{\check{G}\ltimes
_{\alpha ,\mathrm{r}}A}\text{,}
\end{equation*}%
where $\Delta _{A}\colon A\rightarrow \prod_{\mathcal{F}}^{\check{G}}A$ and $%
\Delta _{\check{G}\ltimes _{\alpha ,\mathrm{r}}A}\colon \check{G}\ltimes
_{\alpha ,\mathrm{r}}A\rightarrow \prod_{\mathcal{F}}^{G}(\check{G}\ltimes
_{\alpha ,\mathrm{r}}A)$ are the diagonal embeddings.
\end{remark}

\subsection{Stabilizations\label{Subsection:stabilizations}}

We continue to fix a compact quantum group $G$. Let $( A,\alpha ) $ is a $G$%
-C*-algebra. Fix a nondegenerate faithful representation $A\to B( \mathcal{H}%
) $. Let $V\in M(c_{0}(\hat{G})\otimes C( G) )\subseteq B( L^{2}( G) \otimes
L^{2}( G) ) $ denote the multiplicate unitary associated with $G$, and set $%
X=\Sigma \circ V\circ \Sigma \in M(C( G) \otimes c_{0}(\hat{G}))\subseteq B(
L^{2}( G) \otimes L^{2}( G) ) $.

\begin{definition}
(\cite[Section 2]{nest_equivariant_2010}). Adopt the notation from the
discussion above. The \emph{stabilization }of $(A,\alpha )$ is the $G$%
-C*-algebra $(\mathbb{K}_{G}\otimes A,\alpha _{\mathbb{K}})$, where $\alpha
_{\mathbb{K}}\colon \mathbb{K}_{G}\otimes A\rightarrow C(G)\otimes \mathbb{K}%
_{G}\otimes A$ is the *-homomorphism given by%
\begin{equation*}
\alpha _{\mathbb{K}}(T\otimes a)=X_{12}^{\ast }(1\otimes T\otimes 1)\alpha
(a)_{13}X_{12}\in C(G)\otimes \mathbb{K}_{G}\otimes A\subseteq
B(L^{2}(G)\otimes L^{2}(G)\otimes \mathcal{H})\text{.}
\end{equation*}

In the following, given a $G$-C*-algebra $( A,\alpha ) $, we regard $\mathbb{%
K}_{G}\otimes A$ as a $G$-C*-algebra with respect to $\alpha _{\mathbb{K}}$.
\end{definition}

When $G$ is a classical compact group, $\alpha_{\mathbb{K}}$ is just the
diagonal action, where $\mathbb{K}_{G}$ is endowed with the action of $G$ by
conjugation induced by the left regular representation.

The assignment $( A,\alpha ) \mapsto ( \mathbb{K}_{G}\otimes A,\alpha _{%
\mathbb{K}}) $ is functorial, in the sense that any *-homomorphism $\phi
\colon ( A,\alpha ) \rightarrow ( B,\beta ) $ induces a *-homomorphism 
\textrm{id}$\otimes \phi \colon ( \mathbb{K}_{G}\otimes A,\alpha _{\mathbb{K}%
}) \rightarrow ( \mathbb{K}_{G}\otimes B,\beta _{\mathbb{K}}) $.

\begin{proposition}
\label{Proposition:stabilization} Let $\mathcal{F}$ be a filter on a set $I$%
, and let $( A_{\ell },\alpha ^{(\ell) })_{\ell\in I} $ be $G$-C*-algebras.
Then the assignment $T\otimes \lbrack a^{(\ell )}]_{\mathcal{F}}\mapsto
\lbrack T\otimes a^{(\ell )}]_{\mathcal{F}}$ defines a $G$-equivariant
*-homomorphism%
\begin{equation*}
\mathbb{K}_{G}\otimes \prod\nolimits_{\mathcal{F}}^{G}A_{\ell }\rightarrow
\prod\nolimits_{\mathcal{F}}^{G}( \mathbb{K}_{G}\otimes A_{\ell }) \text{.}
\end{equation*}
\end{proposition}

\begin{proof}
Observe that tensor product $\mathbb{K}_{G}\otimes \prod\nolimits_{\mathcal{F%
}}^{G}A_{\ell }$ coincides with the maximal tensor product. Therefore the
assignment in the statement gives a well-defined *-homomorphism $\psi \colon 
\mathbb{K}_{G}\otimes \prod\nolimits_{\mathcal{F}}^{G}A_{\ell }\rightarrow
\prod\nolimits_{\mathcal{F}}^{G}( \mathbb{K}_{G}\otimes A_{\ell }) $ in view
of the universal property of the maximal tensor product. A straightforward
computation shows that such a map is $G$-equivariant when $\mathbb{K}%
_{G}\otimes \prod\nolimits_{\mathcal{F}}^{G}A_{\ell }$ is endowed with the
stabilization of the reduced product action $\alpha _{\mathcal{F}}$, and $%
\prod\nolimits_{\mathcal{F}}^{G}( \mathbb{K}_{G}\otimes A_{\ell }) $ is
endowed with the reduced product of the stabilizations $(\alpha _{\mathcal{F}%
}^{(\ell )})_{\ell \in I}$.
\end{proof}

Denote by $1\in C( G) $ the unit, and let $\left\vert 1\right\rangle \in
L^{2}( G) $ be the corresponding vector, and $\left\vert 1\right\rangle
\left\langle 1\right\vert \in \mathbb{K}_{G}$ the corresponding rank one
projection.

\begin{lemma}
The C*-subalgebra $\left\vert 1\right\rangle \left\langle 1\right\vert
\otimes A\subseteq \mathbb{K}_{G}\otimes A$ is a $G$-invariant C*-subalgebra
of $( \mathbb{K}_{G}\otimes A,\alpha _{\mathbb{K}}) $, and the injective
*-homomorphism $a\mapsto \left\vert 1\right\rangle \left\langle 1\right\vert
\otimes a$ is $G$-equivariant.
\end{lemma}

\begin{proof}
The invariance property of the Haar state $h$ can be written as $%
h(y)1=\sum_{i}h(y_{0,i})y_{1,i}=\sum_{i}h(y_{1,i})y_{0,i}$ where $\Delta
(y)=\sum_{i}y_{0,i}\otimes y_{1,i}$. We claim that $(1\otimes \left\vert
1\right\rangle \left\langle 1\right\vert )X=1\otimes \left\vert
1\right\rangle \left\langle 1\right\vert $. Indeed for every $x,y\in C(G)$
we have that%
\begin{align*}
(1\otimes \left\vert 1\right\rangle \left\langle 1\right\vert )X\left\vert
x\otimes y\right\rangle & =(1\otimes \left\vert 1\right\rangle \left\langle
1\right\vert )\Sigma V\left\vert y\otimes x\right\rangle =(1\otimes
\left\vert 1\right\rangle \left\langle 1\right\vert )\Sigma \left\vert
\Delta (y)(1\otimes x)\right\rangle \\
& =(1\otimes \left\vert 1\right\rangle \left\langle 1\right\vert )\left\vert
\Delta ^{\mathrm{op}}(y)(x\otimes 1)\right\rangle =\sum_{i}(1\otimes
\left\vert 1\right\rangle \left\langle 1\right\vert )\left\vert
y_{1,i}x\otimes y_{0,i}\right\rangle \\
& =\sum_{i}h(y_{0,i})\left\vert y_{1,i}x\otimes 1\right\rangle
=h(y)\left\vert x\otimes 1\right\rangle =(1\otimes \left\vert 1\right\rangle
\left\langle 1\right\vert )\left\vert x\otimes y\right\rangle \text{.}
\end{align*}%
Henceforth 
\begin{align*}
\alpha _{\mathbb{K}}(\left\vert 1\right\rangle \left\langle 1\right\vert
\otimes a)& =X_{12}(1\otimes \left\vert 1\right\rangle \left\langle
1\right\vert \otimes 1)\alpha (a)_{13}X_{12}=X_{12}(1\otimes \left\vert
1\right\rangle \left\langle 1\right\vert \otimes 1)\alpha (a)_{13}(1\otimes
\left\vert 1\right\rangle \left\langle 1\right\vert \otimes 1)X_{12} \\
& =(1\otimes \left\vert 1\right\rangle \left\langle 1\right\vert \otimes
1)\alpha (a)_{13}(1\otimes \left\vert 1\right\rangle \left\langle
1\right\vert \otimes 1)=(1\otimes \left\vert 1\right\rangle \left\langle
1\right\vert \otimes 1)\alpha (a)_{13}\text{.}
\end{align*}%
This concludes the proof.
\end{proof}

In the following lemma, we consider the ordered selfadjoint operator space
language $\mathcal{L}^{\mathrm{osos}}$ as introduced in \cite[Subsection 3.1]%
{gardella_equivariant_2016}, as well as its $A$-bimodule version $\mathcal{L}%
^{\mathrm{osos}\text{,}A\text{-}A}$ as introduced in \cite[Subsection 3.5]%
{gardella_equivariant_2016}. Recall that the language $\mathcal{L}^{\mathrm{%
osos}\text{,}A\text{-}A}$ does not have a distinguished relation symbol for
the metric. Instead, for every finite subset $F$ of $A$, the language $%
\mathcal{L}^{\mathrm{osos}\text{,}A\text{-}A}$ contains a distinguished
pseudometric symbol $d_{F}$, to be interpreted in $\mathbb{K}_{G}\otimes A$
as the pseudometric%
\begin{equation*}
d_{F}^{A}(x,y)=\sup \left\{ \Vert (1\otimes a)(x-y)\Vert ,\Vert
(x-y)(1\otimes a)\Vert \colon a\in F\right\} \text{.}
\end{equation*}%
Given a compact or discrete quantum group $G$, one can add symbols for the $%
G $-action to obtain languages $\mathcal{L}_{G}^{\mathrm{osos}}$ and $%
\mathcal{L}_{G}^{\mathrm{osos}\text{,}A\text{-}A}$. If $A$ is a $G$%
-C*-algebra, then $\mathbb{K}_{G}\otimes A$ can be regarded as structure in
the language of $\mathcal{L}_{G}^{\text{C*,}\mathbb{K}_{G}\otimes A\text{-}%
\mathbb{K}_{G}\otimes A}$.

Fix an increasing approximate unit $(u_{j})_{j\in J}$ for $A$ contained in $%
A^{\alpha }$. Suppose now that $B$ is a $G$-C*-algebra containing $A$ as a
nondegenerate $G$-C*-subalgebra. In particular the approximate unit $%
(u_{j})_{j\in J}$ for $A$ is also an approximate unit for $B$. The notion of
positively quantifier-free definable substructure in a language $\mathcal{L}$
is recalled in Definition \ref{Definition:definable-substructure}.

\begin{lemma}
\label{Lemma:definable-stabilization} Adopt the notation and assumptions
from the discussion above. The $G$-C*-subalgebra $\left\vert 1\right\rangle
\left\langle 1\right\vert \otimes B$ of $\mathbb{K}_{G}\otimes B$ is a
positively quantifier-free $\mathcal{L}_{G}^{\mathrm{osos}\text{,}\mathbb{K}%
_{G}\otimes A\text{-}\mathbb{K}_{G}\otimes A}$-definable substructure
relative to the class of $G$-C*-algebras of the form $\mathbb{K}_{G}\otimes
B $.
\end{lemma}

\begin{proof}
We have already observed that $\left\vert 1\right\rangle \left\langle
1\right\vert \otimes B$ is indeed a $G$-C*-subalgebra of $\mathbb{K}%
_{G}\otimes B$. The fact that it is positively quantifier-free $\mathcal{L}%
_{G}^{\mathrm{osos}\text{,}\mathbb{K}_{G}\otimes A\text{-}\mathbb{K}%
_{G}\otimes A}$-definable is witnessed by the formulas $\varphi _{j,F}(x)$
defined by%
\begin{equation*}
\sup_{a\in F}\Vert a((\left\vert 1\right\rangle \left\langle 1\right\vert
\otimes u_{j})x(\left\vert 1\right\rangle \left\langle 1\right\vert \otimes
u_{i})-x)\Vert
\end{equation*}%
where $(u_{j})_{j\in J}$ is the fixed approximate unit for $A^{\alpha }$,
and $F$ ranges among the finite subsets of $\mathbb{K}_{G}\otimes A$.
\end{proof}

In the statement of Lemma \ref{Lemma:definable-stabilization}, it is
important that the $G$-C*-algebra $\mathbb{K}_{G}\otimes B$ is regarded as a
structure in the language $\mathcal{L}_{G}^{\mathrm{osos}\text{,}\mathbb{K}%
_{G}\otimes A\text{-}\mathbb{K}_{G}\otimes A}$, rather than a structure in
the language $\mathcal{L}_{G}^{\mathrm{osos}}$.

\begin{remark}
A similar discussion as above can be done when $( A,\alpha ) $ is a $\hat{G}$%
-C*-algebra. Fix a nondegenerate faithful representation $A\to B( \mathcal{H}%
) $. Set $\hat{X}=V^{\ast }\in M(c_{0}(\hat{G})\otimes C( G) )$. We define
the stabilization of $( A,\alpha ) $ as the $\hat{G}$-C*-algebra $( \mathbb{K%
}_{G}\otimes A,\alpha _{\mathbb{K}}) $, where the action $\alpha _{\mathbb{K}%
}$ is defined by%
\begin{equation*}
\alpha _{\mathbb{K}}( T\otimes a) =\hat{X}_{12}^{\ast }( 1\otimes T\otimes
1) \alpha ( a) _{13}\hat{X}_{12}\in c_{0}(\hat{G})\in \mathbb{K}_{G}\otimes
A\subseteq B( L^{2}( G) \otimes L^{2}( G) \otimes \mathcal{H}) \text{.}
\end{equation*}
\end{remark}

\section{Existential embeddings and the Rokhlin property\label%
{Section:existential}}

\subsection{Existential embeddings}

Let $G$ be either a compact or discrete quantum group. Considering $G$%
-C*-algebras as $\mathcal{L}_{G}^{\text{C*}}$-structures gives the notion of
positively $\mathcal{L}_{G}^{\text{C*}}$-existential *-homomorphism between $%
G$-C*-algebras; see Definition \ref{Definition:existential-embeddings}. When
the algebras are separable, and the group is second countable and either
compact and coexact or discrete and exact (which is the only case considered
in \cite{barlak_spatial_2017}), a $G$-equivariant homomorphism is positively 
$\mathcal{L}_{G}^{\text{C*}}$-existential if and only if it is sequentially
split in the sense of \cite[Definition 3.1]{barlak_spatial_2017}; see \ref%
{Proposition:existential-characterize}. In this section, we show that the
results from \cite{barlak_spatial_2017}, phrased in terms of positive
existential embeddings, can be obtained without any assumptions on the
algebras or the group.

%Suppose that $A,B$ are $G$-C*-algebras, and $\phi \colon A\rightarrow B$ is a $G$%
%-equivariant *-homomorphism.\ Then it follows from Proposition \ref%
%{Proposition:existential-characterize} that, in the case when $%
%A,B,L^{2}( G) $ are separable, and $G$ is either compact and
%coexact or discrete and exact (which is the only case considered in \cite[%
%Definition 3.1]{barlak_spatial_2017}), $\phi $ is $\mathcal{L}_{G}^{\text{C*}%
%}$-existential if and only it is $G$-equivariantly sequentially split in the
%sense of \cite[Definition 3.1]{barlak_spatial_2017}.

Positive existential *-homomorphisms are preserved by functors under general
assumptions. If $G$ is a compact or discrete quantum group, we regard $G$%
-C*-algebras as the objects of a category with $G$-equivariant
*-homomorphisms as morphisms. In the following proposition, we denote by $%
\Delta _{A}\colon A\rightarrow \prod_{\mathcal{U}}^{G}A$ the canonical \emph{%
diagonal }$\mathcal{L}_{G}^{\text{C*}}$-embedding of a $G$-C*-algebra into
the corresponding ultrapower.

\begin{proposition}
\label{Proposition:functor} Let $G_{0}$ and $G_{1}$ be either compact or
discrete quantum groups, let $F$ be a functor from the category of $G_{0}$%
-C*-algebras to the category of $G_{1}$-C*-algebras. Assume that for any
index set $I$, for any countably incomplete ultrafilter $\mathcal{U}$ over $%
I $, and for any $G_{0}$-C*-algebra $A$, there exists a $G_{1}$-equivariant
*-homomorphism $\theta \colon F(\prod_{\mathcal{U}}^{G_{0}}A)\rightarrow
\prod_{\mathcal{U}}^{G_{1}}F(A)$ such that $\theta \circ F(\Delta
_{A})=\Delta _{F(A)}$. Then $F$ maps $\mathcal{L}_{G_{0}}^{\text{C*}}$%
-existential *-homomorphisms to $\mathcal{L}_{G_{1}}^{\text{C*}}$%
-existential *-homomorphisms.
\end{proposition}

\begin{proof}
This is immediate using the semantic characterization of positively
existential morphisms from Proposition \ref%
{Proposition:existential-characterize}.
\end{proof}

Let $\phi \colon (A,\alpha )\rightarrow (B,\beta )$ be an injective
nondegenerate *-homomorphism. We identify $A$ with a $G$-invariant
subalgebra of $B$ via $\phi $. Suppose that $A_{0}$ is a $G$-C*-subalgebra
of $A$. Both $A$ and $B$ have a natural $A_{0}$-bimodule structure, and
hence can be regarded as structures in the language $\mathcal{L}_{G}^{\text{%
C*,}A_{0}\text{-}A_{0}}$ \cite[Subsection 3.5]{gardella_equivariant_2016}.
Recall that this is obtained from the language of $G$-C*-algebras by adding
function symbols for the $A_{0}$-bimodule structure, and replacing the
distinguished relation symbol for the metric with pseudometric function
symbols $d_{F}$ where $F$ ranges among the finite subsets of $A_{0}$. Then $%
d_{F}$ is interpreted in $B$ as the pseudometric%
\begin{equation*}
d_{F}(x,y)=\sup \left\{ \Vert a(x-y)\Vert ,\Vert (x-y)a\Vert \colon a\in
F\right\} \text{.}
\end{equation*}

\begin{proposition}
\label{Proposition:existential-bimodule}Let $G$ be a compact or discrete
quantum group, let $( A,\alpha )$ and $( B,\beta ) $ be $G$-C*-algebras, and
let $\phi \colon A\rightarrow B$ be a nondegenerate injective $G$%
-equivariant *-homomorphism. Fix a $G$-C*-subalgebra $A_{0}$ of $A$
containing an approximate unit for $A$. Then $\phi $ is $\mathcal{L}_{G}^{%
\text{C*}}$-existential if and only if it is $\mathcal{L}_{G}^{\text{C*,}%
A_{0}\text{-}A_{0}}$-existential.
\end{proposition}

\begin{proof}
As remarked above, one can assume that $\phi \colon A\rightarrow B$ is the
inclusion map. The forward implication is obvious. The converse implication
is easily shown using an increasing approximate unit for $A$ contained in $%
A_{0}$, which is also an approximate unit for $B$ since the inclusion $%
A\subseteq B$ is nondegenerate by assumption.
\end{proof}

A reason to consider the notion of positive $\mathcal{L}_{G}^{\text{C*}}$%
-existential *-homomorphism is that it allows one to conclude that several
properties pass from the target algebra to the domain algebra; see also \cite%
{barlak_sequentially_2016}. The following proposition is just a special
instance of Proposition \ref{Proposition:existential-preservation}.

\begin{proposition}
\label{Proposition:existential-preservation-algebras}Let $G$ be a compact or
discrete quantum group. Suppose that $\mathcal{C}$ is a class of $G$%
-C*-algebras that is definable by a uniform family of positive existential $%
\mathcal{L}_{G}^{\text{C*}}$-formulas. Suppose that $( A,\alpha ) $ and $(
B,\beta ) $ are $G$-C*-algebras, and $\phi \colon ( A,\alpha ) \rightarrow (
B,\beta ) $ is an $\mathcal{L}_{G}^{\text{C*}}$-existential *-homomorphism.
If $( B,\beta ) $ belongs to $\mathcal{C}$, then $( A,\alpha ) $ belongs to $%
\mathcal{C}$.
\end{proposition}

The following result is established in \cite[Proposition 3.3]%
{barlak_spatial_2017} when $G$ is compact and coexact or discrete and exact,
and second countable. Here, we remove these assumptions.

\begin{proposition}
\label{Proposition:crossed}Let $G$ be a compact or discrete quantum group,
and let $(A,\alpha )$ and $(B,\beta )$ be $G$-C*-algebras. Suppose that $%
\phi \colon A\rightarrow B$ is a nondegenerate $G$-equivariant
*-homomorphism. If $\phi $ is positively $\mathcal{L}_{G}^{\text{C*}}$%
-existential, then $G\ltimes \phi \colon (G\ltimes _{\alpha ,\mathrm{r}}A,%
\check{\alpha})\rightarrow (G\ltimes _{\beta ,\mathrm{r}}B,\check{\beta})$
is positively $\mathcal{L}_{\check{G}}^{\text{C*}}$-existential.
\end{proposition}

\begin{proof}
When $G$ is compact, this is a consequence of Proposition \ref%
{Proposition:functor} and Proposition \ref{Proposition:ultra-crossed-compact}%
. When $G$ is discrete, this follows from Lemma \ref%
{Lemma:ultraproduct-discrete} and Remark \ref{Remark:ultraproduct-discrete},
together with the semantic characterization of positively existential
embeddings from Proposition \ref{Proposition:existential-characterize}.
\end{proof}

The following result is established in \cite[Proposition 3.6, Proposition
3.7, Proposition 3.8]{barlak_spatial_2017} when $G$ is coexact, and second
countable. Here, we remove these assumptions, and provide a simpler proof.

\begin{proposition}
\label{Proposition:crossed2}Let $G$ be a compact quantum group, let $%
(A,\alpha )$ and $(B,\beta )$ be $G$-C*-algebras, and let $\phi \colon
A\rightarrow B$ be a nondegenerate $G$-equivariant *-homomorphism.

\begin{enumerate}
\item If $\phi $ is positively $\mathcal{L}_{G}^{\text{C*}}$-existential,
then $\phi |_{A^{\alpha }}\colon A^{\alpha }\rightarrow B^{\beta }$ is
positively $\mathcal{L}^{\text{C*}}$-existential.

\item $\phi \colon (A,\alpha )\rightarrow (B,\beta )$ is positively $%
\mathcal{L}_{G}^{\text{C*}}$-existential if and only if $\mathrm{id}\otimes
\phi \colon (\mathbb{K}_{G}\otimes A,\alpha _{\mathbb{K}})\rightarrow (%
\mathbb{K}_{G}\otimes B,\beta _{\mathbb{K}})$ is positively $\mathcal{L}%
_{G}^{\text{C*}}$-existential.

\item $\phi $ is positively $\mathcal{L}_{G}^{\text{C*}}$-existential if and
only if $G\ltimes \phi \colon (G\ltimes _{\alpha ,\mathrm{r}}A,\check{\alpha}%
)\rightarrow (G\ltimes _{\beta ,\mathrm{r}}B,\check{\beta})$ is positively $%
\mathcal{L}_{\check{G}}^{\text{C*}}$-existential.
\end{enumerate}
\end{proposition}

\begin{proof}
(1): This is an immediate consequence of Proposition \ref%
{Proposition:definable-substructure}, after observing that the fixed point
algebra of a $G$-C*-algebra is an $\mathcal{L}_{G}^{\text{C*}}$-definable $G$%
-C*-subalgebra.

(2): We can assume that $A\subseteq B$ and $\phi \colon A\rightarrow B$ is
the inclusion map. The forward implication is a consequence of Proposition %
\ref{Proposition:existential-preservation-algebras} and Proposition \ref%
{Proposition:stabilization}. For the converse, observe that we can identify $%
A$ with the $G$-C*-subalgebra $\left\vert 1\right\rangle \left\langle
1\right\vert \otimes A$ of $\mathbb{K}_{G}\otimes A$. 
%, where $1\in C(G)$ is
%the identity element, $\left\vert 1\right\rangle \in L^{2}(G)$ is the
%corresponding vector under the GNS construction associated with the Haar
%state, and $\left\vert 1\right\rangle \left\langle 1\right\vert \in \mathbb{K%
%}_{G}$ is the rank one projection onto the subspace of $L^{2}(G)$ spanned by 
%$\left\vert 1\right\rangle $. 
A similar observation applies to $B$, so we identify $\mathrm{id}\otimes
\phi $ with the inclusion map $\mathbb{K}_{G}\otimes A\subseteq \mathbb{K}%
_{G}\otimes B$, and $\phi $ with the restriction of $\mathrm{id}\otimes \phi 
$ to $\left\vert 1\right\rangle \left\langle 1\right\vert \otimes A$. We can
regard $\mathbb{K}_{G}\otimes A$ and $\mathbb{K}_{G}\otimes B$ as $\mathbb{K}%
_{G}\otimes A$-bimodules, and hence as structures in the language $\mathcal{L%
}_{G}^{\text{C*,}\mathbb{K}_{G}\otimes A\text{-}\mathbb{K}_{G}\otimes A}$.
Observe that $\mathrm{id}\otimes \phi $ is nondegenerate. It follows from
Proposition \ref{Proposition:existential-bimodule} that $\mathrm{id}\otimes
\phi $ is positively $\mathcal{L}_{G}^{\text{C*,}\mathbb{K}_{G}\otimes A%
\text{-}\mathbb{K}_{G}\otimes A}$-existential. Furthermore, by Lemma \ref%
{Lemma:definable-stabilization}, $\left\vert 1\right\rangle \left\langle
1\right\vert \otimes A\subseteq \mathbb{K}_{G}\otimes A$ and $\left\vert
1\right\rangle \left\langle 1\right\vert \otimes B\subseteq \mathbb{K}%
_{G}\otimes B$ are positively quantifier-free $\mathcal{L}_{G}^{\text{C*,}%
\mathbb{K}_{G}\otimes A\text{-}\mathbb{K}_{G}\otimes A}$-definable.
Therefore $\phi $ is positively $\mathcal{L}_{G}^{\text{C*,}\mathbb{K}%
_{G}\otimes A\text{-}\mathbb{K}_{G}\otimes A}$-existential by Proposition %
\ref{Proposition:definable-substructure}. In particular, $\phi $ is
positively $\mathcal{L}_{G}^{\text{C*,}A\text{-}A}$-existential. Since $\phi 
$ is nondegenerate, a further application of Proposition \ref%
{Proposition:existential-bimodule} shows that $\phi $ is positively $%
\mathcal{L}_{G}^{\text{C*}}$-existential. This concludes the proof.

(3): The forward implication is a consequence of Proposition \ref%
{Proposition:crossed}. The converse implication follows from the other
implication and Item (2) above, in view of the
Baaj--Skandalis--Takesaki--Takai duality for compact and discrete quantum
groups; see \cite[Theorem 1.20]{barlak_spatial_2017}, \cite[Chapter 9]%
{timmermann_invitation_2008}, \cite[Theorem 5.33]{de_commer_actions_2016}.
\end{proof}

\subsection{The Rokhlin property}

A generalization of the Rokhlin property for actions of classical compact
groups on separable C*-algebras has been considered in \cite[Section 4]%
{barlak_spatial_2017} for coexact second countable compact quantum groups.
Here, we remove all separability and coexactness assumptions. We fix a
compact quantum group $G$ throughout the rest of this section.

\begin{definition}
\label{Definition:Rokhlin-property} A $G$-C*-algebra $( A,\alpha ) $ is said
to have the (spatial) \emph{Rokhlin property }if the map $\alpha \colon (
A,\alpha ) \rightarrow ( C( G) \otimes A,\Delta \otimes \mathrm{id}_{A}) $
is positively $\mathcal{L}_{G}^{\text{C*}}$-existential.
\end{definition}

It follows from \cite[Lemma 1.24]{barlak_spatial_2017} and Proposition \ref%
{Proposition:existential-characterize} that Definition \ref%
{Definition:Rokhlin-property} agrees with \cite[Definition 4.1]%
{barlak_spatial_2017} in the particular case when $G$ is coexact and second
countable, and $A$ is separable (which is the only case considered in \cite[%
Definition 4.1]{barlak_spatial_2017}).

The following result generalizes \cite[Proposition 4.5]{barlak_spatial_2017}%
, with a simple and conceptual proof.

\begin{proposition}
Suppose that $G$ is a compact quantum group, and $( A,\alpha ) $ is a $G$%
-C*-algebra with the Rokhlin property. Then $( A,\alpha ) $ is free.
\end{proposition}

\begin{proof}
Observe that $( C( G) \otimes A,\Delta \otimes \mathrm{id}_{A}) $ is a free $%
G$-C*-algebra. If $( A,\alpha ) $ has the Rokhlin property, then $( A,\alpha
) $ is free in view of this observation and Proposition \ref%
{Proposition:existential-preservation-algebras}.
\end{proof}

The main result of \cite{barlak_spatial_2017} asserts that, whenever $G$ is
a coexact second countable compact quantum groups and $(A,\alpha )$ is a
separable $G$-C*-algebra, then several properties of $A$ are preserved under
taking crossed products or passing to the fixed point algebra. One can
deduce the natural generalization of such a statement to arbitrary compact
quantum groups from the properties of positively existential embeddings
established above.

\begin{theorem}
\label{Theorem:preservation-rokhlin} Let $(A,\alpha )$ be a $G$-C*-algebra
with the Rokhlin property. Then the canonical nondegenerate inclusion $%
A^{\alpha }\hookrightarrow A$ is positively $\mathcal{L}^{\text{C*}}$%
-existential, and the canonical $\check{G}$-equivariant nondegenerate
inclusion $G\ltimes _{\alpha }A\hookrightarrow \mathbb{K}_{\check{G}}\otimes
A$ is positively $\mathcal{L}_{\check{G}}^{\text{C*}}$-existential. Here $%
\mathbb{K}_{\check{G}}\otimes A$ is regarded as a $\check{G}$-C*-algebra
endowed with the stabilization of the trivial action of $\check{G}$ on $A$.
\end{theorem}

\begin{proof}
If $( A,\alpha ) $ has the Rokhlin property, then by definition $\alpha
\colon ( A,\alpha ) \rightarrow ( C( G) \otimes A,\Delta \otimes \mathrm{id}%
_{A}) $ is positively $\mathcal{L}_{G}^{\text{C*}}$-existential. By part (1)
of Proposition \ref{Proposition:crossed2}, the restriction of $\alpha $ to
the fixed point algebras is positively $\mathcal{L}^{\text{C*}}$%
-existential. Observing that the fixed point algebra of $( C( G) \otimes
A,\Delta \otimes \mathrm{id}_{A}) $ is equal to $1\otimes A\subseteq C( G)
\otimes A$, we deduce that the embedding $A^{\alpha }\hookrightarrow
1\otimes A$ is positively $\mathcal{L}^{\text{C*}}$-existential. This
concludes the proof of the first assertion.

By Proposition \ref{Proposition:crossed}, the positively $\mathcal{L}_{G}^{%
\text{C*}}$-existential *-homomorphism $\alpha \colon (A,\alpha )\rightarrow
(C(G)\otimes A,\Delta \otimes \mathrm{id}_{A})$ induces a positively $%
\mathcal{L}_{\check{G}}^{\text{C*}}$-existential *-homomorphism $G\ltimes
\alpha \colon G\ltimes _{\alpha }A\rightarrow G\ltimes _{\Delta \otimes 
\mathrm{id}_{A}}(C(G)\otimes A)$. A particular instance of the
Baaj--Skandalis--Takesaki--Takai duality for compact quantum groups---see 
\cite[Theorem 1.20]{barlak_spatial_2017}, \cite[Chapter 9]%
{timmermann_invitation_2008}, \cite[Theorem 5.33]{de_commer_actions_2016}%
---gives that there exists a $\check{G}$-equivariant *-isomorphism $\rho
\colon G\ltimes _{\Delta \otimes \mathrm{id}_{A}}(C(G)\otimes A)\rightarrow 
\mathbb{K}_{G}\otimes A$ such that $\rho \circ (G\ltimes \alpha )\colon
G\ltimes _{\alpha }A\rightarrow \mathbb{K}_{G}\otimes A$ is the canonical
inclusion. Therefore we conclude that the latter *-homomorphism is
positively $\mathcal{L}_{\check{G}}^{\text{C*}}$-existential as well.
\end{proof}

As an application, we extend several preservations results for crossed
products by actions with the Rokhlin property; see \cite%
{gardella_crossed_2014} and \cite{barlak_sequentially_2016}. When $G$ is
coexact and second countable, this recovers the main result of \cite%
{barlak_spatial_2017}, although the assertions concerning real rank and
stable rank are new even in this case.

\begin{corollary}
\label{Corollary:preservation-rokhlin} Let $(A,\alpha )$ be a $G$-C*-algebra
with the Rokhlin property. If $A$ satisfies any of the following properties,
then so do the fixed point algebra $A^{\alpha }$ and the crossed product $%
G\ltimes _{\alpha }A$:

\begin{enumerate}
\item being simple;

\item being separable, nuclear, and satisfying the UCT;

\item being separable and $\mathcal{D}$-absorbing for a given strongly
self-absorbing C*-algebra $\mathcal{D}$ or for $\mathcal{D}=\mathbb{K}( 
\mathcal{H}) $;

\item being expressible as a direct limit of certain weakly semiprojective
C*-algebras (see Theorem 3.10 in \cite{gardella_crossed_2014} for the
precise statement). This includes UHF-algebras (or matroid algebras),
AF-algebras, AI-algebras, AT-algebras, countable inductive limits of
one-dimensional NCCW-complexes, and several other classes.

\item having nuclear dimension at most $n$;

\item having decomposition rank at most $n$;

\item having real rank at most $n$;

\item having stable rank at most $n$.
\end{enumerate}
\end{corollary}

\begin{proof}
In view of Theorem \ref{Theorem:preservation-rokhlin}, the canonical
inclusions $A^{\alpha }\hookrightarrow A$ and $G\ltimes _{\alpha
}A\hookrightarrow \mathbb{K}_{G}\otimes A$ are positively $\mathcal{L}_{G}^{%
\text{C*}}$-existential. Items (1),(3),(4),(5),(8),(9) can be obtained by
observing that the corresponding properties are definable by a uniform
family of $\mathcal{L}_{G}^{\text{C*}}$-formulas as shown in \cite[Theorem
2.5.1, Theorem 2.5.2]{farah_model_2016} and \cite[Proposition 4.3]%
{gardella_equivariant_2016}; see also \cite[Remark 3.2]%
{gardella_equivariant_2016}. The items (2), (6), (7) can alternatively be
obtained by applying the fact that the corresponding properties are
definable by a uniform family of $\mathcal{L}_{G}^{\text{C*,nuc}}$-formulas
as shown in \cite[Section 5]{farah_model_2016}; see also \cite[Remark 3.5]%
{gardella_equivariant_2016}. The language $\mathcal{L}_{G}^{\text{C*,nuc}}$
is the \emph{nuclear language }for C*-algebras introduced in \cite[%
Subsection 3.3]{gardella_equivariant_2016}.
\end{proof}

Finally, the K-theory formula for fixed point algebras of Rokhlin actions of
finite groups from \cite{izumi_finite_2004} generalizes to the setting of
Rokhlin actions of compact quantum groups. This has been shown in \cite[%
Theorem 5.11]{barlak_spatial_2017} for separable coexact compact quantum
groups, but the proof applies equally well in general.

\begin{theorem}
\label{Theorem:k-theory}Let $G$ be a compact quantum group, and $( A,\alpha
) $ be a $G$-C*-algebra.\ If $( A,\alpha ) $ has the Rokhlin property, then
the canonical inclusion $A^{\alpha }\hookrightarrow A$ induces an injective
morphism $K_{\ast }( A^{\alpha }) \hookrightarrow K_{\ast }( A) $ in $K$%
-theory, whose range is%
\begin{equation*}
\left\{ x\in K_{\ast }( A) \colon K_{\ast }( \alpha ) ( x) =K_{\ast }( \iota
_{A}) ( x) \right\} \text{,}
\end{equation*}%
where $\iota _{A}\colon A\rightarrow C( G) \otimes A$, $a\mapsto 1\otimes a$
is the trivial action of $G$ on $A$.
\end{theorem}

In fact, under the assumptions of Theorem \ref{Theorem:k-theory}, one can
conclude that the inclusion of $K_{0}( A^{\alpha }) $ into $K_{0}( A) $ is
positively existential. Here we regard $K_{0}$-group as structures in the
language of dimension groups $( G,+,u) $ \emph{endowed with domains of
quantifications }to be interpreted as the subsets $\left\{ x\in G\colon
-nu\leq x\leq nu\right\} $ for $n\in \mathbb{N}$. This corresponds to the
notion of ultrapower of dimension groups considered in \cite%
{rainone_crossed_2016}.

\subsection{Rigidity}

Suppose that $G$ is a compact quantum group, and $(B,\beta )$ is a $G$%
-C*-algebra. Let $\tilde{B}$ be the minimal unitization of $B$. Then the $G$%
-action $\beta $ has a unique extension to a $G$-action $\beta $ on $\tilde{B%
}$. It is easy to see that if a *-homomorphism $\phi \colon A\rightarrow B$
is positively $\mathcal{L}_{G}^{\text{C*}}$-existential, then its unique
extension to a unital *-homomorphism $\phi \colon \tilde{A}\rightarrow 
\tilde{B}$ is positively $\mathcal{L}_{G}^{\text{C*,1}}$-existential, where $%
\mathcal{L}_{G}^{\text{C*,1}}$ is the language obtained from $\mathcal{L}%
_{G}^{\text{C*}}$ by adding a constant symbols for the multiplicative unit.
The following definition is introduced in the setting of compact quantum
group actions in \cite[Definition 5.1]{barlak_spatial_2017}.

\begin{definition}
\label{Definition:appr-eq}Let $G$ be a compact quantum group, $( A,\alpha )
,( B,\beta ) $ be two $G$-C*-algebras, and $\phi _{1},\phi _{2}\colon
A\rightarrow B$ be $G$-equivariant *-homomorphisms. Then $\phi _{1},\phi
_{2} $ are approximately $G$-unitarily equivalent, in formulas $\phi
_{1}\thickapprox _{\mathrm{u},G}\phi _{2}$, if there exists a net $( v_{i}) $
of unitary elements of the fixed point algebra $\tilde{B}^{\beta }$ such
that $\phi _{2}$ is the limit of $\mathrm{Ad}( v_{i}) \circ \phi _{1}$ is
the topology of pointwise norm convergence.
\end{definition}

When $B$ is separable, one can replace nets with sequences in Definition \ref%
{Definition:appr-eq}. In the case when $G$ is the trivial group, Definition %
\ref{Definition:appr-eq} recovers the usual notion of approximate unitary
equivalence $\phi _{1}\thickapprox _{\mathrm{u}}\phi _{2} $ for the
*-homomorphisms $\phi _{1},\phi _{2}$.

A \emph{rigidity result }for Rokhlin actions of coexact second countable
compact quantum groups, generalizing results for finite and compact from 
\cite%
{izumi_finite_2004,nawata_finite_2016,gardella_equivariant_2016,gardella_classification_2014}
and for finite quantum groups from \cite{kodaka_rohlin_2015}, has been
obtained in \cite[Theorem 5.10]{barlak_spatial_2017}. In the rest of this
section, we observe here that such a result holds for arbitrary (not
necessarily coexact) second countable compact quantum groups.

\begin{theorem}
\label{Theorem:rigidity}Let $G$ be a second countable compact quantum group,
let $A$ be a separable C*-algebra, and let $\alpha ^{(0)},\alpha ^{(1)}$ be $%
G$-actions on $A$. If $(A,\alpha ^{(0)})$ and $(A,\alpha ^{(1)})$ have the
Rokhlin property, then $\alpha ^{(0)}\thickapprox _{\mathrm{u}}\alpha ^{(1)}$
if and only if $\alpha ^{(0)}\thickapprox _{\mathrm{u},G}\alpha ^{(1)}$.
\end{theorem}

One can deduce from Theorem \ref{Theorem:rigidity} the following corollary,
similarly as Proposition \cite[Proposition 6.4]{barlak_spatial_2017} is
deduced from \cite[Theorem 5.10]{barlak_spatial_2017}.

\begin{corollary}
\label{Corollary:Rokhlin-ssa}Let $G$ be a second countable compact quantum
group, and $D$ be a strongly self-absorbing C*-algebra. Then there exists at
most one conjugacy class of $G$-actions on $D$ with the Rokhlin property.
\end{corollary}

The rest of this subsection is dedicated to the proof of Theorem \ref%
{Theorem:rigidity}.

We fix separable $G$-C*-algebras $( A,\alpha ) $ and $( B,\beta ) $, and
homomorphisms $\phi ,\phi _{1},\phi _{2}\colon (A,\alpha)\rightarrow
(B,\beta)$. We will use tacitly the fact that the unitary group of a unital
C*-algebras is positively quantifier-free $\mathcal{L}^{\text{C*,1}}$%
-definable with respect to the class of unital C*-algebras. Indeed, it is
the zeroset of the stable positive quantifier-free $\mathcal{L}^{\text{C*,1}%
} $-formula $\max \left\{ \| xx^{\ast }-1\| ,\| x^{\ast }x-1\| \right\} $.

\begin{lemma}
\label{Lemma:equivariant-modulo-2}Suppose that $\beta \circ \phi
\thickapprox _{\mathrm{u}}( \mathrm{id}\otimes \phi ) \circ \alpha $, and
that $\beta $ has the Rokhlin property. Then for every finite subset $F$ of $%
A$ and every $\varepsilon >0$, there exists a unitary $v$ in the unitization
of $C( G) \otimes B$ such that%
\begin{equation*}
\sup_{x\in F}\| ( \beta \circ \mathrm{\mathrm{Ad}}( v) \circ \phi -( \mathrm{%
id}\otimes ( \mathrm{\mathrm{Ad}}( v) \circ \phi ) ) \circ \alpha ) ( x) \|
<\varepsilon
\end{equation*}%
and%
\begin{equation*}
\sup_{x\in F}( \| [ \phi ( x) ,v] \| -\| ( \beta \circ \phi -( \mathrm{id}%
\otimes \phi ) \circ \alpha ) ( x) \| ) <\varepsilon \text{.}
\end{equation*}
\end{lemma}

\begin{proof}
Fix a finite subset $F$ of $A$ and $\varepsilon >0$. By \cite[Lemma 5.4]%
{barlak_spatial_2017}, there exists a unitary $v$ in the unitization of $%
C(G)\otimes B$ such that%
\begin{equation}
\sup_{x\in F}\Vert ((\Delta \boxtimes \beta )\circ \mathrm{\mathrm{Ad}}%
(v)\circ (1\boxtimes \phi )-(\mathrm{id}\otimes (\mathrm{\mathrm{Ad}}%
(v)\circ (1\boxtimes \phi )))\circ \alpha )(x)\Vert <\varepsilon \text{\label%
{Equation:equivariant-modulo-1}}
\end{equation}%
and%
\begin{equation}
\sup_{x\in F}(\Vert \lbrack (1\boxtimes \phi )(x),v]\Vert -\Vert (\beta
\circ \phi -(\mathrm{id}\circ \phi )\circ \alpha )(x)\Vert )<\varepsilon 
\text{\label{Equation:equivariant-modulo-2}.}
\end{equation}%
Since $\beta $ has the Rokhlin property, the map 
\begin{equation*}
1\boxtimes \mathrm{id}_{B}\colon (B,\beta )\rightarrow (C(G)\boxtimes
B,\Delta \boxtimes \beta )
\end{equation*}%
is positively $\mathcal{L}_{G}^{\text{C*}}$-existential. Therefore its
unique unital extension 
\begin{equation*}
1\boxtimes \mathrm{id}_{B}\colon (\tilde{B},\beta )\rightarrow
(C(G)\boxtimes \tilde{B},\Delta \boxtimes \beta )
\end{equation*}%
is positively $\mathcal{L}_{G}^{\text{C*,1}}$-existential. Since the
conditions from Equation \ref{Equation:equivariant-modulo-1} and Equation %
\ref{Equation:equivariant-modulo-2} can be expressed by positive
quantifier-free $\mathcal{L}_{G}^{\text{C*,1}}$-formulas, we conclude that
there exists a unitary $v\in \tilde{B}$ as wanted.
\end{proof}

\begin{lemma}
\label{Lemma:equivariant-modulo4}Suppose that $\beta \circ \phi \thickapprox
_{\mathrm{u}}( \mathrm{id}\otimes \phi ) \circ \alpha $, and $\beta $ has
the Rokhlin property. Then there exists a $G$-equivariant *-homomorphism $%
\psi \colon ( A,\alpha ) \rightarrow ( B,\beta ) $ such that $\psi
\thickapprox _{\mathrm{u}}\phi $.
\end{lemma}

\begin{proof}
One replaces \cite[Lemma 5.5]{barlak_spatial_2017} with Lemma \ref%
{Lemma:equivariant-modulo-2} in the proof \cite[Proposition 5.6]%
{barlak_spatial_2017}.
\end{proof}

We prove the following lemma directly using the definition of positively $%
\mathcal{L}_{G}^{\text{C*}}$-existential $G$-equivariant *-homomorphism.

\begin{lemma}
\label{Lemma:equivariant-modulo-4}Suppose that $( C,\gamma ) $ is a $G$%
-C*-algebra, and $\psi \colon ( B,\beta ) \rightarrow ( C,\gamma ) $ is
positively $\mathcal{L}_{G}^{\text{C*}}$-existential $G $-equivariant
*-homomorphism. If $\psi \circ \phi _{1}\thickapprox _{\mathrm{u},G}\psi
\circ \phi _{2}$ then $\phi _{1}\thickapprox _{\mathrm{u}}\phi _{2} $.
\end{lemma}

\begin{proof}
Observe that the unital extension $\psi \colon (\tilde{B},\beta )\rightarrow
(\tilde{C},\gamma )$ is a positively $\mathcal{L}_{G}^{\text{C*,1}}$%
-existential $G$-equivariant *-homomorphism. Fix $\varepsilon >0$ and a
finite subset $F$ of $A$. Since by assumption $\psi \circ \phi
_{1}\thickapprox _{\mathrm{u},G}\psi \circ \phi _{2}$, there exists $v\in 
\tilde{C}^{\gamma }$ such that%
\begin{equation*}
\sup_{x\in F}\| v(\psi \circ \phi _{1})(x)v^{\ast }-(\psi \circ \phi
_{2})(x)\| <\varepsilon \text{.}
\end{equation*}%
Since $\psi $ is a positively $\mathcal{L}_{G}^{\text{C*}}$-existential $G$%
-equivariant *-homomorphism, there exists a unitary $u\in \tilde{B}^{\beta }$
such that $\sup_{x\in F}\| v{}\phi _{1}(x){}v^{\ast }-\phi _{2}(x)\|
<\varepsilon$. This concludes the proof.
\end{proof}

The following proposition in the case when $G$ is coexact is \cite[Corollary
5.9]{barlak_spatial_2017}. The proof is analogous, where one replaces \cite[%
Proposition 5.8]{barlak_spatial_2017} with Lemma \ref%
{Lemma:equivariant-modulo-4}.

\begin{proposition}
\label{Proposition:equivariant-modulo}Suppose that $G$ is a second countable
compact quantum group, and $( A,\alpha ) $ and $( B,\beta ) $ are separable $%
G$-C*-algebras. Let $\phi _{1},\phi _{2}\colon ( A,\alpha ) \rightarrow (
B,\beta ) $ be $G$-equivariant *-homomorphisms. If $\phi _{1}\thickapprox _{%
\mathrm{u}}\phi _{2}$ and $( B,\beta ) $ has the Rokhlin property, then $%
\phi _{1}\thickapprox _{\mathrm{u},G}\phi _{2}$.
\end{proposition}

Finally, using Proposition \ref{Proposition:equivariant-modulo} instead of 
\cite[Corollary 5.9]{barlak_spatial_2017}, and Lemma \ref%
{Lemma:equivariant-modulo4} instead of \cite[Corollary 5.6]%
{barlak_spatial_2017}, one can prove Theorem \ref{Theorem:rigidity}
reasoning as in the proof of \cite[Theorem 5.10]{barlak_spatial_2017}.

As fruitful as the Rokhlin property is, it is also very rare. In fact, there
are many very interesting C*-algebras that do not admit any Rokhlin action
of a nontrivial compact quantum group. For example, we have the following
result. For $\theta \in \mathbb{R}\setminus \mathbb{Q}$, we denote by $%
A_{\theta }$ the irrational rotation algebra. Also, we write $\mathcal{O}%
_{\infty }$ for the Cuntz algebra on infinitely many generators.

\begin{proposition}
Let $G$ be a nontrivial finite quantum group and let $\theta\in\mathbb{R}%
\setminus\mathbb{Q}$. There do not exist any actions of $G$ on either $%
A_{\theta }$ or $\mathcal{O}_{\infty }$ with the Rokhlin property.
\end{proposition}

\begin{proof}
For classical finite groups, this is well known; see, for example, Section~3
in~\cite{gardella_rokhlin_2014}. Suppose that $G$ is not classical finite
quantum group. Then $C(G)$ is a finite dimensional C*-algebra which is not
commutative. Find $n\in \mathbb{N}$, with $n>1$, and orthogonal projections $%
p_{1},\ldots ,p_{n}\in C(G)$ which add up to $1_{C(G)}$ and are unitarily
equivalent in $C(G)$.

Let $A$ be either $A_{\theta }$ or $\mathcal{O}_{\infty }$. Suppose that $%
\alpha $ is an action of $G$ on $A$, and assume by contradiction that $%
\alpha $ has the Rokhlin property. Therefore $\alpha :\left( A,\alpha
\right) \rightarrow \left( C\left( G\right) \otimes A,\Delta \otimes \mathrm{%
id}_{A}\right) $ is positively $\mathcal{L}_{G}^{\text{C*}}$-existential. By
considering the projections $p_{j}\otimes 1_{A}\in C(G)\otimes A$, we deduce
that there exist orthogonal projections $q_{1},\ldots ,q_{n}\in A$ which add
up to $1_{A}$ and are unitarily equivalent in $A$. In the case of $\mathcal{O%
}_{\infty }$, this would imply that the class of unit of $\mathcal{O}%
_{\infty }$ in its $K_{0}$-group is divisible by $n>1$, which is not true.
For the case of $A_{\theta }$, and denoting its unique trace by $\tau $, we
would get $1=\tau (1_{A_{\theta }})=\sum_{j=1}^{n}\tau (q_{j})=n\tau (q_{1})$%
, since unitarily equivalent projections have the same value on traces. The
range of $\tau $ on traces is known not to contain any rational which is not
an integer, so this is again a contradiction. This finishes the proof.
\end{proof}

Nonexistence results like the one just explained are the main motivation for
introducing a more flexible notion, the Rokhlin dimension, which is the
content of the next section. Despite not admitting any Rokhlin action, the
algebras $\mathcal{O}_{\infty }$ and $A_{\theta }$ have many actions with
finite Rokhlin dimension; see, for instance, \cite{gardella_rokhlin_2014}.

\section{Order zero dimension and Rokhlin dimension\label{Section:order-zero}%
}

\subsection{Order zero dimension}

The notion of positive $\mathcal{L}_{G}^{\text{C*}}$-existential $G$%
-equivariant *-homomorphism admits a natural generalization, which has been
introduced in the classical setting in \cite[Section 5]%
{gardella_equivariant_2016}. We consider here its natural extension to
compact quantum groups. In the following definition, given a *-homomorphism $%
\theta \colon A\rightarrow B$, we consider $B$ as an $A$-bimodule, with
respect to the $A$-bimodule structure defined by $a\cdot b=\theta (a)\ b$
and $b\cdot a=b\ \theta (a)$ for every $a\in A$ and $b\in B$.

\begin{definition}
\label{Definition:oz-dimension} Let $G$ be either a discrete or a compact
quantum group, let $(A,\alpha )$ and $(B,\beta )$ be $G$-C*-algebras. Fix a
cardinal number $\kappa $ larger than the density characters of $A$, $B$,
and $L^{2}(G)$, and a countably incomplete $\kappa $-good filter $\mathcal{F}
$. We say that a *-homomorphism $\theta \colon (A,\alpha )\rightarrow
(B,\beta )$ has \emph{$G$-equivariant order zero dimension at most $d$},
written $\dim _{\mathrm{\mathrm{oz}}}^{G}(\theta )\leq d$, if there exist $G$%
-equivariant completely positive contractive order zero $A$-bimodule maps $%
\psi _{0},\ldots ,\psi _{d}\colon (B,\beta )\rightarrow (\prod_{\mathcal{F}%
}^{G}A,\alpha _{\mathcal{F}})$ such that the sum $\psi =\psi _{0}+\cdots
+\psi _{d}$ is a contractive linear map such that the following diagram
commutes 
\begin{equation*}
\xymatrix{ A\ar[dr]_-{\theta}\ar[rr]^-{\Delta_A} & &
\prod_{\mathcal{F}}^{G}A\\ & B\ar[ur]_-{\psi} & }
\end{equation*}
\end{definition}

The $G$-equivariant order zero dimension $\dim _{\mathrm{\mathrm{oz}}}^{G}(
\theta ) $ of $\theta $ is the least $d\in \mathbb{N}$ such that $\dim _{%
\mathrm{\mathrm{oz}}}^{G}( \theta )\leq d$, if such a $d$ exists, and $%
\infty $ otherwise. The order zero dimension of a *-homomorphism between
C*-algebras can be obtained as the particular instance of Definition \ref%
{Definition:oz-dimension} when $G$ is the trivial group.

\begin{remark}
\label{Remark:independent}Definition \ref{Definition:oz-dimension} does not
depend on the choice of the countably incomplete $\kappa $-good ultrafilter $%
\mathcal{F}$. This can be seen, for instance, by considering the syntactic
characterization presented in Remark \ref{Remark:semantic} below.
Furthermore, when $G$ is second countable, and $A,B$ are separable, one can
choose $\mathcal{F}$ to be any countably incomplete filter, such as the
filter of cofinite subsets of $\mathbb{N}$.
\end{remark}

\begin{remark}
\label{Remark:semantic}One can give a syntactic reformulation of the notion
of $G$-equivariant order zero dimension. To this purpose, one can consider
the ordered operator space language $\mathcal{L}^{\mathrm{osos}}$ introduced
in \cite[Subsection 3.1]{gardella_equivariant_2016}, and the order zero
language $\mathcal{L}^{\mathrm{oz}}$ introduced in \cite[Subsection 3.2]%
{gardella_equivariant_2016}. Adding function symbols for the $A$-bimodule
operations give the $A$-bimodule ordered operator space language $\mathcal{L}%
^{\mathrm{osos}\text{,}A\text{-}A}$ and the $A$-bimodule order zero language 
$\mathcal{L}^{\mathrm{oz}\text{,}A\text{-}A}$. In both these languages, the
distinguished symbol for the metric is replaced by pseudometric symbols $%
d_{F}$ for $F$ ranging among the finite subsets of $A$, to be interpreted as
the pseudometric%
\begin{equation*}
d_{F}( x,y) =\max \left\{ \| a( x-y) \| ,\| ( x-y) a\| \colon a\in F\right\} 
\text{.}
\end{equation*}

One can consider the $G$-equivariant version $\mathcal{L}_{G}^{\mathrm{osos}%
\text{,}A\text{-}A}$ and $\mathcal{L}_{G}^{\mathrm{oz}\text{,}A\text{-}A}$
for each of these languages, which are defined starting from these languages
by adding symbols for a $G$-action as in Subsection \ref%
{Subsection:axiomatization-compact}. The $\mathcal{L}_{G}^{\mathrm{osos}%
\text{,}A\text{-}A}$-morphisms between $G$-C*-algebras are precisely the $G$%
-equivariant completely positive contractive $A$-bimodule maps, while the $%
\mathcal{L}_{G}^{\mathrm{oz}\text{,}A\text{-}A}$-morphisms are precisely the 
$G$-equivariant completely positive contractive order zero $A$-bimodule
maps. One can then rephrase Definition \ref{Definition:oz-dimension} by
asserting that the order zero dimension of a nondegenerate *-homomorphism $%
\theta \colon ( A,\alpha ) \rightarrow ( B,\beta ) $ is at most $d$ if and
only if for every positive quantifier-free $\mathcal{L}_{G}^{\mathrm{oz},A%
\text{-}A}$-formula $\varphi ( \overline{z},\overline{y}) $, for every
positive quantifier-free $\mathcal{L}_{G}^{\mathrm{osos}\text{,}A\text{-}A}$%
-formula $\psi ( \overline{x},\overline{z},\overline{y}) $, where the
variables $\overline{z}$ have finite-dimensional C*-algebras as sorts, for
every tuples $\overline{a}$ in $A$, $\overline{b}$ in $B$, and $\overline{w}$
in finite-dimensional C*-algebras, and for every $\varepsilon >0$, there
exist tuples $\overline{c}_{0},\ldots ,\overline{c}_{d}$ in $A$ such that%
\begin{equation*}
\psi ( \overline{a},\overline{w},\overline{c}_{0}+\cdots +\overline{c}_{d})
\leq \psi ( \theta ( \overline{a}) ,\overline{w},\overline{b}) +\varepsilon
\quad \text{and}\quad \varphi ( \overline{w},\overline{c}_{j}) \leq \varphi
( \overline{w},\overline{b}) +\varepsilon \text{ for }j=0,1,\ldots ,d\text{.}
\end{equation*}%
If $A_{0}$ is a $G$-C*-subalgebra of $A$ containing an approximate unit for $%
A$, then one can replace $\mathcal{L}_{G}^{\mathrm{osos}\text{,}A\text{-}A}$
and $\mathcal{L}_{G}^{\mathrm{\mathrm{oz}}\text{,}A\text{-}A}$ with $%
\mathcal{L}_{G}^{\mathrm{osos}\text{,}A_{0}\text{-}A_{0}}$ and $\mathcal{L}%
_{G}^{\mathrm{\mathrm{oz}}\text{,}A_{0}\text{-}A_{0}}$ in the discussion
above.
\end{remark}

\begin{remark}
\label{Remark:nondegenerate}Suppose that $\theta \colon (A,\alpha
)\rightarrow (B,\beta )$ is a nondegenerate *-homomorphism. Fix a countably
incomplete $\kappa $-good ultrafilter $\mathcal{U}$, where $\kappa $ is
larger than the density character of $A$, $B$, and $L^{2}(G)$. Let also $%
A_{0}$ be a $G$-C*-subalgebra of $A$ containing an approximate unit for $A$.
In the case when $G$ is compact, assume furthermore that $A_{0}$ is
contained in the fixed point algebra (which is a nondegenerate C*-subalgebra
of $A$). Following \cite[Section 1]{kirchberg_central_2006} and \cite[Remark
1.3]{barlak_sequentially_2016}, we consider the $G$-C*-subalgebra $\overline{%
A_{0}\cdot \prod_{\mathcal{U}}^{G}A\cdot A_{0}}$ of $\prod_{\mathcal{U}%
}^{G}A $. It is easy to see that $\overline{A_{0}\cdot \prod_{\mathcal{U}%
}^{G}A\cdot A_{0}}$ can be identified with the ultrapower of $(A,\alpha )$
regarded as an $\mathcal{L}_{G}^{\text{C*,}A_{0}\text{-}A_{0}}$-structure.
It follows from the observations above that if $\mathrm{dim}_{\mathrm{oz}%
}^{G}(\theta )\leq d$, then there exist $G$-equivariant completely positive
contractive order zero $A$-bimodule maps $\psi _{0},\ldots ,\psi _{d}\colon
B\rightarrow \overline{A_{0}\cdot \prod_{\mathcal{U}}^{G}A\cdot A_{0}}$ such
that $\psi :=\psi _{0}+\cdots +\psi _{d}$ is contractive and $\psi \circ
\theta $ is the diagonal embedding from $A$ to $\overline{A_{0}\cdot \prod_{%
\mathcal{U}}^{G}A\cdot A_{0}}$.
\end{remark}

Recall that a \emph{unital} completely positive order zero map is a
*-homomorphism. Next, we prove a similar result with `unital' being replaced
by `nondegenerate'. If $\alpha $ is a $G$-action on $A$, then $\alpha $
admits a unique extension to a $G$-action on $\tilde{A}$, which we still
denote by $\alpha $.

\begin{lemma}
\label{Lemma:oz-homomorphism} Let $\phi \colon A\rightarrow B$ be a
completely positive order zero map between C*-algebras. If $\phi$ is
nondegenerate, then it is a *-homomorphism.
\end{lemma}

\begin{proof}
If $A$ is unital, then $\phi \colon A\rightarrow B\subset M(B)$ is unital,
and the conclusion follows from the structure theorem for completely
positive order zero maps from \cite{winter_completely_2009}. Suppose that $A$
is not unital, and let $\tilde{A}$ be its unitization. We let $B^{\ast \ast
} $ be the second dual of $B$, which we identify with the enveloping von
Neumann algebra of $B$. Fix an increasing approximate unit $(u_{j})_{j\in J}$
for $A$, and set $g=\sup_{j\in J}\phi (u_{j})\in B^{\ast \ast }$. By \cite[%
Proposition 3.2]{winter_completely_2009}, the (unique) linear map $\tilde{%
\phi}\colon \tilde{A}\rightarrow B^{\ast \ast }$ extending $\phi $ with $%
\phi (1)=g$, is completely positive of order zero. Since $\phi $ is
nondegenerate, we must have $g=1\in B^{\ast \ast }$. Therefore $\tilde{\phi}$
is a *-homomorphism, and hence so is $\phi $.
\end{proof}

\begin{lemma}
\label{Lemma:inclusion-unitization}Suppose that $G$ is a compact or discrete
quantum group, and $( A,\alpha ) $ is a $G$-C*-algebra. The inclusion map $(
A,\alpha ) \hookrightarrow (\tilde{A},\alpha )$ is positively $\mathcal{L}%
_{G}^{\text{C*,}A\text{-}A}$-existential.
\end{lemma}

\begin{proof}
Consider an approximate unit $(u_{j})_{j\in J}$ for $A$. Fix a
quantifier-free $\mathcal{L}_{G}^{\text{C*,}A\text{-}A}$-formula, $\varphi
\left( \bar{x},\bar{y}\right) $, and a tuple $\overline{a}$ in $A$. If $%
\left( b_{k}+\lambda _{k}1\right) _{k=1}^{n}$ is a tuple in $\tilde{A}$
satisfying the condition $\varphi (\overline{a},\bar{y})<r$, then, for a
suitable $j\in J$, $\left( b_{k}+\lambda _{k}u_{j}\right) _{k=1}^{n}$ is a
tuple in $A$ satisfying the same condition.
\end{proof}

\begin{proposition}
\label{Proposition:properties-oz-dim} Let $G$ be either a compact or
discrete quantum group, and let $\theta\colon ( A,\alpha ) \to( B,\beta ) $
be an injective *-homomorphism between $G$-C*-algebras.

\begin{enumerate}
\item If $( C,\gamma ) $ is a $G$-C*-algebra, and $\psi \colon ( B,\beta )
\rightarrow ( C,\gamma ) $ is a *-homomorphism, then 
\begin{equation*}
\dim _{\mathrm{\mathrm{oz}}}^{G}( \psi \circ \theta ) +1\leq (\dim _{\mathrm{%
\mathrm{oz}}}^{G}( \psi ) +1)(\dim _{\mathrm{\mathrm{oz}}}^{G}( \theta ) +1).
\end{equation*}

\item If $D$ is any C*-algebra, then $\dim _{\mathrm{oz}}^G(\theta \otimes 
\mathrm{id}_{D})\leq \dim _{\mathrm{\mathrm{oz}}}^G(\theta)$.

\item Let $I$ be a directed set, and let $( ( A_{i},\alpha _{i}) ,\theta
_{ij}) _{i,j\in I}$ be a direct system of $G$-C*-algebras. For every $i\in I$%
, let $\theta _{i\infty }\colon ( A_{i},\alpha ) \rightarrow (%
\underrightarrow{\lim }A_{i},\underrightarrow{\lim }\alpha _{i})$ be the
canonical *-homomorphism. Then $\dim _{\mathrm{oz}}^{G}( \theta _{i\infty })
\leq \limsup_{j}\dim _{\mathrm{oz}}^{G}( \theta _{ij}) $.

\item Let $I$ be a directed set, and for $k=0,1$, let $((A_{i}^{(k)},\alpha
_{i}^{(k)}),\theta _{ij}^{(k)})_{i,j\in I}$ be a direct system of $G$%
-C*-algebras. Let $(\eta _{i}\colon (A_{i}^{(0)},\alpha
_{i}^{(0)})\rightarrow (A_{i}^{(1)},\alpha _{i}^{(1)}))$ be a compatible
family of *-homomorphism. Then $\dim _{\mathrm{oz}}^{G}(\underrightarrow{%
\lim }\eta _{i})\leq \limsup_{i\in I}\dim _{\mathrm{oz}}^{G}( \eta _{i})$.

\item Suppose that $S$ is a positively quantifier-free $\mathcal{L}_{G}^{%
\mathrm{osos}\text{,}A\text{-}A}$-definable $G$-C*-subalgebra relative the
class $\mathcal{C}$ of $G$-C*-algebras that contain $(A,\alpha )$. Then $%
\theta $ maps $S^{(A,\alpha )}$ to $S^{(B,\beta )}$, and $\dim _{\mathrm{oz}%
}^{G}(\theta |_{S^{(A,\alpha )}})\leq \dim _{\mathrm{\mathrm{oz}}%
}^{G}(\theta )$.
\end{enumerate}
\end{proposition}

\begin{proof}
(1)--(4) can be proved similarly as for classical groups; see \cite[%
Proposition 5.4]{gardella_equivariant_2016}. (5) is an easy consequence of
the definition.
\end{proof}

The following is one of our main results regarding order zero dimension. In
the proof of part (5), we will use the fact that if $\psi \colon
A\rightarrow B$ is a completely positive contractive order zero map, then $%
\psi (a)\psi (bc)=\psi (ab)\psi (c)$ for all $a,b,c\in A$. This fact follows
easily by considering the induced *-homomorphism $C_{0}((0,1])\otimes
A\rightarrow B$ \cite[Corollary 4.1]{winter_completely_2009}.

\begin{theorem}
\label{Theorem:properties-oz-dim-2} Let $G$ be either a compact or discrete
quantum group, and let $\theta \colon (A,\alpha )\rightarrow (B,\beta )$ be
a \emph{nondegenerate} injective *-homomorphism. Then:

\begin{enumerate}
\item $\dim _{\mathrm{oz}}^{G}( \theta ) =0$ if and only if $\theta $ is
positively $\mathcal{L}_{G}^{\text{C*}}$-existential.

\item The $G$-equivariant order zero dimension of $\tilde{\theta}\colon (%
\tilde{A},\alpha )\rightarrow (\tilde{B},\beta )$ is equal to $\dim _{%
\mathrm{oz}}^{G}(\theta )$.

\item The $\check{G}$-equivariant order zero dimension of $G\ltimes \theta
\colon (G\ltimes _{\alpha ,\mathrm{r}}A,\check{\alpha})\rightarrow (G\ltimes
_{\beta ,\mathrm{r}}B,\check{\beta}))$ is less than or equal to $\dim _{%
\mathrm{oz}}^{G}(\theta )$.

\item If $(B,\beta )$ is free, and $\dim _{G}^{\mathrm{\mathrm{oz}}}(\theta
)<+\infty $, then $(A,\alpha )$ is free.
\end{enumerate}
\end{theorem}

\begin{proof}
We prove the theorem for compact $G$, since the case of discrete $G$ is
analogous.

(1): It is obvious that if $\theta $ is positively $\mathcal{L}_{G}^{\text{C*%
}}$-existential, then $\dim _{\mathrm{oz}}^{G}(\theta )=0$. Conversely,
suppose that $\dim _{\mathrm{oz}}^{G}(\theta )=0$. Fix a countably
incomplete $\kappa $-good ultrafilter $\mathcal{U}$, where $\kappa $ is
larger than the density character of $A$, $B$, and $L^{2}(G)$. Then by
Remark \ref{Remark:nondegenerate}, there exists a completely positive
contractive order zero $A$-bimodule map $\psi :B\rightarrow \overline{A\cdot
\prod_{\mathcal{U}}^{G}A\cdot A}$ such that $\psi \circ \phi \colon
A\rightarrow \overline{A\cdot \prod_{\mathcal{U}}^{G}A\cdot A}$ is the
diagonal inclusion. Since $\phi $ is nondegenerate and $\psi $ is an $A$%
-bimodule map, $\psi $ is nondegenerate. Therefore $\psi $ is a
*-homomorphism, witnessing that $\theta $ is positively $\mathcal{L}_{G}^{%
\text{C*}}$-existential.

(2): We can assume that $\theta \colon A\rightarrow B$ is the inclusion map.
Suppose that $\dim _{\mathrm{oz}}^{G}( \theta ) \leq d$. By Lemma \ref%
{Lemma:inclusion-unitization} and Remark \ref{Remark:nondegenerate}, there
exist $G$-equivariant completely positive order zero $A$-bimodule maps 
\begin{equation*}
\psi _{0},\ldots ,\psi _{d}\colon \tilde{B}\rightarrow \overline{A^{\alpha
}\cdot \prod\nolimits_{\mathcal{U}}^{G}A\cdot A^{\alpha }}\subseteq 
\overline{A^{\alpha }\cdot \prod\nolimits_{\mathcal{U}}^{G}\tilde{A}\cdot
A^{\alpha }}
\end{equation*}%
such that $\psi =\psi _{0}+\cdots +\psi _{d}$ is contractive and $\psi \circ 
\tilde{\theta}|_{A}\colon A\rightarrow \overline{A^{\alpha }\cdot
\prod\nolimits_{\mathcal{U}}^{G}A\cdot A^{\alpha }}$ is the diagonal
embedding. Since $\psi _{0},\ldots ,\psi _{d}$ are $A$-bimodule maps, given $%
a\in A^{\alpha }$ we have 
\begin{equation*}
a\psi ( 1) =\psi ( a1) =\psi ( a) =a,
\end{equation*}
and similarly $\psi ( 1) a=a$. This shows that $\psi ( 1) =1$ and that 
\begin{equation*}
\psi \circ \tilde{\theta}\colon \tilde{A}\rightarrow \overline{A^{\alpha
}\cdot \prod\nolimits_{\mathcal{U}}^{G}\tilde{A}\cdot A^{\alpha }}\subseteq
\prod\nolimits_{\mathcal{U}}^{G}\tilde{A}
\end{equation*}%
is the diagonal embedding. We also have, for $j=0,1,\ldots ,d$. 
\begin{equation*}
( a+\lambda 1) \psi _{j}( b+\mu 1) =a\psi _{j}( b+\mu 1) +\lambda \psi _{j}(
b+\mu 1) =\psi _{j}( ab+\mu a+\lambda b+\lambda \mu 1) =\psi _{j}( (
a+\lambda 1) ( b+\mu 1) ) \text{.}
\end{equation*}%
Therefore $\psi _{j}$ is an $\tilde{A}$-bimodule map. Hence $\psi
_{0},\ldots ,\psi _{d}$ witness that $\dim _{\mathrm{oz}}^{G}(\tilde{\theta}%
)\leq d$, so $\dim _{\mathrm{oz}}^{G}(\tilde{\theta})\leq \dim _{\mathrm{oz}%
}^{G}(\theta)$.

Conversely, suppose that $\dim _{\mathrm{oz}}^{G}(\tilde{\theta})\leq d$.
Observe that $A^{\alpha }$ is a closed two-sided ideal of the fixed point
algebra of $(\tilde{A},\alpha )$.\ Therefore, by Remark \ref%
{Remark:nondegenerate}, there exist $G$-equivariant completely positive
contractive order zero $A$-bimodule maps $\psi _{0},\ldots ,\psi _{d}\colon 
\tilde{B}\rightarrow \overline{A^{\alpha }\cdot \prod_{\mathcal{U}}^{G}%
\tilde{A}\cdot A^{\alpha }}$. Note that $\overline{A^{\alpha }\cdot \prod_{%
\mathcal{U}}^{G}\tilde{A}\cdot A^{\alpha }}=\overline{A^{\alpha }\cdot
\prod_{\mathcal{U}}^{G}A\cdot A^{\alpha }}$. Therefore the restriction of
the maps $\psi _{0},\ldots ,\psi _{d}$ to $B$ witness that $\dim _{\mathrm{oz%
}}^{G}( \theta ) \leq d$, so $\dim _{\mathrm{oz}}^{G}(\theta)\leq \dim _{%
\mathrm{oz}}^{G}(\tilde{\theta})$, as desired.

(3): Suppose that $\dim _{\mathrm{\mathrm{oz}}}^{G}(\theta )=d$. Let us
first consider the case when $A,B$ are unital C*-algebras, in which case $%
\theta $ is a unital *-homomorphism. Fix a cardinal $\kappa $ larger than
the density characters of $A$, $B$, and $L^{2}(G)$, and a countably
incomplete $\kappa $-good ultrafilter $\mathcal{U}$. Let $\psi _{0},\ldots
,\psi _{d}\colon B\rightarrow \prod_{\mathcal{U}}^{G}A$ be maps as in the
definition of $\dim _{\mathrm{oz}}^{G}(\theta )\leq d$. Consider the
following maps:

\begin{itemize}
\item the $\check{G}$-equivariant completely positive contractive $A$%
-bimodule maps $G\ltimes \psi _{j}\colon G\ltimes _{\beta }B\rightarrow
G\ltimes _{\alpha }(\prod_{\mathcal{U}}^{G}A)$ for $j=0,\ldots ,d$ obtained
as in Lemma \ref{Lemma:crossed-oz};

\item the $\check{G}$-equivariant *-homomorphism $G\ltimes \theta \colon
G\ltimes _{\alpha }A\rightarrow G\ltimes _{\beta }B$; and

\item the $\check{G}$-equivariant injective *-homomorphism $\Psi \colon
G\ltimes _{\alpha }(\prod_{\mathcal{U}}^{G}A)\rightarrow \prod_{\mathcal{U}%
}^{\check{G}}A$ from Proposition \ref{Proposition:ultra-crossed-compact}.
\end{itemize}

Then $G\ltimes \psi =\sum_{j=0}^{d}G\ltimes \psi _{j}$ is a $\check{G}$%
-equivariant completely positive $A$-bimodule map satisfying 
\begin{equation*}
\Psi \circ (G\ltimes \psi )\circ (G\ltimes \theta )=\Delta _{A}\colon
A\rightarrow \prod\nolimits_{\mathcal{U}}^{\check{G}}A.
\end{equation*}%
Since $G\ltimes \psi $ is completely positive, we have $\Vert G\ltimes \psi
\Vert =\Vert (G\ltimes \psi )(1)\Vert =1$. Therefore $\dim _{\mathrm{oz}}^{%
\check{G}}(G\ltimes \theta )\leq d$.

Consider now the case when $A$ and $B$ are not necessarily unital. Then $%
\dim _{\mathrm{oz}}^{G}(\tilde{\theta})\leq d$ by (2). By applying the
result to the unital case, we have $\dim _{\mathrm{oz}}^{\hat{G}}(G\ltimes 
\tilde{\theta})\leq d$. Observe now that $G\ltimes _{\alpha }A\subseteq
G\ltimes _{\alpha }\tilde{A}$ and $G\ltimes _{\beta }B\subseteq G\ltimes 
\tilde{B}$ are positively quantifier-free $\mathcal{L}_{\check{G}%
}^{A^{\alpha }\text{-}A^{\alpha }}$-definable with respect to the class of $%
\check{G}$-C*-algebras of the form $G\ltimes _{\gamma }C$ for some $G$%
-C*-algebra $(C,\gamma )$ such that $(A,\alpha )$ embeds equivariantly into $%
(C,\gamma )$. Therefore, by part (5) of Proposition \ref%
{Proposition:properties-oz-dim}, the restriction $G\ltimes \theta \colon
G\ltimes _{\alpha }A\rightarrow G\ltimes _{\beta }B$ also has $\check{G}$%
-equivariant order zero dimension at most $d$ in view of the semantic
characterization of $\check{G}$-equivariant order zero dimension from\
Remark \ref{Remark:semantic}. We conclude that $\dim _{\mathrm{\mathrm{oz}}%
}^{\check{G}}(G\ltimes \theta )\leq \dim _{\mathrm{\mathrm{oz}}}^{G}(\theta
) $.

(4): We can assume that $A\subseteq B$ is a nondegenerate $G$-C*-subalgebra,
and $\theta :A\rightarrow B$ is the inclusion map. Let $x\in C(G)\otimes A$,
and $\varepsilon >0$. Using freeness for $\beta $, find $n\in \mathbb{N}$,
and tuples $b_{1},\ldots ,b_{n},c_{1},\ldots ,c_{n}\in B$ with%
\begin{equation*}
\Vert x-\sum_{k=1}^{n}\beta (b_{k})(1\otimes c_{k})\Vert <\varepsilon /2%
\text{.}
\end{equation*}%
Set $M=\max_{k=1,\ldots ,n}\Vert b_{k}\Vert $ and $d=\dim _{\mathrm{oz}%
}^{G}(\theta )<\infty $. Since $A\subseteq B$ is nondegenerate, choose $u\in
A_{+}$ satisfying 
\begin{equation*}
\Vert x(1\otimes u)-x\Vert <\varepsilon /3,\ \mbox{ and }\ \ \Vert
c_{k}u-c_{k}\Vert <\varepsilon /(3(d+1)nM)
\end{equation*}%
for all $k=1,\ldots ,n$.

Let $\kappa $ be a cardinal larger than the density characters of $A$, $B$,
and $L^{2}(G)$. Find $G$-equivariant completely positive order zero maps $%
\psi _{0},\ldots ,\psi _{d}\colon B\rightarrow \prod_{\mathcal{U}}^{G}A$ as
in the definition of order zero dimension. For $j=0,\ldots,d$, let $%
\pi_j\colon B\rightarrow \prod_{\mathcal{U}}^{G}A$ be the completely
positive contractive order zero map given by $\pi_j=\psi_j^{1/2}$, using
functional calculus for order zero maps (see \cite[Corollary 3.2]%
{winter_completely_2009}). It is easily checked that $\pi_j$ is again a $G$%
-equivariant $A$-bimodule map.

For $j=0,\ldots ,d$ and $k=1,\ldots ,n$, set $b_{j,k}=\pi _{j}(b_{k})$ and $%
c_{j,k}=\psi _{j}(c_{k})$. For $j=0,\ldots ,d$, denote by $\phi _{j}\colon
C(G)\otimes B\rightarrow C(G)\otimes \prod_{\mathcal{U}}^{G}A$ the $G$%
-equivariant completely positive order zero $A$-bimodule map $\phi _{j}=%
\mathrm{id}_{C(G)}\otimes \pi _{j}$. In the rest of this proof, for elements 
$e$ and $f$ in some C*-algebra and $\delta >0$, we write $e\approx _{\delta
}f$ to mean $\Vert e-f\Vert <\delta $. By the comments before this
proposition, we have%
\begin{align*}
\phi _{j}(\beta (b_{k})(1\otimes c_{k}))\phi _{j}(1\otimes u)& =\phi
_{j}(\beta (b_{k}))\phi _{j}(1\otimes c_{k}u) \\
& =\alpha _{\mathcal{U}}(b_{j,k})(1\otimes c_{j,k}u) \\
& \approx _{\varepsilon /(3(d+1)n)}\alpha _{\mathcal{U}}(b_{j,k})(1\otimes
c_{j,k}).
\end{align*}%
In the following computation, we use the observation above at the first
step; the definition of the elements $b_{k}$ and $c_{k}$ at the third; the
fact that $u$ has a square root at the fourth, together with the comments
before this proposition regarding order zero maps; the fact that $\phi _{j}$
is a $(C(G)\otimes A)$-bimodule map at the fifth; the definition of $\phi
_{j}$ at the seventh; and the properties of the maps $\psi _{j}$ at the
eighth: 
\begin{align*}
\sum_{j=0}^{d}\sum_{k=1}^{n}\alpha _{\mathcal{U}}(b_{j,k})(1\otimes
c_{j,k})& \approx _{\varepsilon /3}\sum_{j=0}^{d}\sum_{k=1}^{n}\phi
_{j}(\beta (b_{k})(1\otimes c_{k}))\phi _{j}(1\otimes u)=\sum_{j=0}^{d}\phi
_{j}\left( \sum_{k=1}^{n}\beta (b_{k})(1\otimes c_{k})\right) \phi
_{j}(1\otimes u) \\
& \approx _{\varepsilon /3}\sum_{j=0}^{d}\phi _{j}(x)\phi _{j}(1\otimes
u)=\sum_{j=0}^{d}\phi _{j}(x(1\otimes u^{1/2}))\phi _{j}(1\otimes u^{1/2}) \\
& =x\sum_{j=0}^{d}\phi _{j}(1\otimes u^{1/2})\phi _{j}(1\otimes
u^{1/2})=x\sum_{j=0}^{d}\phi _{j}^{2}(1\otimes u) \\
& =x\sum_{j=0}^{d}(\mathrm{id}_{C(G)}\otimes \psi )(1\otimes u)=x(1\otimes
u)\approx _{\varepsilon /3}x\text{.}
\end{align*}%
We conclude that $\Vert x-\sum_{j=0}^{d}\sum_{k=1}^{n}\alpha _{\mathcal{U}%
}(b_{j,k})(1\otimes c_{j,k})\Vert <\varepsilon $. Since the diagonal
embedding of $A\rightarrow \prod_{\mathcal{U}}^{G}A$ is positively $\mathcal{%
L}_{G}^{\text{C*}}$-existential, we conclude that there exist $\tilde{b}%
_{j,k},\tilde{c}_{j,k}\in A$ satisfying%
\begin{equation*}
\Vert x-\sum_{k=1}^{n}\sum_{j=0}^{d}\alpha (\tilde{b}_{j,k})(1\otimes \tilde{%
c}_{j,k}){}||{}<\varepsilon \text{.}
\end{equation*}%
This shows that $(A,\alpha )$ is free.
\end{proof}

\begin{theorem}
\label{Theorem:properties-oz-dim-3} Let $G$ be a compact quantum group, and
let $\theta \colon (A,\alpha )\rightarrow (B,\beta )$ be a \emph{%
nondegenerate} injective *-homomorphism. Then:

\begin{enumerate}
\item The $G$-equivariant order zero dimension of $\mathrm{id}_{\mathbb{K}%
_{G}}\otimes \theta \colon (\mathbb{K}_{G}\otimes A,\alpha _{\mathbb{K}%
})\rightarrow (\mathbb{K}_{G}\otimes B,\beta _{\mathbb{K}})$ is equal to $%
\dim _{\mathrm{oz}}^{G}(\theta )$.

\item The $\check{G}$-equivariant order zero dimension of $G\ltimes \theta
\colon (G\ltimes _{\alpha ,\mathrm{r}}A,\check{\alpha})\rightarrow (G\ltimes
_{\beta ,\mathrm{r}}B,\check{\beta})$ is equal to $\dim _{\mathrm{oz}%
}^{G}(\theta )$.

\item The $G$-equivariant order zero dimension of $\theta |_{A^{\alpha }}$
is less than or equal to $\dim _{\mathrm{oz}}^{G}(\theta )$.
\end{enumerate}
\end{theorem}

\begin{proof}
(1): The proof of this fact is identical to the proof of part (2) of
Proposition \ref{Proposition:crossed2}.

(2): One inequality follows from Theorem \ref{Theorem:properties-oz-dim-2}.
As in the proof of part (3) of Proposition \ref{Proposition:crossed2}, one
can now deduce that in fact equality holds using Item (1) above and the
Baaj--Skandalis--Takesaki--Takai duality for compact quantum groups.

(3): This is a particular instance of part (5) of Proposition \ref%
{Proposition:properties-oz-dim} in the case of the fixed point subalgebra,
which is an $\mathcal{L}_{G}^{\mathrm{osos}}$-definable $G$-C*-subalgebra
relative to the class of $G$-C*-algebras when $G$ is a compact quantum group.
\end{proof}

The following preservation result for *-homomorphisms with finite order zero
dimension has been established in \cite[Proposition 5.22]%
{gardella_equivariant_2016}. Given a C*-algebra $A$, we let $\dim _{\mathrm{%
nuc}}( A) $ be the nuclear dimension of $A$ \cite{winter_nuclear_2010}, and $%
\mathrm{dr}( A) $ be the decomposition rank of $A$ \cite[Definition 3.1]%
{kirchberg_covering_2004}.

\begin{proposition}
\label{Proposition:oz-nuclear} Let $A,B$ be C*-algebras, and $\theta \colon
A\rightarrow B$ be a *-homomorphism. Then%
\begin{equation*}
\dim _{\mathrm{nuc}}( A) +1\leq ( \dim _{\mathrm{oz}}( \theta ) +1) ( \dim _{%
\mathrm{nuc}}( B) +1)
\end{equation*}%
and%
\begin{equation*}
\mathrm{dr}( A) +1\leq ( \dim _{\mathrm{oz}}( \theta ) +1) ( \mathrm{dr}( B)
+1) \text{.}
\end{equation*}
\end{proposition}

More generally, Proposition \ref{Proposition:oz-nuclear} applies to any
dimension function for (nuclear) C*-algebras that is (nuclearly) positively $%
\forall \exists $-axiomatizable in the sense of \cite[Definition 5.15,
Definition 5.16.]{gardella_equivariant_2016}.

\subsection{Rokhlin dimension}

In this subsection we fix a compact quantum group $G$. We consider $%
C(G)\otimes A$ as a $G$-C*-algebra with respect to the action given by $%
\Delta \otimes \mathrm{id}_{A}$.

\begin{definition}
\label{Definition:rokhlin-dimension} The \emph{Rokhlin dimension} $\dim_{%
\mathrm{Rok}}(A,\alpha)$ of a $G$-C*-algebra $( A,\alpha )$, is the $G$%
-equivariant order zero dimension of $\alpha \colon ( A,\alpha ) \rightarrow
( C( G) \otimes A,\Delta \otimes \mathrm{id}_{A}) $.
\end{definition}

It follows from \cite[Lemma 1.24]{barlak_spatial_2017} and \cite[Lemma 5.13]%
{gardella_equivariant_2016} that when $G$ is classical, Definition \ref%
{Definition:rokhlin-dimension} recovers the usual notion of Rokhlin
dimension for $G$-C*-algebras from \cite[Definition 3.2]%
{gardella_rokhlin_2016}.

The following is the main technical fact for actions with finite Rokhlin
dimension.

\begin{theorem}
\label{Theorem:preservation-rokhlin-dim} Let $(A,\alpha )$ be a $G$%
-C*-algebra. Then 
\begin{equation*}
\dim _{\mathrm{oz}}(A^{\alpha }\hookrightarrow A)\leq \dim _{\mathrm{Rok}%
}(A,\alpha )=\dim _{\mathrm{oz}}^{\check{G}}(G\ltimes _{\alpha
}A\hookrightarrow \mathbb{K}_{\check{G}}\otimes A)\text{.}
\end{equation*}%
Here $\mathbb{K}_{\check{G}}\otimes A$ is regarded as a $\check{G}$%
-C*-algebra endowed with the stabilization of the trivial action of $\check{G%
}$ on $A$.
\end{theorem}

\begin{proof}
Let $d$ be the Rokhlin dimension of $\left( A,\alpha \right) $. We consider $%
\alpha $ as a $G$-equivariant *-homomorphism $\alpha :\left( A,\alpha
\right) \rightarrow \left( C\left( G\right) \otimes A,\Delta \otimes \mathrm{%
id}_{A}\right) $. We thus have that $d$ is equal to the $G$-equivariant
order zero dimension of $\alpha $. By part (3) of Theorem \ref%
{Theorem:properties-oz-dim-3}, $\dim _{\mathrm{oz}}^{G}\left( \alpha
|_{A^{\alpha }}\right) \leq \dim _{\mathrm{\mathrm{oz}}}^{G}\left( \alpha
\right) =\dim _{\mathrm{Rok}}\left( A,\alpha \right) $. The restriction $%
\alpha |_{A^{\alpha }}$ is the embedding $A^{\alpha }\hookrightarrow
1\otimes A\subseteq C(G)\otimes A$, so this proves the first equality.

By part~(2) of Theorem~\ref{Theorem:properties-oz-dim-3}, the $\check{G}$%
-equivariant *-homomorphism $G\ltimes \alpha \colon G\ltimes _{\alpha
}A\rightarrow G\ltimes _{\Delta \otimes \mathrm{id}_{A}}C(G)\otimes A$
induced by $\alpha $ is equal to $d$. A particular instance of the
Baaj--Skandalis--Takesaki--Takai duality for compact quantum groups yields a 
$\check{G}$-equivariant *-isomorphism $\rho \colon G\ltimes _{\Delta \otimes 
\mathrm{id}_{A}}(C(G)\otimes A)\rightarrow \mathbb{K}_{G}\otimes A$ such
that $\rho \circ (G\ltimes \alpha )\colon G\ltimes _{\alpha }A\rightarrow 
\mathbb{K}_{G}\otimes A$ is the canonical inclusion. This shows the $\check{G%
}$-equivariant *-homomorphism $G\ltimes _{\alpha }A\hookrightarrow \mathbb{K}%
_{G}\otimes A$ also has $\check{G}$-equivariant order zero dimension $d$.
\end{proof}

\begin{corollary}
Suppose that $G$ is a compact quantum group, and $( A,\alpha ) $ is a $G$%
-C*-algebra. Then%
\begin{equation*}
\dim _{\mathrm{nuc}}( A^{\alpha }) +1\leq ( \dim _{\mathrm{Rok}}( A,\alpha )
+1) ( \dim _{\mathrm{nuc}}( A) +1)
\end{equation*}
and 
\begin{equation*}
\mathrm{dr}( G\ltimes _{\alpha }A) +1\leq ( \dim _{\mathrm{Rok}}( A,\alpha )
+1) ( \mathrm{dr}( A) +1) \text{.}
\end{equation*}
\end{corollary}

\begin{proof}
This is an immediate consequence of Theorem \ref%
{Theorem:preservation-rokhlin-dim} and Proposition \ref%
{Proposition:oz-nuclear}.
\end{proof}

\begin{remark}
Let $\theta \colon (A,\alpha )\rightarrow (B,\beta )$ be a $G$-equivariant
*-homomorphism. Using part~(1) of~Proposition \ref%
{Proposition:properties-oz-dim} at the fourth step, we get%
\begin{align*}
\dim _{\mathrm{Rok}}(A,\alpha )+1 &=\dim _{\mathrm{\mathrm{oz}}}^{G}( \alpha
) +1\leq \dim _{\mathrm{\mathrm{oz}}}^{G}( ( \mathrm{id}\otimes \theta )
\circ \alpha ) +1=\dim _{\mathrm{\mathrm{oz}}}^{G}( \beta \circ \theta ) +1
\\
&\leq (\dim _{\mathrm{\mathrm{oz}}}^{G}( \beta ) +1)(\dim _{\mathrm{\mathrm{%
oz}}}^{G}( \theta ) +1)=(\dim _{\mathrm{Rok}}( B,\beta ) +1)(\dim _{\mathrm{%
\mathrm{oz}}}^{G}( \theta ) +1)\text{.}
\end{align*}
\end{remark}

The following dimensional inequalities follow from the remark above and \cite%
[Theorem 5.40 and Theorem 4.41]{gardella_equivariant_2016}. We denote by $%
\mathcal{Z}$ the Jiang-Su algebra, and by $\mathcal{O}_2$ and $\mathcal{O}_{\infty }$ the Cuntz algebras on two and infinitely many generators,
respectively.

\begin{theorem}
Let $( A,\alpha ) $ be a $G$-C*-algebra, and let $U$ be UHF-algebra of
infinite type. Then 
\begin{equation*}
\dim _{\mathrm{Rok}}( A\otimes \mathcal{Z},\alpha \otimes \mathrm{id}_{%
\mathcal{Z}}) \leq 2\dim _{\mathrm{Rok}}( A\otimes U,\alpha \otimes \mathrm{%
id}_{U}) +1
\end{equation*}%
and%
\begin{equation*}
\dim _{\mathrm{Rok}}( A\otimes \mathcal{O}_{\infty },\alpha \otimes \mathrm{%
id}_{\mathcal{O}_{\infty }}) \leq 2\dim _{\mathrm{Rok}}( A\otimes \mathcal{O}%
_{2},\alpha \otimes \mathrm{id}_{\mathcal{O}_{2}}) +1\text{.}
\end{equation*}
\end{theorem}

Finally, we show that actions with finite Rokhlin dimension are free. This
result is new even for actions of classical finite groups. Previous partial
results were considerably more technical, and the quantum group perspective
makes the argument significantly more transparent.

\begin{theorem}
Let $(A,\alpha)$ be a $G$-C*-algebra. If $\dim_{\mathrm{Rok}}(\alpha)<\infty$%
, then $\alpha$ is free.
\end{theorem}

\begin{proof}
By Example~\ref{eg:free}, the $G$-C*-algebra $C(G)\otimes A$ is free. Since $%
\alpha\colon A\to C(G)\otimes A$ has finite $G$-equivariant order zero
dimension, the result follows from part~(4) of Theorem~\ref%
{Theorem:properties-oz-dim-2}.
\end{proof}

\subsection{Duality}

One can isolate a notion of dimension for actions of \emph{discrete }quantum
groups, which is dual to the notion of Rokhlin dimension for actions of
compact quantum groups. Suppose that $G$ is a compact quantum group, and let 
$( A,\alpha ) $ be a $\check{G}$-C*-algebra. Consider the reduced crossed
product $\check{G}\ltimes _{\alpha ,\mathrm{r}}A$, and the canonical
inclusion $\iota _{C(G)}\colon C( G) \rightarrow M(\check{G}\ltimes _{\alpha
,\mathrm{r}}A)$. Denote by $V$ the fundamental unitary of $G$. Consider now
the unitary 
\begin{equation*}
V_{\alpha }:=(\mathrm{id}_{c_{0}(\check{G})}\otimes \iota _{C(G)})(V)\in
M(c_{0}(\check{G})\otimes (\check{G}\ltimes _{\alpha ,\mathrm{r}}A))\text{.}
\end{equation*}%
Then the map $\mathrm{\mathrm{Ad}}(V_{\alpha }^{\ast })$ turns $M(\check{G}%
\ltimes _{\alpha ,\mathrm{r}}A)$ into a $\check{G}$-C*-algebra.

\begin{definition}
\label{Definition:approximation-dimension}Let $(A,\alpha )$ be a $\check{G}$%
-C*-algebra. The \emph{\ representation dimension } $\dim _{\mathrm{rep}%
}(A,\alpha )$ of $(A,\alpha )$ is the $\check{G}$-equivariant order zero
dimension of the canonical embedding $(A,\alpha )\rightarrow (\check{G}%
\ltimes _{\alpha ,\mathrm{r}}A,\mathrm{\mathrm{Ad}}(V_{\alpha }^{\ast }))$.
\end{definition}

In the case of dimension zero, and when $G$ is coexact and $A$ is separable,
Definition \ref{Definition:approximation-dimension} recovers the notion of
spatial approximate representability from \cite[Definition 4.7]%
{barlak_spatial_2017}. The following result generalizes \cite[Theorem 4.12]%
{barlak_spatial_2017}, which is then the case of dimension zero.

\begin{theorem}
\label{Theorem:duality}Let $G$ be a compact quantum group and let $(A,\alpha
)$ be a $G$-C*-algebra. Then 
\begin{equation*}
\dim _{\mathrm{Rok}}(A,\alpha )=\dim _{\mathrm{rep}}(G\rtimes _{\alpha ,r}A,%
\check{\alpha}).
\end{equation*}
\end{theorem}

\begin{proof}
Using the identifications from \cite[Proposition 4.10, Proposition 4.11]%
{barlak_spatial_2017}, the conclusion follows from part~(2) of Theorem \ref%
{Theorem:properties-oz-dim-3} applied to $\alpha \colon A\rightarrow
C(G)\otimes A$.
\end{proof}

\setcounter{section}{0} %\setcounter{theorem}{0}
\setcounter{equation}{0} \renewcommand{\theequation}{\thesection.%
\arabic{equation}} \setcounter{figure}{0} \setcounter{table}{0}

\appendix

\section*{Appendix}

\renewcommand{\thesection}{A}

\setcounter{theorem}{0}

We recall here the fundamental notions concerning first order logic for
metric structure. A good introduction to this subject can be found in \cite%
{ben_yaacov_model_2008}. The systematic study of C*-algebras from the
perspective of model theory has been undertaken in \cite{farah_model_2016}.
We will consider the framework of languages with domains of quantification
as introduced in \cite{farah_model_2014}, which is particularly suitable for
dealing with structures from functional analysis. In fact, we will consider
a slightly more general setting, which is necessary for our purposes, where
the interpretation of the domains of quantifications of a given sort are
only required to be \emph{dense }in the interpretation of the sort. We will
also consider the situation where, rather than a single distinguished metric
single, the language contains an upward directed collection of \emph{%
pseudometric }symbols. This is useful when considering structures such as
modules and bimodules. It is clear that all the results of \cite%
{farah_model_2014} go through in this slightly more general setting.

\subsection{Syntax}

A \emph{language }$\mathcal{L}$ is given by:

\begin{itemize}
\item a collection of \emph{sorts} $\mathcal{S}$,

\item a collection of function symbols $f$,

\item a collection of relation symbols $R$,

\item for each sort $\mathcal{S}$, an ordered, \emph{upward directed }%
collection of \emph{domains of quantification} $D$ for $\mathcal{S}$,

\item for each sort $\mathcal{S}$, a collection of \emph{variables }of sort $%
\mathcal{S}$,

\item for each sort $\mathcal{S}$, an ordered, \emph{upward directed}
collection of \emph{pseudometric symbols} $d_{\mathcal{S}}$ of arity $2$
with input sorts all equal to $\mathcal{S}$,

\item for each function and relation symbol $B$, a natural number $n_{B}$
(its \emph{arity}),

\item for each function symbol $f$, a distinguished tuple $\mathcal{S}%
_{1}^{f},\ldots ,\mathcal{S}_{n_{f}}^{f}$ of \emph{input sorts} and an \emph{%
output sort} $\mathcal{S}^{f}$,

\item for each relation symbol $R$, a distinguished tuple $\mathcal{S}%
_{1}^{R},\ldots ,\mathcal{S}_{n_{R}}^{R}$ of \emph{input sorts},

\item for each function symbol $f$ and for each tuple $D_{1},\ldots
,D_{n_{f}}$ of domains of quantifications for $\mathcal{S}_{1}^{f},\ldots ,%
\mathcal{S}_{n_{f}}^{f}$, and for each tuple of pseudometric symbols $%
d_{1},\ldots ,d_{n_{f}}$ of sort $\mathcal{S}_{1},\ldots ,\mathcal{S}%
_{n_{f}} $, distinguished domains of quantification $D_{D_{1},\ldots
,D_{n_{f}}}^{f}$ and a distinguished pseudometric symbol $d_{D_{1},\ldots
,D_{n_{f}},d_{1},\ldots ,d_{n_{f}}}^{f}$ for the output sort $\mathcal{S}%
^{f} $ of $f$, and a distinguished \emph{continuity modulus} $\varpi
_{D_{1},\ldots ,D_{n_{f}},d_{1},\ldots ,d_{n_{f}}}^{f}$;

\item for each relation symbol $R$ and for each tuple $D_{1},\ldots
,D_{n_{R}}$ of domains of quantifications for $\mathcal{S}_{1}^{R},\ldots ,%
\mathcal{S}_{n_{f}}^{R}$, and for each tuple of pseudometric symbols $%
d_{1},\ldots ,d_{n_{f}}$ of sort $\mathcal{S}_{1},\ldots ,\mathcal{S}%
_{n_{f}} $, a distinguished compact interval $D_{D_{1},\ldots
,D_{n_{R}}}^{R} $ in $\mathbb{R}$ and a distinguished \emph{continuity
modulus} $\varpi _{D_{1},\ldots ,D_{n}}^{R}$.
\end{itemize}

A function symbol $f$ is allowed to have arity $n_{f}=0$, in which case it
is called a \emph{constant symbol}. Given a language $\mathcal{L}$, one can
define the notion of $\mathcal{L}$-term by recursion. Formally, one declares
that variables are terms, and if $t_{1},\ldots ,t_{n}$ are $\mathcal{L}$%
-terms and $f$ is an $n$-ary function symbol in $\mathcal{L}$, then $f(
t_{1},\ldots ,t_{n}) $ is an $\mathcal{L}$-term. If $t$ is an $\mathcal{L}$%
-term and $x_{1},\ldots ,x_{k}$ are the variables that appear in $t$, then
one also writes $t$ as $t( x_{1},\ldots ,x_{k}) $.

Formulas are defined starting from terms by recursion. A \emph{basic }$%
\mathcal{L}$-\emph{formula} in the free variables $\overline{x}=(
x_{1},\ldots ,x_{k}) $ is an expression $\varphi ( \overline{x}) $ of the
form $R( t_{1},\ldots ,t_{n}) $ where $t_{1}( \overline{x}) ,\ldots ,t_{n}( 
\overline{x}) $ are $\mathcal{L}$-terms and $R$ is an $n$-ary relation
symbol in $\mathcal{L} $. A \emph{quantifier-free }$\mathcal{L}$-formula is
an expression $\varphi ( \overline{x}) $ of the form%
\begin{equation*}
q( \psi _{1}( \overline{x}) ,\ldots ,\psi _{n}( \overline{x}) )
\end{equation*}%
where $\psi _{1},\ldots ,\psi _{n}$ are basic $\mathcal{L}$-formulas in the
free variables $\overline{x}$ and $q\colon \mathbb{R}^{n}\rightarrow \mathbb{%
R}$ is a continuous function. Such a quantifier-free formula is \emph{%
positive }if $q\colon \mathbb{R}^{n}\rightarrow \mathbb{R}$ has the property
that $s_{i}\leq t_{i}$ for $i=1,2,\ldots ,n$ implies that $q( s_{1},\ldots
,s_{n}) \leq q( t_{1},\ldots ,t_{n}) $. If furthermore $q$ is of the form $%
q( z_{1},\ldots ,z_{n}) =\max \left\{ u_{1}( z_{1}) ,\ldots ,u_{n}( z_{n})
\right\} $ where $u_{1},\ldots ,u_{n}\colon \mathbb{R}\rightarrow \mathbb{R}$
are continuous nondecreasing functions, then we say that $\varphi ( 
\overline{x}) $ is a \emph{positive primitive }quantifier-free $\mathcal{L}$%
-formula.

A (positive/positive primitive) \emph{existential} $\mathcal{L}$-formula in
the free variables $x_{1},\ldots ,x_{k}$ is an expression 
of the form $\varphi (
x_{1},\ldots ,x_{k})=\inf_{y_{1}}\cdots \inf_{y_{\ell }}\psi
( x_{1},\ldots ,x_{k},y_{1},\ldots ,y_{\ell }) $ where $\psi ( \overline{x},%
\overline{y}) $ is a (positive/positive primitive) quantifier-free $\mathcal{%
L}$-formula. A (positive/positive primitive) $\forall \exists $-$\mathcal{L}$%
-formula in the free variables $x_{1},\ldots ,x_{k}$ is an expression $%
\varphi ( x_{1},\ldots ,x_{k}) $ of the form $\sup_{z_{1}}\cdots
\sup_{z_{m}}\inf_{y_{1}}\cdots \inf_{y_{\ell }}\psi ( x_{1},\ldots
,x_{k},y_{1},\ldots ,y_{\ell },z_{1},\ldots ,z_{m}) $ where $\psi ( 
\overline{x},\overline{y}) $ is a (positive/positive primitive)
quantifier-free $\mathcal{L}$-formula. Finally, arbitrary (positive/positive
primitive) $\mathcal{L}$-formulas are defined similarly as above, but
allowing an arbitrary finite number of alternations between $\sup $ and $%
\inf $. A (positive/positive primitive) $\mathcal{L}$-\emph{sentence }is a
(positive/positive primitive) $\mathcal{L}$-formula with no free variables.

For each $\mathcal{L}$-formula $\varphi ( x_{1},\ldots ,x_{k}) $, for each
choice $D_{1},\ldots ,D_{k}$ of domains of quantification for the sorts of $%
x_{1},\ldots ,x_{k}$, and for each choice of pseudometric symbols $%
d_{1},\ldots ,d_{k}$ for the sorts $x_{1},\ldots ,x_{k}$ one can define by
recursion on the complexity of the formula, a uniform continuity modulus $%
\varpi _{D_{1},\ldots ,D_{k},d_{1},\ldots ,d_{k}}^{\varphi }$ for $\varphi $%
, as well as a compact interval $D_{D_{1},\ldots ,D_{k}}^{\varphi }$ in $%
\mathbb{R}$ where $\varphi $ takes values. A collection $\mathcal{F}$ of
formulas in the free variables $x_{1},\ldots ,x_{k}$ is\emph{\ uniformly
equicontinuous} if for any choice $D_{1},\ldots ,D_{k}$ of domains of
quantification for the sorts of $x_{1},\ldots ,x_{k}$, the continuity moduli 
$\left\{ \varpi _{D_{1},\ldots ,D_{k},d_{1},\ldots ,d_{k}}^{\varphi }\colon
\varphi \in \mathcal{F}\right\} $ are uniformly bounded, and the compact
intervals $\left\{ D_{D_{1},\ldots ,D_{k}}^{\varphi }\colon \varphi \in 
\mathcal{F}\right\} $ are all contains in a common compact interval in $%
\mathbb{R}$.

\subsection{Semantics}

An $\mathcal{L}$-structure $M$ if given by:

\begin{itemize}
\item for each sort $\mathcal{S}$, a set $\mathcal{S}^{M}$,

\item for each sort $\mathcal{S}$, and for each pseudometric symbol $d_{%
\mathcal{S}}$ of sort $\mathcal{S}$, $d_{\mathcal{S}}^{M}$ is a pseudometric
on $\mathcal{S}$, such that the uniformity defined by these pseudometrics 
\cite[Definition 2.1]{pachl_uniform_2013} is complete \cite[Definition 1.14]%
{pachl_uniform_2013}, and such that the assignment $d_{\mathcal{S}}\mapsto
d_{\mathcal{S}}^{M}$ is order preserving (where the order for pseudometric
is pointwise comparison);

\item for each domain of quantification $D$ for $\mathcal{S}$, a closed
subset $D^{M}$ of $\mathcal{S}^{M}$ such that the assignment $D\mapsto D^{M}$
is order preserving (where the collection of closed subsets of $\mathcal{S}%
^{M}$ is ordered with respect to inclusion), and such that the union of $%
D^{M}$ when \ $D$ ranges among the domains of quantification for $\mathcal{S}
$ is dense in $\mathcal{S}^{M}$ with respect to the uniformity described
above,

\item for each function symbol $f$ with input sorts $\mathcal{S}%
_{1}^{f},\ldots ,\mathcal{S}_{n_{f}}^{f}$ and output sort $\mathcal{S}^{f}$,
a function $f^{M}\colon (\mathcal{S}_{1}^{f})^{M}\times \cdots \times (%
\mathcal{S}_{n_{f}}^{f})^{M}\rightarrow (\mathcal{S}^{f})^{M}$ (the \emph{%
interpretation }of $f$ in $M$) such that, for any choice of domains of
quantification $D_{1},\ldots ,D_{n_{f}}$ and pseudometric symbols $%
d_{1},\ldots ,d_{n_{f}}$ for the sorts $\mathcal{S}_{1}^{f},\ldots ,\mathcal{%
S}_{n_{f}}^{f}$, the restriction of $f^{M}$ to $D_{1}^{M}\times \cdots
\times D_{n_{f}}^{M}$ is a uniformly continuous map with continuity modulus $%
\varpi _{D_{1},\ldots ,D_{n_{f}},d_{1},\ldots ,d_{n_{f}}}^{f}$ with respect
to the metrics $\max \{d_{1},\ldots ,d_{n_{f}}\}$ and $d_{D_{1},\ldots
,D_{n_{f}},d_{1},\ldots ,d_{n_{f}}}^{f}$, and its range is contained in $%
(D_{D_{1},\ldots ,D_{n_{f}}}^{f})^{M}$;

\item for each relation symbol $R$ with input sorts $\mathcal{S}%
_{1}^{R},\ldots ,\mathcal{S}_{n_{f}}^{R}$ and output sort $\mathcal{S}^{R}$,
a function $R^{M}\colon (\mathcal{S}_{1}^{R})^{M}\times \cdots \times (%
\mathcal{S}_{n_{R}}^{R})^{M}\rightarrow \mathbb{R}$ (the \emph{interpretation%
} of $R$ in $M$) such that, for any choice of domains of quantification $%
D_{1},\ldots ,D_{n_{f}}$ and pseudometric symbols $d_{1},\ldots ,d_{n_{f}}$
for $\mathcal{S}_{1}^{R},\ldots ,\mathcal{S}_{n_{R}}^{R}$, the restriction
of $f^{M}$ to $D_{1}^{M}\times \cdots \times D_{n_{f}}^{M}$ is a uniformly
continuous map with continuity modulus $\varpi _{D_{1},\ldots
,D_{n_{f}},d_{1},\ldots ,d_{n_{f}}}^{R}$ and its range is contained in $%
D_{D_{1},\ldots ,D_{n_{R}}}^{R}$.
\end{itemize}

Suppose now that $t( x_{1},\ldots ,x_{k}) $ is an $\mathcal{L}$-term in the
variables $x_{1},\ldots ,x_{k}$ of sorts $\mathcal{S}_{1},\ldots ,\mathcal{S}%
_{k}$, and that $M$ is an $\mathcal{L}$-structure. Then one can define by
recursion on the complexity of $t$ the output sort $\mathcal{S}$ of $t$ and
the \emph{interpretation }$t^{M}$ of $t$ in $M$, which is a function $%
t^{M}\colon \mathcal{S}_{1}^{M}\times \cdots \times \mathcal{S}%
_{k}^{M}\rightarrow \mathcal{S}^{M}$. Similarly, for any $\mathcal{L}$%
-formula $\varphi ( x_{1},\ldots ,x_{k}) $ in the free variables $%
x_{1},\ldots ,x_{k}$ of sorts $\mathcal{S}_{1},\ldots ,\mathcal{S}_{k}$, one
can defined by recursion on the complexity of $\varphi $ its interpretation $%
\varphi ^{M}$ in $M$, which is a function $\varphi ^{M}\colon \mathcal{S}%
_{1}^{M}\times \cdots \times \mathcal{S}_{k}^{M}\rightarrow \mathbb{R}$. For
any choice of domains of quantification $D_{1},\ldots ,D_{k}$ for the sorts
of $x_{1},\ldots ,x_{k}$, the restriction of $\varphi ^{M}$ to $%
D_{1}^{M}\times \cdots \times D_{k}^{M}$ is uniformly continuous with the
continuity modulus $\varpi _{D_{1},\ldots ,D_{k}}^{\varphi }$ as in the
previous section,\ and takes values in the compact interval $D_{D_{1},\ldots
,D_{k}}^{\varphi }$ of $\mathbb{R}$.

Suppose now that $\mathcal{C}$ is a class of $\mathcal{L}$-structures. Then $%
\mathcal{C}$ defines a metric in the class of $\mathcal{L}$-formulas by
setting, for every $\mathcal{L}$-formulas $\varphi ( x_{1},\ldots ,x_{k})
,\psi ( x_{1},\ldots ,x_{k}) $,%
\begin{equation*}
d_{\mathcal{C}}( \varphi ,\psi ) =\sup \left\vert \varphi ( \overline{a})
-\psi ( \overline{a}) \right\vert
\end{equation*}%
where $M$ ranges among all the $\mathcal{L}$-structures in $\mathcal{C}$,
and $\overline{a}$ ranges among all the tuples in $M$ of the correct sorts.

\begin{definition}
\label{Definition:separable-language}The language $\mathcal{L}$ is \emph{%
separable }for a class of $\mathcal{L}$-structures\emph{\ }$\mathcal{C}$ if
the space of $\mathcal{L}$-formulas endowed with the metric associated with $%
\mathcal{C}$ as above is separable. More generally, we define the \emph{%
density character }of $\mathcal{L}$ for $\mathcal{C}$ to be the density
character of the space of $\mathcal{L}$-formulas with respect to the metric $%
d_{\mathcal{C}}$.
\end{definition}

\begin{definition}
\label{Definition:axiomatizable}A class $\mathcal{C}$ of $\mathcal{L}$%
-structures is \emph{axiomatizable} in the language $\mathcal{L}$ if there
exists a collection $\mathcal{A}$ of $\mathcal{L}$-conditions of the form $%
\sigma \leq r$ (\emph{axioms}) where $\sigma $ is an $\mathcal{L}$-sentence
and $r\in \mathbb{R}$, such that for any $\mathcal{L}$-structure $M$, $M$
belongs to $\mathcal{C}$ if and only if $\sigma ^{M}\leq r$ for any
condition in $\mathcal{A}$.
\end{definition}

\begin{definition}
\label{Definition:pp-axiomatizable}A class of $\mathcal{L}$-structures is
(positively/positively primitively)\emph{\ }$\forall \exists $-\emph{%
axiomatizable} in the language $\mathcal{L}$ if it is axiomatizable by a
collection of $\mathcal{L}$-conditions of the form $\sigma \leq r$ where $%
\sigma $ is a (positive/positive primitive)\emph{\ }$\forall \exists $-$%
\mathcal{L}$-sentence and $r\in \mathbb{R}$.
\end{definition}

The following notion has been essentially introduced in \cite[Definition
5.7.1]{farah_model_2016}.

\begin{definition}
\label{Definition:pp-definable-collection}A class $\mathcal{C}$ of $\mathcal{%
L}$-structures is\emph{\ definable by a uniform collection of positive
existential} $\mathcal{L}$-\emph{formulas} if for every choice of sorts $%
\mathcal{S}_{1},\ldots ,\mathcal{S}_{k}$ in $\mathcal{L}$ and domains of
quantification $D_{1},\ldots ,D_{k}$ for $\mathcal{S}_{1},\ldots ,\mathcal{S}%
_{k}$, there exist and a uniformly equicontinuous family $\mathcal{F}(
D_{1},\ldots ,D_{k}) $ of positive existential $\mathcal{L}$-formulas in the
free variables $x_{1},\ldots ,x_{k}$ of sorts $\mathcal{S}_{1},\ldots ,%
\mathcal{S}_{k}$ such that, for every $\mathcal{L}$-structure $M$, the
following assertions are equivalent:

\begin{enumerate}
\item $M$ belongs to $\mathcal{C}$;

\item for every choice of sorts $\mathcal{S}_{1},\ldots ,\mathcal{S}_{k}$ in 
$\mathcal{L}$, of domains of quantification $D_{1},\ldots ,D_{k}$ for $%
\mathcal{S}_{1},\ldots ,\mathcal{S}_{k}$, of elements $a_{i}\in D_{i}^{M}$
for $i=1,2,\ldots ,k$, and for every $\varepsilon >0$, there exists a
formula $\varphi ( x_{1},\ldots ,x_{k}) $ in $\mathcal{F}( D_{1},\ldots
,D_{k}) $ such that $M\models \varphi ( a_{1},\ldots ,a_{k}) \leq
\varepsilon $.
\end{enumerate}
\end{definition}

It is clear that, if a class of $\mathcal{L}$-structures is positively $%
\forall \exists $-axiomatizable in the language $\mathcal{L}$, then in
particular it is definable by a uniform collection of positive existential $%
\mathcal{L}$-formulas.

\subsection{Reduced products and ultraproducts}

Reduced products and ultraproducts are a fundamental construction in model
theory. We recall here these notions and their basic properties. Suppose
that $\mathcal{L}$ is a language as above, $I$ is an index set, $M_{i}$ for $%
i\in I$ is an $\mathcal{L}$-structure, and $\mathcal{F}$ is a filter over $I$%
. One can define the reduced \emph{$\mathcal{L}$}-product $\prod_{\mathcal{F}%
}M_{i}$ of the $I$-sequence $( M_{i}) _{i\in I}$ with respect to $\mathcal{F}
$ as follows. For each sort $\mathcal{S}$ in $\mathcal{L}$ and domain of
quantification $D$ for $\mathcal{S}$, and for each pseudometric symbol $d$
of sort $\mathcal{S}$, consider the pseudometric $d^{M}$ on $%
\prod_{i}D_{i}^{M}$ defined by $d( ( x_{i}) ,( y_{i}) ) =\limsup_{\mathcal{F}%
}d^{M_{i}}( x_{i},y_{i}) $. The collection of such pseudometrics, when $d$
ranges among the pseudometric symbols of sort $\mathcal{S}$, defines a
uniformity on $\prod_{i}D_{i}^{M}$, and then one can define $D^{M}$ to be
the Hausdorff completion of such a uniform space, endowed with the canonical
pseudometrics induces by the pseudometrics on $\prod_{i}D_{i}^{M}$. We will
denote by $[ x_{i}] _{\mathcal{F}}$ the element of $D^{M}$ associated with
the sequence $( x_{i}) _{i\in I}$, and we will call $( x_{i}) $ a \emph{%
representative sequence }for $[ x_{i}] _{\mathcal{F}}$.

If $D_{1},D_{2}$ are domains of quantification for $\mathcal{S}$ with $%
D_{1}\leq D_{2}$, then one can canonically identify $D_{1}^{M}$ as a closed
subspace of $D_{2}^{M}$. Since the collection of domains of quantification
for $\mathcal{S}$ is directed, the union of $D^{M}$, where $D$ ranges among
all the domains of quantifications for $\mathcal{S}$, is a (possibly
incomplete) metric space. Define then $\mathcal{S}^{M}$ to be the completion
of such a metric space. One can regard $D^{M}$ as a closed subspace of $%
\mathcal{S}^{M}$. Clearly, by definition of $\mathcal{S}^{M}$, the union of $%
D^{M}$, where $D$ ranges among all the quantifications for $\mathcal{S}$, is
dense in $\mathcal{S}^{M}$.

For any function symbol $f$ in $\mathcal{L}$ with input sorts $\mathcal{S}%
_{1}^{f},\ldots ,\mathcal{S}_{n_{f}}^{f}$ and output sort $\mathcal{S}^{f}$,
one can then define a function $f^{M}\colon (\mathcal{S}_{1}^{f})^{M}\times
\cdots \times (\mathcal{S}_{n_{f}}^{f})^{M}\rightarrow (\mathcal{S}^{f})^{M}$
by setting, for every choice of domains of quantification $D_{1},\ldots
,D_{n}$ for $\mathcal{S}_{1}^{f},\ldots ,\mathcal{S}_{n_{f}}^{f}$, and
elements $[ a_{i}^{1}] _{\mathcal{F}},\ldots ,[ a_{i}^{n}] _{\mathcal{F}}$
of $D_{1}^{M},\ldots ,D_{n}^{M}$, 
\begin{equation*}
f( [ a_{i}^{1}] _{\mathcal{F}},\ldots ,[ x_{i}^{n}] _{\mathcal{F}}) =[ f(
a_{i}^{1},\ldots ,a_{i}^{n}) ] _{\mathcal{F}}\in (D_{D_{1},\ldots
,D_{n}}^{f})^{M}\text{.}
\end{equation*}%
Similarly, for any relation symbol $R$ in $\mathcal{L}$ with input sorts $%
\mathcal{S}_{1}^{R},\ldots ,\mathcal{S}_{n_{f}}^{R}$ and output sort $%
\mathcal{S}^{R}$, one can define a function $R^{M}\colon (\mathcal{S}%
_{1}^{R})^{M}\times \cdots \times (\mathcal{S}_{n_{R}}^{R})^{M}\rightarrow 
\mathbb{R}$ by setting, for every choice of domains of quantification $%
D_{1},\ldots ,D_{n}$ for $\mathcal{S}_{1}^{R},\ldots ,\mathcal{S}%
_{n_{R}}^{R} $, and elements $[ a_{i}^{1}] _{\mathcal{F}},\ldots ,[
a_{i}^{n}] _{\mathcal{F}}$ of $D_{1}^{M},\ldots ,D_{n}^{M}$,%
\begin{equation*}
R( [ a_{i}^{1}] _{\mathcal{F}},\ldots ,[ a_{i}^{n}] _{\mathcal{F}})
=\limsup_{\mathcal{F}}R( a_{i}^{1},\ldots ,a_{i}^{n}) \text{.}
\end{equation*}

One obtains in this way an $\mathcal{L}$-structure $M$, which is called the 
\emph{reduced $\mathcal{L}$}-\emph{product }of the $I$-sequence of $\mathcal{%
L}$-structures $( M_{i}) _{i\in I}$ with respect to the filter $\mathcal{F}$%
, and denoted by $\prod_{\mathcal{F}}M_{i}$. In the case when $\mathcal{F}$
is an \emph{ultrafilter}, then $\prod_{\mathcal{F}}M_{i}$ is called \emph{$%
\mathcal{L}$}-\emph{ultraproduct} of the $I$-sequence $( M_{i}) _{i\in I}$.
The function that maps an element $a$ of $M$ to the element of $\prod_{%
\mathcal{F}}M_{i}$ with representative sequence constantly equal to $a$
defines the canonical \emph{diagonal $\mathcal{L}$}-\emph{embedding }of $M$
into $\prod_{\mathcal{F}}M_{i}$. The fundamental properly of reduced \emph{$%
\mathcal{L}$}-products is the following result, known as \L os' theorem; see 
\cite[Proposition 4.3]{farah_model_2014}.

\begin{proposition}
\label{Proposition:Los}Suppose that $\mathcal{F}$ is a filter on a set $I$, $%
( M_{i}) _{i\in I}$ is an $I$-sequence of $\mathcal{L}$-structures, and $%
\varphi ( x_{1},\ldots ,x_{k}) $ is a formula in the free variables $%
x_{1},\ldots ,x_{n}$. If either $\varphi $ is positive primitive
quantifier-free, or $\mathcal{F}$ is an ultrafilter, then for any $%
[a_{i}^{1}],\ldots ,[a_{i}^{k}]$ in $\prod_{\mathcal{F}}M_{i}$ one has%
\begin{equation*}
\varphi ^{\prod_{\mathcal{F}}M_{i}}( [a_{i}^{1}],\ldots ,[a_{i}^{k}])
=\limsup_{\mathcal{F}}\varphi ^{M_{i}}( a_{i}^{1},\ldots ,a_{i}^{k}) \text{.}
\end{equation*}
\end{proposition}

\begin{corollary}
\label{Corollary:Los}Suppose that $\mathcal{C}$ is a class of $\mathcal{L}$%
-structures that is axiomatizable in the language $\mathcal{L}$. Then $%
\mathcal{C}$ is closed under $\mathcal{L}$-ultraproducts. If $\mathcal{C}$
is furthermore positively primitively $\forall \exists $-axiomatizable in
the language $\mathcal{L}$, then $\mathcal{C}$ is closed under reduced $%
\mathcal{L}$-products.
\end{corollary}

\subsection{Types and saturation}

A (closed) $\mathcal{L}$-condition is an expression of the form $\varphi ( 
\overline{x}) \leq r$ for some $\mathcal{L}$-formula $\varphi $ in the free
variables $\overline{x}=( x_{1},\ldots ,x_{k}) $ and $r\in \mathbb{R}$. An $%
\mathcal{L}$-type $p( x_{1},\ldots ,x_{k}) $ in the free variables $%
x_{1},\ldots ,x_{k}$ is a collection of $\mathcal{L}$-conditions $\varphi ( 
\overline{x}) \leq r$. Such a type $p$ is \emph{quantifier-free }if all the
conditions that appear in it involve quantifier-free $\mathcal{L}$-formulas,
and \emph{positive quantifier-free }if all the conditions that appear in it
involve positive primitive quantifier-free $\mathcal{L}$-formulas. A \emph{%
realization }of the $\mathcal{L}$-type $p( \overline{x}) $ in an $\mathcal{L}
$-structure $M$ is a tuple $\overline{a}=( a_{1},\ldots ,a_{k}) $ in $M$
such that each $a_{i}$ belongs to the interpretation of the sort of $x_{i}$,
and $\varphi ^{M}( \overline{a}) \leq r$ for every condition $\varphi ( 
\overline{x}) \leq r$ in $p( \overline{x}) $. In this case, we also write $%
M\models \varphi ( \overline{a}) \leq r$ and $M\models p( \overline{a}) $.
An $\mathcal{L}$-type is \emph{realized }in an $\mathcal{L}$-structure $M$
if it admits a realization in $M$. It is \emph{approximately realized }in $M$
if for any \emph{finite }subset $p_{0}( \overline{x}) $ of $p( \overline{x}) 
$ and for any $\varepsilon >0$, the type $p_{0}^{\varepsilon }( x_{1},\ldots
,x_{k}) $ consisting of the conditions $\varphi ( \overline{x}) \leq
r+\varepsilon $ for any condition $\varphi ( \overline{x}) \leq r$ in $%
p_{0}( x_{1},\ldots ,x_{k}) $, is realized in $M$.

Suppose that $M$ is an $\mathcal{L}$-structure, and $A$ is a subset of $M$.
Then one can consider the language $\mathcal{L}( A) $ obtained by adding to $%
\mathcal{L}$ a constant symbol for each element of $A$. Then $M$ or any $%
\mathcal{L}$-structure containing $M$ can be canonically regarded as an $%
\mathcal{L}( A) $-structure.

\begin{definition}
\label{Definition:saturation}Suppose that $\kappa $ is an uncountable
cardinal. An $\mathcal{L}$-structure $M$ is $\mathcal{L}$-$\kappa $%
-saturated if for every subset $A$ of $M$ of density character less than $%
\kappa $ and for any $\mathcal{L}( A) $-type $p$, if $p$ is approximately
realized in $M$, then it is realized in $M$. Replacing arbitrary types with
positive quantifier-free types gives the notion of positively
quantifier-free $\mathcal{L}$-$\kappa $-saturated structure.
\end{definition}

In the particular case when $\kappa =\aleph _{1}$, $\mathcal{L}$-$\aleph
_{1} $-saturation is also called countable $\mathcal{L}$-saturation, and
positive quantifier-free $\mathcal{L}$-$\kappa $-saturation is also called
positive quantifier-free countable $\mathcal{L}$-saturation. The fundamental
property of ultrapowers of structures with respect to nonprincipal
ultrafilters over $\mathbb{N}$ is that they are countably saturated; see 
\cite[Proposition 4.11]{farah_model_2014}.

\begin{proposition}
\label{Proposition:countable-saturation}Suppose that $\mathcal{F}$ is a
countably incomplete filter. Let $\mathcal{C}$ be a class of $\mathcal{L}$%
-structures such that $\mathcal{L}$ is separable for $\mathcal{C}$.\ If $M$
is an $\mathcal{L}$-structure in $\mathcal{C}$, then the reduced power $%
\prod_{\mathcal{F}}M$ is countably positively quantifier-free $\mathcal{L}$%
-saturated.\ If $\mathcal{F}$ is a countably incomplete ultrafilter, then
the ultrapower $\prod_{\mathcal{U}}M$ is countably $\mathcal{L}$-saturated.
\end{proposition}

Proposition \ref{Proposition:countable-saturation} admits a generalization
to an arbitrary uncountable cardinal $\kappa $. In this more general
setting, one needs to consider (ultra)filters that are moreover $\kappa $-%
\emph{good}. The definition of $\kappa $-good is given in \cite[Section 6.1]%
{chang_model_1977} for ultrafilters, but it applies equally well to filters.
Every countably incomplete filter is $\aleph _{1}$-good. In particular,
every nonprincipal ultrafilter over $\mathbb{N}$ is $\aleph _{1}$-good. The
same proof as \cite[Theorem 6.1.8]{chang_model_1977} gives the following.

\begin{proposition}
\label{Proposition:kappa-saturation}Let $\kappa $ be an uncountable
cardinal, and let $\mathcal{F}$ be a countably incomplete $\kappa $-good
filter. Let $\mathcal{C}$ be a class of $\mathcal{L}$-structures such that $%
\mathcal{L}$ has density character less than $\kappa $ for $\mathcal{C}$. If 
$M$ is an $\mathcal{L}$-structure, then the reduced power $\prod_{\mathcal{F}%
}M$ is positively quantifier-free $\mathcal{L}$-$\kappa $-saturated.\ If $%
\mathcal{U}$ is a countably incomplete $\kappa $-good ultrafilter, then the
ultrapower $\prod_{\mathcal{U}}M$ is $\mathcal{L}$-$\kappa $-saturated.
\end{proposition}

\subsection{Existential embeddings}

Suppose that $\mathcal{L}$ is a language in the logic for metric structures
as above, $M,N$ are $\mathcal{L}$-structures, and $T\colon M\rightarrow N$
is a function. Then $T$ is an $\mathcal{L}$-\emph{morphism} if, for every
sort $\mathcal{S}$ in $\mathcal{L}$ and for every domain of quantification $%
D $ for $\mathcal{S}$, $T$ maps $\mathcal{S}^{M}$ to $\mathcal{S}^{N}$ and $%
D^{M}$ to $D^{N}$, and for any atomic formula $\varphi ( x_{1},\ldots
,x_{k}) $ one has that $\varphi ^{N}( T( a_{1}) ,\ldots ,T( a_{k}) ) \leq
\varphi ^{M}( a_{1},\ldots ,a_{k}) $ for any $a_{1},\ldots ,a_{n}$ of the
same sorts as $x_{1},\ldots ,x_{k}$. An $\mathcal{L}$-morphism is an $%
\mathcal{L}$-\emph{embedding }if for any atomic formula $\varphi (
x_{1},\ldots ,x_{k}) $ one has that $\varphi ^{N}( T( a_{1}) ,\ldots ,T(
a_{k}) ) =\varphi ^{M}( a_{1},\ldots ,a_{k}) $ for any $a_{1},\ldots ,a_{n}$
of the same sorts as $x_{1},\ldots ,x_{k}$. Suppose that $T\colon
M\rightarrow N$ is an $\mathcal{L}$-embedding, and $A$ is a subset of $M$.
Recall that $\mathcal{L}( A) $ is the language obtained from $\mathcal{L}$
by adding a constant symbols $c_{a}$ for any element $a$ of $A$. Then one
can regard both $M$ and $N$ as $\mathcal{L}( A) $-structures, by
interpreting $c_{a}$ as $a$ in $M$ and as $T( a) $ in $N$. We recall here
the notion of (positively) \emph{$\mathcal{L}$}-existential $\mathcal{L}$%
-embedding.

\begin{definition}
\label{Definition:existential-embeddings}Suppose that $M,N$ are $\mathcal{L}$%
-structures, and $T\colon M\rightarrow N$ is an $\mathcal{L}$-embedding.
Then $T$ is a (positively) \emph{$\mathcal{L}$}-\emph{existential }if for
every (positive) quantifier-free $\mathcal{L}( M) $-condition $\varphi (
x_{1},\ldots ,x_{n}) \leq r$ satisfied in $N$ and for every $\varepsilon >0$%
, then $\mathcal{L}( M) $-condition $\varphi ( x_{1},\ldots ,x_{n}) \leq
r+\varepsilon $ is satisfied in $M$.
\end{definition}

The following characterization of \emph{$\mathcal{L}$}-existential \emph{$%
\mathcal{L}$}-embeddings is an immediate consequence of \L os' theorem and
saturation of reduced powers. Recall that any countably incomplete filter is 
$\aleph _{1}$-good.

\begin{proposition}
\label{Proposition:existential-characterize}Fix an uncountable cardinal $%
\kappa $, a class $\mathcal{C}$ of $\mathcal{L}$-structures, and structures $%
M,N$ in $\mathcal{C}$ of density character less than $\kappa $. Suppose that 
$\mathcal{F}$ is a countably incomplete $\kappa $-good filter. Let $T\colon
M\rightarrow N$ be an $\mathcal{L}$-morphism. The following assertions are
equivalent:

\begin{enumerate}
\item $T$ is a positively \emph{$\mathcal{L}$}-existential $\mathcal{L}$%
-embedding;

\item there exists an $\mathcal{L}$-morphism $S\colon N\rightarrow \prod_{%
\mathcal{F}}M$ such that $S\circ T$ is the diagonal $\mathcal{L}$-embedding $%
M\rightarrow \prod_{\mathcal{F}}M$.
\end{enumerate}

Furthermore, if $\mathcal{F}$ is an ultrafilter, then the following
assertions are equivalent:

\begin{enumerate}
\item $T$ is an \emph{$\mathcal{L}$}-existential $\mathcal{L}$-embedding;

\item there exists an $\mathcal{L}$-embedding $S\colon N\rightarrow \prod_{%
\mathcal{U}}M$ such that $S\circ T$ is the diagonal $\mathcal{L}$-embedding $%
M\rightarrow \prod_{\mathcal{U}}M$.
\end{enumerate}
\end{proposition}

The following \emph{preservation result} for positively $\mathcal{L}$%
-existential $\mathcal{L}$-embedding can be easily verified directly.

\begin{proposition}
\label{Proposition:existential-preservation}Suppose that $\mathcal{C}$ is a
collection of $\mathcal{L}$-structures that is definable by a uniform
collection of positive existential $\mathcal{L}$-formulas. Suppose that $A$
and $B$ are $\mathcal{L}$-structures, and $T\colon A\rightarrow B$ is a
positively $\mathcal{L}$-existential $\mathcal{L}$-embedding. If $B$ belongs
to $\mathcal{C}$, then $A$ belongs to $\mathcal{C}$.
\end{proposition}

\subsection{Definability}

Suppose that $\mathcal{C}$ is a class of $\mathcal{L}$-structures. Fix
domains of quantifications $D_{1},\ldots ,D_{n}$ for $\mathcal{L}$ of sorts $%
\mathcal{S}_{1},\ldots ,\mathcal{S}_{n}$. A \emph{uniform assignment }$S$ in 
$D_{1}\times \cdots \times D_{n}$ relative to the class $\mathcal{C}$ is an
assignment $M\mapsto S^{M}$, where $M$ ranges among the $\mathcal{L}$%
-structures in $\mathcal{C}$ and $S^{M}$ is a closed subset $D_{1}^{M}\times
\cdots \times D_{k}^{M}$ ; see \cite[Definition 3.2.1]{farah_model_2016}.
The following definition is equivalent to the one given in \cite[Definition
3.2.1]{farah_model_2016}; see \cite[Theorem 3.2.2]{farah_model_2016}.

\begin{definition}
\label{Definition:definable-set}Suppose that $S$ is a uniform assignment in $%
D_{1}\times \cdots \times D_{n}$ relative to the class $\mathcal{C}$. Then $%
S $ is an $\mathcal{L}$-\emph{definable set} in $D_{1}\times \cdots \times
D_{n}$ relative to the class $\mathcal{C}$ if for every $\varepsilon >0$,
for every choice of pseudometric symbols $d_{1},\ldots ,d_{n}$ of sorts $%
\mathcal{S}_{1},\ldots ,\mathcal{S}_{n}$ there exists an $\mathcal{L}$%
-formula $\varphi ( x_{1},\ldots ,x_{k}) $ in the free variables $%
x_{1},\ldots ,x_{n}$ of sorts $\mathcal{S}_{1},\ldots ,\mathcal{S}_{n}$ such
that, for every $\mathcal{L}$-structure $M$ in $\mathcal{C}$ and for every $%
( a_{1},\ldots ,a_{n}) \in D_{1}\times \cdots \times D_{n}$, 
\begin{equation*}
\left\vert \varphi ^{M}( a_{1},\ldots ,a_{n}) -\inf_{( b_{1},\ldots ,b_{n})
\in S^{M}}\max_{1\leq i\leq n}d_{i}( a_{i},b_{i}) \right\vert <\varepsilon 
\text{.}
\end{equation*}

We say that $S$ is positively existentially $\mathcal{L}$-definable in the $%
\mathcal{L}$-formulas $\varphi $ above can be chosen to be positive
existential.
\end{definition}

Observe that, every choice of pseudometric symbols $d_{1},\ldots ,d_{n}$ of
sort $\mathcal{S}_{1},\ldots ,\mathcal{S}_{n}$ defines a pseudometric $\rho $
on the space of $\mathcal{L}$-definable sets in $D_{1}\times \cdots \times
D_{n}$ relative to the class $\mathcal{C}$ , obtained by setting%
\begin{equation*}
\rho ( S,S^{\prime }) =\sup_{M\in \mathcal{C}}\rho ^{M}( S^{M},S^{\prime M}) 
\text{,}
\end{equation*}%
where $\rho ^{M}$ is the Hausdorff metric on the space of closed subsets of $%
D_{1}^{M}\times \cdots \times D_{n}^{M}$ associated with the metric $d( 
\overline{a},\overline{b}) =\max_{1\leq i\leq n}d_{i}( a_{i},b_{i}) $ on $%
D_{1}^{M}\times \cdots \times D_{n}^{M}$. The collection of such
pseudometrics $\rho $, obtaining by letting $d_{1},\ldots ,d_{n}$ range
among all the pseudometric symbols of sort $\mathcal{S}_{1},\ldots ,\mathcal{%
S}_{n}$, define a uniform structure on the space of $\mathcal{L}$-definable
sets in $D_{1}\times \cdots \times D_{n}$ relative to the class $\mathcal{C}$
. It is clear from the definition that the space of (positively
existentially) $\mathcal{L}$-definable sets is \emph{complete }with respect
to such a uniform structure.

Suppose now that $M$ is an $\mathcal{L}$-structure, and $\mathcal{C}$ is a
class of $\mathcal{L}$-structures as above. Then one can consider the class $%
\mathcal{C}( M) $ of $\mathcal{L}$-structures in $\mathcal{C}$ that contain
a distinguished copy of $M$. Then one can naturally regard the structures in 
$\mathcal{C}( M) $ as $\mathcal{L}( M) $-structures. We consider a natural
notion of $\mathcal{L}$-\emph{definable substructure}. Suppose as above that 
$\mathcal{C}$ is a class of $\mathcal{L} $-structures.

\begin{definition}
\label{Definition:definable-substructure}A positively existentially $%
\mathcal{L}$-definable substructure relative to the class $\mathcal{C}$ is
an assignment $( M,D) \mapsto S^{D,M}$ where $M$ ranges among the $\mathcal{L%
}$-structures in $\mathcal{C}$ and $D$ ranges among the domains of
quantification in $\mathcal{L}$ such that:

\begin{itemize}
\item $S^{D,M}$ is a closed subset of $D^{M}$,

\item setting $S( M) ^{D}:=S^{D,M}$ for every sort $\mathcal{S}$ and every
domain of quantification $D$ of sort $\mathcal{S}$ defines an $\mathcal{L}$%
-substructure $S( M) $ of $M$,

\item for every domain $D$, the assignment $M\mapsto S^{D,M}$ is a
positively existentially $\mathcal{L}$-definable set in $D$ relative to the
class $\mathcal{C}$.
\end{itemize}
\end{definition}

The following proposition is an easy consequence of the definition of
positive existential $\mathcal{L}$-embedding.

\begin{proposition}
\label{Proposition:definable-substructure}Suppose that $M,N$ are $\mathcal{L}
$-structures, and $T\colon M\rightarrow N$ is a positively \emph{$\mathcal{L}
$}-existential $\mathcal{L}$-embedding. Suppose that $S$ is a positively
existentially $\mathcal{L}( M) $-definable substructure relative to a class $%
\mathcal{C}$ of $\mathcal{L}( M) $-structures containing $M$ and $N$. Then $%
T $ maps $S( M) $ to $S( N) $, and the restriction $T|_{S( M) }\colon S( M)
\rightarrow S( N) $ is a positively \emph{$\mathcal{L}$}-existential $%
\mathcal{L}$-embedding.
\end{proposition}

%\bibliographystyle{amsplain}
%\bibliography{biblio-quantum}

\begin{thebibliography}{10}

\bibitem{barlak_sequentially_2016}
Sel{\c{c}}uk Barlak and G{\'{a}}bor Szab{\'{o}}, \emph{Sequentially split
  *-homomorphisms between {C}*-algebras}, International Journal of Mathematics
  \textbf{27} (2016), no.~13.

\bibitem{barlak_spatial_2017}
Sel{\c{c}}uk Barlak, G{\'{a}}bor Szab{\'{o}}, and Christian Voigt, \emph{The
  spatial {Rokhlin} property for actions of compact quantum groups}, Journal of
  Functional Analysis \textbf{272} (2017), no.~6, 2308--2360.

\bibitem{ben_yaacov_model_2008}
Ita{\"{i}} Ben~Yaacov, Alexander Berenstein, C.~Ward Henson, and Alexander
  Usvyatsov, \emph{Model theory for metric structures}, Model theory with
  applications to algebra and analysis. {Vol}. 2, London {Mathematical}
  {Society} {Lecture} {Note} {Series}, vol. 350, Cambridge University Press,
  2008, pp.~315--427.

\bibitem{carlson_omitting_2014}
Kevin Carlson, Enoch Cheung, Ilijas Farah, Alexander Gerhardt-Bourke, Bradd
  Hart, Leanne Mezuman, Nigel Sequeira, and Alexander Sherman, \emph{Omitting
  types and {AF} algebras}, Archive for Mathematical Logic \textbf{53} (2014),
  no.~1-2, 157--169.

\bibitem{chang_model_1977}
Chen~C. Chang and H.~Jerome Keisler, \emph{Model theory}, second ed.,
  North-Holland Publishing Co., Amsterdam-New York-Oxford, 1977, Studies in
  Logic and the Foundations of Mathematics, 73. \MR{0532927}

\bibitem{de_commer_actions_2016}
Kenny De~Commer, \emph{Actions of compact quantum groups}, arXiv:1604.00159
  (2016).

\bibitem{eagle_fraisse_2016}
Christopher~J. Eagle, Ilijas Farah, Bradd Hart, Boris Kadets, Vladyslav
  Kalashnyk, and Martino Lupini, \emph{Fra{\"{i}}ss{\'{e}} limits of
  {C}*-algebras}, Journal of Symbolic Logic \textbf{81} (2016), no.~2,
  755--773.

\bibitem{eagle_pseudoarc_2016}
Christopher~J. Eagle, Isaac Goldbring, and Alessandro Vignati, \emph{The
  pseudoarc is a co-existentially closed continuum}, Topology and its
  Applications \textbf{207} (2016), 1--9.

\bibitem{eagle_saturation_2015}
Christopher~J. Eagle and Alessandro Vignati, \emph{Saturation and elementary
  equivalence of {C}*-algebras}, Journal of Functional Analysis \textbf{269}
  (2015), no.~8, 2631--2664.

\bibitem{ellwood_new_2000}
David~Alexandre Ellwood, \emph{A {new} {characterisation} of {principal}
  {actions}}, Journal of Functional Analysis \textbf{173} (2000), no.~1,
  49--60.

\bibitem{farah_countable_2013}
Ilijas Farah and Bradd Hart, \emph{Countable saturation of corona algebras},
  Comptes Rendus Math{\'{e}}matiques de l'Acad{\'{e}}mie des Sciences
  \textbf{35} (2013), no.~2, 35--56.

\bibitem{farah_model_2016}
Ilijas Farah, Bradd Hart, Martino Lupini, Leonel Robert, Aaron Tikuisis,
  Alessandro Vignati, and Winter Winter, \emph{Model {theory} of
  {C}*-algebras}, arXiv:1602.08072 (2016).

\bibitem{farah_relative_2017}
Ilijas Farah, Bradd Hart, Mikael R{\o}rdam, and Aaron Tikuisis, \emph{Relative
  commutants of strongly self-absorbing {C}*-algebras}, Selecta Mathematica
  \textbf{23} (2017), no.~1, 363--387.

\bibitem{farah_model_2013}
Ilijas Farah, Bradd Hart, and David Sherman, \emph{Model theory of operator
  algebras {I}: {stability}}, Bulletin of the London Mathematical Society
  \textbf{45} (2013), no.~4, 825--838.

\bibitem{farah_model_2014}
\bysame, \emph{Model theory of operator algebras {II}: model theory}, Israel
  Journal of Mathematics \textbf{201} (2014), no.~1, 477--505.

\bibitem{farah_model_2014-1}
\bysame, \emph{Model theory of operator algebras {III}: elementary equivalence
  and {II}{$_1$} factors}, Bulletin of the London Mathematical Society
  \textbf{46} (2014), no.~3, 609--628.

\bibitem{gardella_crossed_2014}
Eusebio Gardella, \emph{Crossed products by compact group actions with the
  {Rokhlin} property}, Journal of Noncommutative Geometry, in press.

\bibitem{gardella_regularity_?}
\bysame, \emph{Regularity properties and {Rokhlin} dimension for compact group
  actions}, Houston Journal of Mathematics, in press.

\bibitem{gardella_rokhlin_2014}
\bysame, \emph{Rokhlin dimension for compact group actions}, Indiana Journal of
  Mathematics, in press.

\bibitem{gardella_classification_2014}
\bysame, \emph{Classification theorems for circle actions on {Kirchberg}
  algebras, {II}}, arXiv:1406.1208 (2014).

\bibitem{gardella_rokhlin_2016}
Eusebio Gardella, Ilan Hirshberg, and Luis Santiago, \emph{Rokhlin dimension:
  tracial properties and crossed products}, preprint.

\bibitem{gardella_equivariant_2016}
Eusebio Gardella and Martino Lupini, \emph{Equivariant logic and applications
  to {C}*-dynamics}, arXiv:1608.05532 (2016).

\bibitem{ghasemi_reduced_2016}
Saeed Ghasemi, \emph{Reduced products of metric structures: a metric
  {Feferman}-{Vaught} theorem}, The Journal of Symbolic Logic \textbf{81}
  (2016), no.~3, 856--875.

\bibitem{goldbring_kirchbergs_2015}
Isaac Goldbring and Thomas Sinclair, \emph{On {Kirchberg}'s embedding problem},
  Journal of Functional Analysis \textbf{269} (2015), no.~1, 155--198.

\bibitem{goldbring_axiomatizability_2016}
\bysame, \emph{On the axiomatizability of {C}*-algebras as operator systems},
  arXiv:1603.05444 (2016).

\bibitem{heinrich_ultraproducts_1980}
Stefan Heinrich, \emph{Ultraproducts in {Banach} space theory}, Journal
  f{\"{u}}r die Reine und Angewandte Mathematik \textbf{313} (1980), 72--104.
  \MR{552464}

\bibitem{hirshberg_rokhlin_2016}
Ilan Hirshberg, G{\'{a}}bor Szab{\'{o}}, Wilhelm Winter, and Jianchao Wu,
  \emph{Rokhlin dimension for flows}, arXiv:1607.02222 (2016).

\bibitem{hirshberg_rokhlin_2015}
Ilan Hirshberg, Wilhelm Winter, and Joachim Zacharias, \emph{Rokhlin dimension
  and {C}*-dynamics}, Communications in Mathematical Physics \textbf{335}
  (2015), no.~2, 637--670.

\bibitem{izumi_finite_2004}
Masaki Izumi, \emph{Finite group actions on {C}*-algebras with the {Rohlin}
  property, {I}}, Duke Mathematical Journal \textbf{122} (2004), no.~2,
  233--280.

\bibitem{kirchberg_central_2006}
Eberhard Kirchberg, \emph{Central sequences in {C}*-algebras and strongly
  purely infinite algebras}, Operator {Algebras}: {The} {Abel} {Symposium}
  2004, Abel {Symp}., vol.~1, Springer, Berlin, 2006, pp.~175--231.

\bibitem{kirchberg_covering_2004}
Eberhard Kirchberg and Wilhelm Winter, \emph{Covering dimension and
  quasidiagonality}, International Journal of Mathematics \textbf{15} (2004),
  no.~01, 63--85.

\bibitem{kishimoto_rohlin_1995}
Akitaka Kishimoto, \emph{The {Rohlin} property for automorphisms of {UHF}
  algebras}, Journal f{\"{u}}r die reine und angewandte Mathematik
  \textbf{1995} (1995), no.~465, 183--196.

\bibitem{kodaka_rohlin_2015}
Kazunori Kodaka and Tamotsu Teruya, \emph{The {Rohlin} property for coactions
  of finite dimensional {C}*-{Hopf} algebras on unital {C}*-algebras}, Journal
  of Operator Theory \textbf{74} (2015), no.~2, 329--369. \MR{3431936}

\bibitem{kustermans_survey_1999}
Johan Kustermans and Lars Tuset, \emph{A survey of {C}*-algebraic quantum
  groups. {I}}, Irish Mathematical Society Bulletin (1999), no.~43, 8--63.

\bibitem{kustermans_survey_2000}
\bysame, \emph{A survey of {C}*-algebraic quantum groups. {II}}, Irish
  Mathematical Society Bulletin (2000), no.~44, 6--54.

\bibitem{kustermans_locally_2000}
Johan Kustermans and Stefaan Vaes, \emph{Locally compact quantum groups},
  Annales Scientifiques de l’École Normale Supérieure \textbf{33} (2000),
  no.~6, 837--934.

\bibitem{kustermans_operator_2000}
\bysame, \emph{The operator algebra approach to quantum groups}, Proceedings of
  the National Academy of Sciences of the United States of America \textbf{97}
  (2000), no.~2, 547--552 (electronic).

\bibitem{kustermans_locally_2003}
\bysame, \emph{Locally compact quantum groups in the von {Neumann} algebraic
  setting}, Mathematica Scandinavica \textbf{92} (2003), no.~1, 68--92.

\bibitem{maes_notes_1998}
Ann Maes and Alfons Van~Daele, \emph{Notes on compact quantum groups}, Nieuw
  Archief voor Wiskunde. Vierde Serie \textbf{16} (1998), no.~1-2, 73--112.
  \MR{1645264}

\bibitem{masumoto_fraisse_2016}
Shuhei Masumoto, \emph{A {Fra\"{i}ss\'{e}} theoretic approach to the
  {Jiang}--{Su} algebra}, arXiv:1612.00646 (2016).

\bibitem{masumoto_jiang-su_2016}
\bysame, \emph{Jiang-{Su} {Algebra} as a {Fra\"{i}ss\'{e}} {limit}},
  arXiv:1602.00124 (2016).

\bibitem{nawata_finite_2016}
Norio Nawata, \emph{Finite group actions on certain stably projectionless
  {C}*-algebras with the {Rohlin} property}, Transactions of the American
  Mathematical Society \textbf{368} (2016), no.~1, 471--493. \MR{3413870}

\bibitem{nest_equivariant_2010}
Ryszard Nest and Christian Voigt, \emph{Equivariant {Poincar{\'{e}}} duality
  for quantum group actions}, Journal of Functional Analysis \textbf{258}
  (2010), no.~5, 1466--1503.

\bibitem{pachl_uniform_2013}
Jan Pachl, \emph{Uniform spaces and measures}, Fields {Institute} {Monographs},
  vol.~30, Springer, New York, 2013.

\bibitem{podles_symmetries_1995}
Piotr Podle{\'{s}}, \emph{Symmetries of quantum spaces. {Subgroups} and
  quotient spaces of quantum {$\mathrm{SU}_2$} and {$\mathrm{SO}_3$} groups},
  Communications in Mathematical Physics \textbf{170} (1995), no.~1, 1--20.
  \MR{MR1331688}

\bibitem{rainone_crossed_2016}
Timothy Rainone and Christopher Schafhauser, \emph{Crossed products of nuclear
  {C}*-algebras by free groups and their traces}, {arXiv}:1601.06090 (2017).

\bibitem{timmermann_invitation_2008}
Thomas Timmermann, \emph{An invitation to quantum groups and duality}, {EMS}
  {Textbooks} in {Mathematics}, European Mathematical Society, 2008.

\bibitem{vaes_locally_2001}
Stefaan Vaes, \emph{Locally compact quantum groups}, Ph.D. thesis, Katholieke
  Universiteit Leuven, 2001.

\bibitem{van_daele_discrete_1996}
Alfons Van~Daele, \emph{Discrete {quantum} {groups}}, Journal of Algebra
  \textbf{180} (1996), no.~2, 431--444.

\bibitem{winter_completely_2009}
Wilhelm Winter and Joachim Zacharias, \emph{Completely positive maps of order
  zero}, M{\"{u}}nster Journal of Mathematics \textbf{2} (2009), 311--324.

\bibitem{winter_nuclear_2010}
\bysame, \emph{The nuclear dimension of {C}*-algebras}, Advances in Mathematics
  \textbf{224} (2010), no.~2, 461--498.

\end{thebibliography}

\providecommand{\MR}[1]{}
\providecommand{\bysame}{\leavevmode\hbox to3em{\hrulefill}\thinspace}
\providecommand{\MR}{\relax\ifhmode\unskip\space\fi MR }
\providecommand{\MRhref}[2]{%
  \href{http://www.ams.org/mathscinet-getitem?mr=#1}{#2}
}
\providecommand{\href}[2]{#2}

\end{document}